\documentclass{amsart}
\usepackage[left=2.5cm,    
right=2.5cm,   
top=2cm,       
bottom=2cm]{geometry} 
\usepackage{amssymb}
\usepackage{amsmath}
\usepackage{mathtools}
\usepackage{booktabs}
\usepackage{float}
\usepackage{verbatim}
\usepackage{tikz-cd}
\usepackage{bm} 
\usepackage[colorlinks,linkcolor=cyan,citecolor=cyan]{hyperref}
\usepackage[hang,flushmargin]{footmisc} 
\usepackage[square,comma,sort&compress,numbers]{natbib} 
\usepackage{mathrsfs} 
\usepackage[font=footnotesize,skip=0pt,textfont=rm,labelfont=rm]{caption,subcaption} 

\usepackage{amsthm}
\newtheorem{theorem}{Theorem}[section]
\newtheorem{proposition}{Proposition}[section]
\newtheorem{lemma}{Lemma}[section]
\newtheorem{definition}{Definition}[section]

\newtheorem{remark}{Remark}[section]
\newcommand{\ssp}{\mathfrak{sp}(2)_{\mathbb{C}}}
\newcommand{\spp}{\mathfrak{sp}_4(\mathbb{C})}

\begin{document}
	
	\title[Morse Index of Yang--Mills Connections over $S^7$]{Morse Index of a $G_2$-Instanton over $S^7$}
	\author[X. Han]{Xiaoli Han}
	\address{Department of Mathematical Sciences, Tsinghua University, Beijing 100084, China}
	\email{hanxiaoli@mail.tsinghua.edu.cn}
	
	\author[Y. Wen]{Yang Wen$^{*}$}
	\address{School of Mathematics and Statistics, Nanjing University of Science and Technology, Nanjing 210094, China}
	\email{wenyang@njust.edu.cn}
	\thanks{$^{*}$Corresponding author.}
	\date{}
	\begin{abstract}
		In this paper, we determine the nonpositive spectrum of the Yang-Mills Jacobi operator in Coulomb gauge at the standard $G_2$-instanton on $S^7$. We prove that the only negative eigenvalue is $-3$, with multiplicity $8$, and the kernel has dimension $20$. Thus the connection attains the sharp lower bound for the Morse index of nonflat Yang-Mills connections on $S^7$. We also give an explicit description of the kernal of the Jacobi operator, which contains Waldron's $15$-dimensional deformation space for the fixed nearly parallel $G_2$-structure. The 5 additional zero modes are isometry-induced variations that do not satisfy the corresponding linearized instanton equation.
	\end{abstract}
	\date{}
	\maketitle
	\section{Introduction}
	Let $(M,g)$ be a compact Riemannian manifold and $P$ be a principal bundle over $M$  with compact structure group $G\subset\operatorname{SO}(r)$. Let $E=P\times_G\mathbb{R}^r$ be the associated vector bundle and $\mathfrak{g}_E=P\times_{\operatorname{Ad}}\mathfrak{g}$ be the adjoint bundle, where $\mathfrak{g}$ is the Lie algebra of $G$. A connection on $P$ is naturally identified with a connection on $E$ preserving its $G$-structure. The Yang-Mills functional is given by
	\begin{align}\label{YM functional}
		\operatorname{YM}(\nabla)=\frac12\int_M|F_\nabla|^2dV
	\end{align}
	for any connection $\nabla$, where $F_\nabla$ is the curvature tensor of $\nabla$. We denote by $d^\nabla$ the exterior covariant derivative on $\mathfrak{g}_E$-valued forms induced by $\nabla$, and by $\delta^\nabla$ its formal $L^2$-adjoint. The Euler-Lagrange equation of $\operatorname{YM}$ is
	\begin{align}\label{YM equation}
		\delta^\nabla F_\nabla=0.
	\end{align}
	A connection satisfying \eqref{YM equation} is called a Yang-Mills connection. At a Yang-Mills connection $\nabla$, the second variation is given by
	\begin{align}
		\frac{d^2}{dt^2}\mid_{t=0}\operatorname{YM}(\nabla+tB)=\int_M\langle\delta^\nabla d^\nabla B+\mathfrak{R}^\nabla(B),B\rangle dV
	\end{align}
	for any $B\in\Omega^1(\mathfrak{g}_E)$, where $\mathfrak R^\nabla$ is the zero-order curvature term defined in section \ref{subsection:connections}. Since the Yang-Mills functional is invariant under gauge transformation, we consider the variation directions $B\in\operatorname{ker}(\delta^\nabla)$, which is the subspace of $\Omega^1(\mathfrak{g}_E)$ $L^2$-orthogonal to the infinitesimal gauge variations. We have
	\begin{align}\label{2nd var}
		\frac{d^2}{dt^2}|_{t=0}\operatorname{YM}(\nabla+tB)=\int_{M}\langle\mathscr{S}^\nabla(B),B\rangle dV
	\end{align}
	for any $B\in\operatorname{ker}(\delta^\nabla)$, where
	\begin{align*}
		\mathscr{S}^\nabla(B)=\Delta^\nabla B+\mathfrak{R}^\nabla(B)
	\end{align*}
	is a self-adjoint elliptic operator on $\operatorname{ker}(\delta^\nabla)$. Therefore, the eigenvalues of $\mathscr{S}^\nabla$ is given by
	\begin{align*}
		\lambda_1<\lambda_2<\dots\to+\infty.
	\end{align*}
	Following Bourguignon and Lawson \cite{JL}, we define the Morse index $i(\nabla)$ as the sum of the dimensions of its negative eigenspaces, and the nullity $n(\nabla)$ as the dimension of its kernel.
	
	Donaldson and Thomas \cite{SR} studied $G_2$-instantons on $G_2$ manifolds, which provide a basic for a gauge theory analogous to Casson--Floer theory in 3 dimensions. One of its central aims is to define an invariant of torsion-free $G_2$-structures by counting $G_2$-instantons on a fixed bundle. Waldron \cite{AW} studied the standard $G_2$-instanton over $S^7$ ,, obtained by pulling back the BPST instanton under the quaternionic Hopf fiberation $\phi:S^7\to S^4$, and proved that its space of infinitesimal deformations as a $G_2$-instanton for the fixed standard nearly parallel $G_2$-structure, modulo gauge transformations, is $15$-dimensional. The standard nearly parallel $G_2$-structure on $S^7$, which is naturally a Yang-Mills connection, provides a useful model for higher-dimensional gauge theory. It also provide models for isolated singularities in 8-dimensional manifolds \cite{AW}. However, unlike in the torsion-free case, the $G_2$-instanton equation on a nearly parallel $G_2$-manifold implies the Yang-Mills equation but does not guarantee Yang-Mills stability. This leads to the following previously open problem (see \cite[36:30-37:05]{AWtalk}).
	
	\textbf{Question:} \emph{What is the Morse index of the standard $G_2$-instanton over $S^7$?}
	
	Simons showed that every nonflat Yang--Mills connection on $S^n$, $n\ge5$, is unstable (see \cite[Theorem 4.3]{JHJ} and \cite[Theorem A]{JL}). Nayatani and Urakawa \cite{SH} established the sharp lower bound $i(\nabla)\ge n+1$. Ni \cite{Ni} proved that the $-(n-4)$-eigenspace has dimension at least $n+1$ by constructing an $(n+1)$-dimensional subspace of this eigenspace. Laquer \cite{HTL} calculated the Morse indices of canonical connections and showed that the canonical connection on $\operatorname{SO}(n+1)\to S^n\cong\operatorname{SO}(n+1)/\operatorname{SO}(n)$ has Morse index $i(\nabla)=n+1$ and nullity $n(\nabla)=0$ for $n\ge5$.
	
	These results raise a natural rigidity question: if the Morse index of a Yang-Mills connection over $S^n$ attains the lower bound $n+1$, must its underlying principal bundle be isomorphic to the canonical bundle $\operatorname{SO}(n+1)\to S^n$, via a bundle isomorphism covering an isometry of the sphere? In this paper, we give a negative answer in $S^7$. We prove that the standard $G_2$-instanton on $S^7$ also attains the lower bound, with Morse index $i(\nabla)=8$. Our main results are stated below.
	
	\begin{theorem}\label{thm1}
		Let $\nabla'$ be the standard anti-self-dual BPST instanton on $P_4\to S^4$, and let $\nabla=\phi^*\nabla'$ be its pullback under the quaternionic Hopf fibration. Then the only negative eigenvalue of $\mathscr S^\nabla|_{\ker\delta^\nabla}$ is $-3$, with multiplicity $8$. In particular,
		\begin{align*}
			i(\nabla)=8.
		\end{align*}
		Therefore, $\nabla$ attains the sharp lower bound for the Morse index of nonflat Yang--Mills connections on $S^7$.
	\end{theorem}
	\begin{remark}
		The negative eigenspace in Theorem \ref{thm1} has a natural geometric description. For any vector $v\in\mathbb{R}^8$, let $F_v(x)=v\cdot x$ be the linear map on $\mathbb{R}^8$ and $f_v=F_v\mid_{S^7}$ be the restriction on $S^7$. Define the conformal gradient vector field
		\begin{align*}
			V_v(x)=\operatorname{grad}_{S^7}f_v(x).
		\end{align*}
		By \cite[Lemma 3.1 and Propositions 3.3--3.4]{Ni}, we have $i_{V_v}F\in\operatorname{ker}\delta^\nabla$ and $\mathscr S^\nabla(i_{V_v}F)=-3i_{V_v}F$. It is easily to see that $\{i_{V_v}F_\nabla\mid v\in\mathbb{R}^8\}$ is an 8-dimensional space. Applying Theorem \ref{thm1}, we obtain
		\begin{align*}
			\ker(\mathscr S^\nabla\mid_{\ker\delta^\nabla}+3I)=\{i_{V_v}F\mid v\in\mathbb{R}^8\}.
		\end{align*}
	\end{remark}
	Our second result determines the nullity and gives an explicit geometric description of all Jacobi fields.
	
	\begin{theorem}\label{thm2}
		Let $\nabla$ be the $G_2$ instanton defined in Theorem \ref{thm1}. Then nullity of $\nabla$ is
		\begin{align*}
			n(\nabla)=20.
		\end{align*}
		Moreover, we have
		\begin{align*}
			\ker\mathscr S^\nabla\mid_{\operatorname{ker}\delta^\nabla}=\{i_VF\mid V\in\mathcal V\},
		\end{align*}
		where
		\begin{align*}
			\mathcal V=\{&x\mapsto(I-xx^\dagger)(M_0x+M_1xi+M_2xj+M_3xk)\mid\\
			&M_r\in M_2(\mathbb H),\ M_r^\dagger=M_r,\ \operatorname{tr}(M_r)=0,\ 0\le r\le3\}
		\end{align*}
		is a 20-dimensional vector space. Here $x\in S^7\subset\mathbb H^2$, the matrices $M_r$ are constant, and $\dagger$ denotes quaternionic conjugate transpose.
	\end{theorem}
	\begin{remark}
		By \cite[Proposition 4.4 and Theorem 5.11]{AW}, the $15$-dimensional infinitesimal instanton deformation space is precisely the subspace in Theorem \ref{thm2} obtained by setting $M_3=0$. Thus the terms $M_0x$, $M_1xi$, and $M_2xj$ account for Waldron's deformations, while the remaining term $M_3xk$ gives a complementary 5-dimensional subspace.
	\end{remark}
	The formula in Theorem \ref{thm2} separates the kernel into 5 directions arising from anti-self-dual deformations on $S^4$ and 15 directions induced by isometries of $S^7$. Indeed, the $M_0$ term is a horizontal gradient field, while the remaining terms are the horizontal projections of the Killing vector fields $V(x)=M_1xi+M_2xj+M_3xk$. 
	\section{Preliminaries}\label{section preliminaries}
	\subsection{The adjoint BPST bundle and its Hopf pullback}
	
	In this section, we introduce the principal $\operatorname{SO}(3)$-bundle $P_4$ underlying the adjoint BPST instanton \cite{AAAY} and its pullback $P=\phi^*P_4$ under the quaternionic Hopf fibration. We then identify the adjoint bundle of $P$ with a homogeneous vector bundle. To this end, we first describe the quaternionic Hopf fibration in terms of homogeneous spaces.
	
	Let $\mathbb{H}$ denote the algebra of quaternions. For $A\in\mathbb{H}^{n\times n}$, let $A^\dagger$ denote its quaternionic conjugate transpose. The compact symplectic group and its Lie algebra are defined by
	\begin{align*}
		\operatorname{Sp}(n)&=\{A\in\mathbb{H}^{n\times n}\mid A^\dagger A=I_n\},\\
		\mathfrak{sp}(n)&=\{A\in\mathbb{H}^{n\times n}\mid A+A^\dagger=0\}.
	\end{align*}
	We identify $S^7$ with the unit sphere in $\mathbb H^2$ and $S^4$ with $\mathbb HP^1$, where $[q_1:q_2]$ denotes the equivalence class of $(q_1,q_2)\ne(0,0)$ under right multiplication by nonzero quaternions. The quaternionic Hopf fibration is $\phi(q_1,q_2)=[q_1:q_2]$. Throughout, let $G=\operatorname{Sp}(2)$ and set
	\begin{align*}
		K=\{\operatorname{diag}(1,q)\mid q\in \operatorname{Sp}(1)\},\qquad H=\{\operatorname{diag}(q_1,q_2)\mid q_1,q_2\in \operatorname{Sp}(1)\},
	\end{align*}
	with Lie algebras $\mathfrak k$ and $\mathfrak h$, respectively. The natural $G$-actions on $S^7$ and $\mathbb HP^1$ are transitive, with stabilizers $K$ and $H$ at $x_0=(1,0)^T$ and $[x_0]$, respectively. Thus
	\begin{align*}
		G/K\longrightarrow S^7,\quad gK\longmapsto gx_0\quad\textrm{and}\quad G/H\longrightarrow S^4,\quad gH\longmapsto[gx_0]
	\end{align*}
	are diffeomorphisms, under which the Hopf fibration becomes $\phi(gK)=gH$.
	
	Identify $\mathbb{R}^3$ with the imaginary quaternions
	\begin{align*}
		\operatorname{Im}\mathbb{H}=\{ai+bj+ck\mid a,b,c\in\mathbb{R}\}.
	\end{align*}
	Quaternionic conjugation defines a group homomorphism
	\begin{align*}
		\lambda:\operatorname{Sp}(1)\to \operatorname{SO}(3),\qquad \lambda(q)x=qxq^{-1}.
	\end{align*}
	Its differential at the identity,
	\begin{align*}
		\lambda_*:\mathfrak{sp}(1)\to\mathfrak{so}(3),
	\end{align*}
	is a Lie algebra isomorphism.
	
	\begin{definition}\ \\
		\begin{enumerate}
			\item Define
			\begin{align*}
				P_4=G\times_H \operatorname{SO}(3),
			\end{align*}
			where the equivalence relation is $[g\operatorname{diag}(q_1,q_2),h]=[g,\lambda(q_2)h]$. The projection and the right $\operatorname{SO}(3)$-action are given by
			\begin{align*}
				\pi_4:P_4\to G/H,\qquad \pi_4([g,h])=gH,\qquad[g,h_1]\cdot h_2=[g,h_1h_2],
			\end{align*}
			respectively. Thus $P_4$ is a principal $\operatorname{SO}(3)$-bundle over $G/H$. This is the adjoint principal bundle underlying the BPST instanton; see \cite[Section~4.2]{AW}.
			
			\item The pullback bundle $P=\phi^*P_4$ is naturally isomorphic to
			\begin{align}
				G\times_K \operatorname{SO}(3),
			\end{align}
			where $[g\operatorname{diag}(1,q),h]=[g,\lambda(q)h]$. The isomorphism is given by
			\begin{align*}
				G\times_K \operatorname{SO}(3)\to\phi^*P_4,\quad [g,h]\mapsto(gK,[g,h]).
			\end{align*}
			Under this identification, the projection and the right $\operatorname{SO}(3)$-action are given by
			\begin{align*}
				\pi:P\to G/K,\qquad \pi([g,h])=gK,\qquad[g,h_1]\cdot h_2=[g,h_1h_2],
			\end{align*}
			respectively. See \cite[Section~4.3]{AW} for the pullback construction.
			
			\item Define the adjoint bundle of $P$ by
			\begin{align*}
				E'=P\times_{\operatorname{Ad}}\mathfrak{so}(3),
			\end{align*}
			where $[p\cdot h,X]=[p,\operatorname{Ad}(h)X]$ for $p\in P$, $h\in \operatorname{SO}(3)$, and $X\in\mathfrak{so}(3)$.
		\end{enumerate}
	\end{definition}
	
	Both the unit round metric on $S^7$ and the pullback BPST connection are $G$-invariant. It is therefore convenient to identify $E'$ with a $G$-homogeneous vector bundle.
	
	\begin{lemma}
		Under the identification $\mathbb R^3\cong\operatorname{Im}\mathbb H\cong\mathfrak k$, the associated bundle $E=P\times_{\operatorname{SO}(3)}\mathbb R^3$ has the homogeneous realization
		\begin{align}
			E\cong G\times_K\mathfrak k,
		\end{align}
		with $[g\operatorname{diag}(1,q),Y]=[g,\operatorname{Ad}(q)Y]$. There is also a natural $G$-equivariant vector-bundle isomorphism $E'\cong E$ over $G/K$.	
	\end{lemma}
	\begin{proof}
		Since $\lambda$ is a group homomorphism, its differential satisfies
		\begin{align}
			\lambda_*\circ\operatorname{Ad}(q)=\operatorname{Ad}(\lambda(q))\circ\lambda_*
		\end{align}
		for every $q\in \operatorname{Sp}(1)$. Define
		\begin{align*}
			\Phi:E'\to E,\qquad\Phi([[g,h],X])=[g,\lambda_*^{-1}(\operatorname{Ad}(h)X)].
		\end{align*}
		For $h_1,h_2\in \operatorname{SO}(3)$, we have
		\begin{align*}
			\Phi([[g,h_1h_2],X])=[g,\lambda_*^{-1}(\operatorname{Ad}(h_1)\operatorname{Ad}(h_2)X)]=\Phi([[g,h_1],\operatorname{Ad}(h_2)X]).
		\end{align*}
		Moreover, for $q\in \operatorname{Sp}(1)$, the equivariance of $\lambda_*$ gives
		\begin{align*}
			\Phi([[g\operatorname{diag}(1,q),h],X])=[g,\operatorname{Ad}(q)\lambda_*^{-1}(\operatorname{Ad}(h)X)]=[g,\lambda_*^{-1}(\operatorname{Ad}(\lambda(q)h)X)]=\Phi([[g,\lambda(q)h],X]).
		\end{align*}
		Thus $\Phi$ is well defined.
		
		Define
		\begin{align*}
			\Psi:E\to E',\qquad\Psi([g,Y])=[[g,I_3],\lambda_*(Y)].
		\end{align*}
		The same equivariance identity shows that $\Psi$ is well defined. Moreover,
		\begin{align*}
			\Phi\circ\Psi=\operatorname{id}_E,\qquad \Psi\circ\Phi=\operatorname{id}_{E'}.
		\end{align*}
		Therefore, $\Phi$ is a $G$-equivariant vector-bundle isomorphism.
	\end{proof}
	The following lemma is easily checked.
	\begin{lemma}\label{lem: isomorphic g_E}
		The adjoint bundle $\mathfrak{g}_E$ is naturally isomorphic to
		\begin{align}
			\mathfrak{g}_E=P\times_{\operatorname{Ad}}\mathfrak{so}(3)=E'\cong G\times_K\mathfrak{k},
		\end{align}
		where $K$ acts on $\mathfrak{k}$ by the adjoint representation. Under this identification, for any $[g,A]\in\mathfrak{g}_E$ and $[g,X]\in E$, we have
		\begin{align*}
			[g,A]\cdot[g,X]=[g,\operatorname{ad}(A)X].
		\end{align*}
		Moreover, $\operatorname{ad}:\mathfrak{k}\to\mathfrak{so}(\mathfrak{k})$ is a Lie algebra isomorphism, and
		\begin{align*}
			\operatorname{ad}([A,B])=[\operatorname{ad}(A),\operatorname{ad}(B)]
		\end{align*}
		for any $A,B\in\mathfrak{k}$.
	\end{lemma}
	\subsection{Connections, curvature, and the Jacobi operator}\label{subsection:connections}
	
	In this section, we recall the formulas for connections and curvature that will be used below. Let $(M,g)$ be an $n$-dimensional compact Riemannian manifold, and let $D$ denote its Levi--Civita connection. Let $P\to M$ be a principal $G$-bundle, where $G\subset \operatorname{SO}(r)$ is compact, and let $E=P\times_G\mathbb{R}^r$ be the associated vector bundle of rank $r$. Denote the Lie algebra of $G$ by $\mathfrak{g}$ and the adjoint bundle of $P$ by $\mathfrak{g}_E=P\times_{\operatorname{Ad}}\mathfrak{g}$. The bundle metric on $E$ induces the pointwise inner product
	\begin{align*}
		\langle\phi,\psi\rangle=\operatorname{Tr}(\phi^T\psi)
	\end{align*}
	on $\Omega^0(\mathfrak{g}_E)$. Its $\operatorname{Ad}(G)$-invariance implies that
	\begin{align*}
		\langle[\phi,\psi],\rho\rangle=\langle\phi,[\psi,\rho]\rangle
	\end{align*}
	for any $\phi,\psi,\rho\in\Omega^0(\mathfrak{g}_E)$. The induced pointwise inner product on $\Omega^p(\mathfrak{g}_E)$ is given by
	\begin{align*}
		\langle\phi,\psi\rangle=\frac{1}{p!}\sum_{1\leq i_1,\ldots,i_p\leq n}\langle\phi(e_{i_1},\ldots,e_{i_p}),\psi(e_{i_1},\ldots,e_{i_p})\rangle,
	\end{align*}
	where $\{e_i\}_{i=1}^n$ is a local orthonormal frame of $TM$.
	
	Let $\nabla$ be a metric connection on $E$ induced by a principal connection on $P$, and use the same symbol for the induced connections on $\mathfrak{g}_E$ and $\Lambda^pT^*M\otimes\mathfrak{g}_E$. Thus, for $\phi\in\Omega^p(\mathfrak{g}_E)$,
	\begin{align*}
		(\nabla_X\phi)(X_1,\ldots,X_p)=\nabla_X(\phi(X_1,\ldots,X_p))-\sum_{k=1}^p\phi(X_1,\ldots,D_XX_k,\ldots,X_p).
	\end{align*}
	This connection induces the exterior covariant derivative
	\begin{align*}
		d^\nabla:\Omega^p(\mathfrak{g}_E)\to\Omega^{p+1}(\mathfrak{g}_E)
	\end{align*}
	and its formal adjoint
	\begin{align*}
		\delta^\nabla:\Omega^p(\mathfrak{g}_E)\to\Omega^{p-1}(\mathfrak{g}_E),
	\end{align*}
	which are given by
	\begin{align*}
		d^\nabla\phi(X_1,\ldots,X_{p+1})&=\sum_{i=1}^{p+1}(-1)^{i+1}\nabla_{X_i}\phi(X_1,\ldots,\widehat{X_i},\ldots,X_{p+1}),\\
		\delta^\nabla\phi(X_1,\ldots,X_{p-1})&=-\sum_{i=1}^n\nabla_{e_i}\phi(e_i,X_1,\ldots,X_{p-1}).
	\end{align*}
	
	In a local orthonormal trivialization of $E$, the connection has the form
	\begin{align*}
		\nabla=d+A,
	\end{align*}
	where $A\in\Omega^1(U,\mathfrak{g})$. We write $F_A=F_\nabla$ for its curvature. Then
	\begin{align*}
		F_A=dA+\frac12[A\wedge A],
	\end{align*}
	where
	\begin{align*}
		[\phi\wedge\psi](X,Y)=[\phi(X),\psi(Y)]-[\phi(Y),\psi(X)]
	\end{align*}
	for any $\phi,\psi\in\Omega^1(\mathfrak{g}_E)$ and $X,Y\in\Gamma(TM)$.
	
	We use the curvature convention
	\begin{align*}
		R_M(X,Y)Z=D_XD_YZ-D_YD_XZ-D_{[X,Y]}Z
	\end{align*}
	and define the Ricci transformation by
	\begin{align*}
		\operatorname{Ric}(X)=\sum_{j=1}^nR_M(X,e_j)e_j.
	\end{align*}
	The covariant Hodge Laplacian and the rough Laplacian are defined by
	\begin{align*}
		\Delta^\nabla&=d^\nabla\delta^\nabla+\delta^\nabla d^\nabla,\\
		\nabla^\ast\nabla&=-\sum_{j=1}^n\left(\nabla_{e_j}\nabla_{e_j}-\nabla_{D_{e_j}e_j}\right).
	\end{align*}
	For $\psi\in\Omega^1(\mathfrak{g}_E)$, define the zero-order curvature term as in \cite[(3.1)]{JL} by
	\begin{align*}
		\mathfrak{R}^\nabla(\psi)(X)=\sum_{j=1}^n[F_A(e_j,X),\psi(e_j)].
	\end{align*}
	The Bochner--Weitzenb\"ock formula \cite[Theorem~(3.2)]{JL} gives
	\begin{align}\label{Bochner}
		\Delta^\nabla\psi=\nabla^\ast\nabla\psi+\mathfrak{R}^\nabla(\psi)+\psi\circ \operatorname{Ric},
	\end{align}
	where
	\begin{align*}
		(\psi\circ \operatorname{Ric})(X)=\psi(\operatorname{Ric}(X)).
	\end{align*}
	Therefore, for $\psi\in\ker(\delta^\nabla)$, the Jacobi operator defined in the Introduction can be written as
	\begin{align*}
		\mathscr{S}^\nabla(\psi)=\nabla^\ast\nabla\psi+2\mathfrak{R}^\nabla(\psi)+\psi\circ \operatorname{Ric}.
	\end{align*}
	\subsection{The round metric on $G/K$} In this section, we describe the unit round metric under the identification $S^7\cong G/K$. This metric is $G$-invariant, but it differs from the normal homogeneous metric induced by the $\operatorname{Ad}(G)$-invariant inner product on $\mathfrak{g}$ introduced below.
	
	Let $\langle\ ,\ \rangle_{alg}$ be the $\operatorname{Ad}(G)$-invariant inner product on $\mathfrak{g}$ defined by
	\begin{align*}
		\langle X,Y\rangle_{alg}=\operatorname{Re}(\operatorname{Tr}(X^\dagger Y))
	\end{align*}
	for any $X,Y\in\mathfrak{g}$. This inner product satisfies
	\begin{align*}
		\langle[X,Y],Z\rangle_{alg}=\langle X,[Y,Z]\rangle_{alg}
	\end{align*}
	for any $X,Y,Z\in\mathfrak{g}$. Let $\pi_K:G\to G/K$ be the quotient map and consider the orthogonal decomposition with respect to $\langle\ ,\ \rangle_{alg}$
	\begin{align}
		\mathfrak{g}=\mathfrak{k}\oplus\mathfrak{m}.
	\end{align}
	Then $\pi_{K*}\mathfrak{m}=T_KG/K$ is the tangent space at $K\in G/K$. Let $\{X_i\mid1\le i\le7\}$ and $\{X_\alpha\mid8\le\alpha\le10\}$ be orthonormal bases of $\mathfrak{m}$ and $\mathfrak{k}$ respectively, with respect to $\langle\ ,\ \rangle_{alg}$, where
	\begin{equation}\label{X_i}
		\begin{aligned}
			X_1 &= \begin{pmatrix} i & 0 \\ 0 & 0 \end{pmatrix}, \quad &
			X_2 &= \begin{pmatrix} j & 0 \\ 0 & 0 \end{pmatrix}, \quad &
			X_3 &= \begin{pmatrix} k & 0 \\ 0 & 0 \end{pmatrix}, \\
			X_4 &= \frac1{\sqrt2}\begin{pmatrix} 0 & 1 \\ -1 & 0 \end{pmatrix}, \quad &
			X_5 &= \frac1{\sqrt2}\begin{pmatrix} 0 & i \\ i & 0 \end{pmatrix}, \quad &
			X_6 &= \frac1{\sqrt2}\begin{pmatrix} 0 & j \\ j & 0 \end{pmatrix}, \quad &
			X_7 &= \frac1{\sqrt2}\begin{pmatrix} 0 & k \\ k & 0 \end{pmatrix}, \\
			X_8 &= \begin{pmatrix} 0 & 0 \\ 0 & i \end{pmatrix}, \quad &
			X_9 &= \begin{pmatrix} 0 & 0 \\ 0 & j \end{pmatrix}, \quad &
			X_{10} &= \begin{pmatrix} 0 & 0 \\ 0 & k \end{pmatrix}.
		\end{aligned}
	\end{equation}
	Let $\langle\ ,\ \rangle$ denote the unit round metric on $G/K$. Then
	\begin{align*}
		\{\pi_{K*}X_i,\sqrt2\pi_{K*}X_j|1\le i\le3,4\le j\le7\}
	\end{align*}
	is an orthonormal basis of $T_{K}G/K$. Define the orthogonal decomposition
	\begin{align}
		\mathfrak{m}=\mathfrak{m}_1\oplus\mathfrak{m}_2,
	\end{align}
	where $\mathfrak{m}_1=\operatorname{span}_{\mathbb{R}}\{X_1,X_2,X_3\}$ and $\mathfrak{m}_2=\operatorname{span}_{\mathbb{R}}\{X_4,X_5,X_6,X_7\}$. We use the same notation $\langle\ ,\ \rangle$ for the corresponding inner product on $\mathfrak m$ such that
	\begin{align*}
		\langle X,Y\rangle=\langle\pi_{K*}X,\pi_{K*}Y\rangle
	\end{align*}
	for any $X,Y\in\mathfrak{m}$. Since $\mathfrak{m}_1$ and $\mathfrak{m}_2$ are $\operatorname{Ad}(K)$-invariant and $\langle\ ,\ \rangle_{alg}$ is $\operatorname{Ad}(G)$-invariant, the identity
	\begin{align*}
		\langle X,Y\rangle_{alg}=\langle X^{(1)},Y^{(1)}\rangle+2\langle X^{(2)},Y^{(2)}\rangle
	\end{align*}
	implies that $\langle\ ,\ \rangle$ is $\operatorname{Ad}(K)$-invariant, that is, for any $k\in K$ and $X,Y\in\mathfrak{m}$, we have
	\begin{align}
		\langle \operatorname{Ad}(k)X,\operatorname{Ad}(k)Y\rangle=\langle X,Y\rangle.
	\end{align}
	Using $[\mathfrak{m}_1,\mathfrak{m}_1]\subset\mathfrak{m}_1$, $[\mathfrak{m}_2,\mathfrak{m}_2]\subset\mathfrak{m}_1\oplus\mathfrak{k}$ and $[\mathfrak{m}_1,\mathfrak{m}_2]\subset\mathfrak{m}_2\oplus\mathfrak{k}$, we have
	\begin{equation}\label{eq:bracket-identity}
		\langle[X,Y]_{\mathfrak{m}},Z\rangle=\left\{
		\begin{array}{ll}
			\langle X,[Y,Z]_{\mathfrak{m}}\rangle,&X,Y,Z\in\mathfrak{m}_1\ or\ X,Z\in\mathfrak{m}_2,\ Y\in\mathfrak{m}_1,\\
			\frac12\langle X,[Y,Z]_{\mathfrak{m}}\rangle,&X\in\mathfrak{m}_1,\ Y,Z\in\mathfrak{m}_2,\\\
			2\langle X,[Y,Z]_{\mathfrak{m}}\rangle,&X,Y\in\mathfrak{m}_2,\ Z\in\mathfrak{m}_1,\\
			0,&\text{otherwise}.
		\end{array}
		\right.
	\end{equation}
	for any $X,Y,Z\in\mathfrak{m}_1\cup\mathfrak{m}_2$, where the subscript $\mathfrak m$ denotes the projection onto $\mathfrak m$.
	
	\section{The equivariant Jacobi operator}\label{section Jacobi}
	\subsection{A $G$-equivariant realization of $\Omega^1(\mathfrak{g}_E)$}
	
	We realize $\mathfrak{g}_E$-valued 1-forms as $K$-equivariant functions on $G$. In this realization, the Jacobi operator is expressed through invariant differential operators on $G$, so that the Peter--Weyl decomposition reduces its spectral analysis to finite-dimensional invariant subspaces. This approach follows the method of Urakawa \cite[Section 8.1]{HU}.
	
	For $g\in G$, let $L_g:G\to G$ and $\tau_g:G/K\to G/K$ denote the left translations on $G$ and $G/K$, respectively. Let $\{\xi_A\mid1\leq A\leq10\}$ be the dual basis of $\mathfrak g^*$ defined by
	\begin{align*}
		\xi_A(X_B)=\langle X_A,X_B\rangle_{alg}
	\end{align*}
	for any $1\le A,B\le10$. We begin by identifying $T^*(G/K)$ with an associated bundle.
	\begin{lemma}\label{lem: isomorphic T^*G/K}
		There is a natural $G$-equivariant vector-bundle isomorphism
		\begin{align*}
			T^*G/K\cong G\times_K\mathfrak{m}^*
		\end{align*}
		with $[gk,\xi]=[g,\operatorname{Ad}^*(k^{-1})\xi]$ for any $g\in G$, $k\in K$ and $\xi\in\mathfrak{m}^*$, where
		\begin{align*}
			\operatorname{Ad}^*(k)\xi(X):=\xi(\operatorname{Ad}(k)(X))
		\end{align*}
		for any $k\in K$, $X\in\mathfrak{m}$ and $\xi\in\mathfrak{m}^*$.
	\end{lemma}
	\begin{proof}
		The isomorphism is given by 
		\begin{align*}
			T^*G/K\to G\times_K\mathfrak{m}^*,\quad\xi\mapsto [g,L_g^*\pi_K^*\xi]
		\end{align*}
		for any $\xi\in T_{gK}^*G/K$. Since
		\begin{align*}
			[gk,L_{gk}^*\pi_K^*\xi]=[g,\operatorname{Ad}^*(k^{-1})L_{gk}^*\pi_K^*\xi]=[g,L_g^*\pi_K^*\xi]
		\end{align*}
		for any $k\in K$, the map is well-defined. The inverse map is given by
		\begin{align*}
			[g,\xi']\mapsto\tau_{g^{-1}}^*(\pi_K^*\mid_{\mathfrak{m}})^{-1}\xi'.
		\end{align*}
	\end{proof}
	The following lemma follows immediately from Lemmas \ref{lem: isomorphic g_E} and \ref{lem: isomorphic T^*G/K}.
	\begin{lemma}\label{lem:Theta}
		There is a bundle isomorphism
		\begin{align}
			\Theta:T^*G/K\otimes\mathfrak{g}_E\to G\times_K(\mathfrak{m}^*\otimes\mathfrak{k})
		\end{align}
		with $[gq,\xi\otimes X]=[g,\operatorname{Ad}^*(q^{-1})\xi\otimes \operatorname{Ad}(q)X]$ via
		\begin{align*}
			\xi\otimes[g,X]\mapsto[g,L_g^*\pi_K^*\xi]\otimes[g,X]\mapsto[g,L_g^*\pi_K^*\xi\otimes X]
		\end{align*}
		on each fiber $(T^*G/K\otimes\mathfrak{g}_E)_{gK}$.\\
		Moreover, $\Theta$ naturally induces an isomorphism
		\begin{align}
			\Theta:\Omega^1(\mathfrak{g}_E)\to\Gamma(G\times_K(\mathfrak{m}^*\otimes\mathfrak{k})).
		\end{align}
	\end{lemma}
	The following lemma identifies $\Omega^1(\mathfrak{g}_E)$ with the space of $K$-equivariant $\mathfrak{m}^*\otimes\mathfrak{k}$-valued maps.
	\begin{lemma}\label{Phi}
		Let
		\begin{align*}
			C_K^\infty(G,\mathfrak{m}^*\otimes\mathfrak{k})=\{f=\sum_{i=1}^7\xi_i\otimes f_i\in C^\infty(G,\mathfrak{m}^*\otimes\mathfrak{k})\mid f(gk)=\sum_{i=1}^7\operatorname{Ad}^*(k)\xi_i\otimes \operatorname{Ad}(k^{-1})f_i(g),\ \forall\ g\in G, k\in K\}.
		\end{align*}
		Then the following map is an isomorphism
		\begin{equation}
			\begin{split}
				\Phi:C_K^\infty(G,\mathfrak{m}^*\otimes\mathfrak{k})&\to\Gamma(G\times_K(\mathfrak{m}^*\otimes\mathfrak{k}))\\
				\Phi(f)(gK)&=[g,\sum_{i=1}^7\xi_i\otimes f_i(g)]
			\end{split}
		\end{equation}
	\end{lemma}
	\begin{proof}
		For any $k\in K$ and $f\in C_K^\infty(G,\mathfrak{m}^*\otimes\mathfrak{k})$, we have
		\begin{align*}
			[gk,\sum_{i=1}^7\xi_i\otimes f_i(gk)]=[g,\sum_{i=1}^7\operatorname{Ad}^*(k^{-1})\xi_i\otimes \operatorname{Ad}(k)f_i(gk)]=[g,\sum_{i=1}^7\xi_i\otimes f_i(g)].
		\end{align*}
		Hence $\Phi$ is well-defined. The inverse map is given by
		\begin{align*}
			\Phi^{-1}(s)(g)=\sum_{i=1}^7\xi_i\otimes s_i(g)
		\end{align*}
		for any $s\in\Gamma(G\times_K(\mathfrak{m}^*\otimes\mathfrak{k}))$ written as $s(gK)=[g,\sum_{i=1}^7\xi_i\otimes s_i(g)]$.
	\end{proof}
	Writing each $\mathfrak{k}$-valued coefficient in the basis $\{X_\alpha\}_{\alpha=8}^{10}$, we immediately obtain the following identification, and hence we omit the proof.
	\begin{lemma}
		Set $R_kf(g)=f(gk)$ for $k\in K$ and $f\in C^\infty(G)$, and define
		\begin{align*}
			&(C^\infty(G)\otimes\mathfrak{m}^*\otimes\mathfrak{k})_K\\
			=&\{\sum_{i=1}^7\sum_{\alpha=8}^{10}f_{i\alpha}\otimes\xi_i\otimes X_\alpha\in C^\infty(G)\otimes\mathfrak{m}^*\otimes\mathfrak{k}\mid \sum_{i=1}^7\sum_{\alpha=8}^{10}R_kf_{i\alpha}\otimes\xi_i\otimes X_\alpha\\
			&=\sum_{i=1}^7\sum_{\alpha=8}^{10}f_{i\alpha}\otimes \operatorname{Ad}^*(k)\xi_i\otimes \operatorname{Ad}(k^{-1})X_\alpha,\ \forall\ k\in K\}.
		\end{align*}
		Then there is an isomorphism
		\begin{equation}
			\begin{split}
				\Psi:C_K^\infty(G,\mathfrak{m}^*\otimes\mathfrak{k})&\to(C^\infty(G)\otimes\mathfrak{m}^*\otimes\mathfrak{k})_K\\
				\Psi(\sum_{i=1}^7\sum_{\alpha=8}^{10}\xi_i\otimes f_{i\alpha}X_\alpha)&=\sum_{i=1}^7\sum_{\alpha=8}^{10}f_{i\alpha}\otimes\xi_i\otimes X_\alpha.
			\end{split}
		\end{equation}
	\end{lemma}
	The compatibility of the preceding identifications with the natural left $G$-actions follows directly from their definitions. We record the precise statement below and omit the proof.
	\begin{lemma}
		Define $G$-actions on $\Gamma(G\times_K(\mathfrak{m}^*\otimes\mathfrak{k}))$, $\Omega^1(\mathfrak{g}_E)$, $C_K^\infty(G,\mathfrak{m}^*\otimes\mathfrak{k})$ and $(C^\infty(G)\otimes\mathfrak{m}^*\otimes\mathfrak{k})_K$ such that for any
		\begin{align*}
			&s(gK)=[g,\sum_{i=1}^7\xi_i\otimes s_i(g)]\in\Gamma(G\times_K(\mathfrak{m}^*\otimes\mathfrak{k})),\\
			&\tilde s(gK)=\sum_{i=1}^7L_{g^{-1}}^*\pi_K^*\xi_i\otimes[g,s_i(g)]\in\Omega^1(\mathfrak{g}_E),\\
			&f(g)=\sum_{i=1}^7\xi_i\otimes f_i(g)\in C_K^\infty(G,\mathfrak{m}^*\otimes\mathfrak{k}),\\
			&\sum_{i=1}^7\sum_{\alpha=8}^{10}f_{i\alpha}\otimes\xi_i\otimes X_\alpha\in(C^\infty(G)\otimes\mathfrak{m}^*\otimes\mathfrak{k})_K,
		\end{align*}
		we have
		\begin{align*}
			\tau_{x}s(gK)&=[g,\sum_{i=1}^7\xi_i\otimes s_i(x^{-1}g)],\\
			\tau_x\tilde s(gK)&=\sum_{i=1}^7L_{g^{-1}}^*\pi_K^*\xi_i\otimes [g,s_i(x^{-1}g)],\\
			\tau_x f(g)&=f(x^{-1}g),\\
			\tau_x(\sum_{i=1}^7\sum_{\alpha=8}^{10}f_{i\alpha}\otimes\xi_i\otimes X_\alpha)&=\sum_{i=1}^7\sum_{\alpha=8}^{10}\tau_xf_{i\alpha}\otimes\xi_i\otimes X_\alpha,
		\end{align*}
		where $\tau_x f(g)=f(x^{-1}g)$. Then $\Phi$, $\Theta$ and $\Psi$ are $G$-equivariant, that is, $\Phi\circ\tau_x=\tau_x\circ\Phi$, $\Theta\circ\tau_x=\tau_x\circ\Theta$, $\tau_x\circ\Psi=\Psi\circ\tau_x$ for any $x\in G$.
	\end{lemma}
	\subsection{The Levi-Civita connection}
	In this section, we compute the Levi-Civita connection at the base point $K\in G/K$ with respect to the local frame constructed below using left translations.
	
	\begin{definition}
		For any $X\in\mathscr{X}(G/K)$, we call $\tilde X$ the horizontal lift of $X$ if
		\begin{align*}
			&(1)\ \pi_{K*}\tilde X(g)=X(gK)\ \forall\ g\in G,\\
			&(2)\ \tilde X(g)\in L_{g*}\mathfrak{m}\ \forall\ g\in G.
		\end{align*}
	\end{definition}
	\begin{definition}
		A connection $D'$ on $G/K$ is called $G$-invariant if
		\begin{align*}
			\tau_{x*}(D'_XY(gK))=D'_{\tau_{x*}X}\tau_{x*}Y(xgK)
		\end{align*}
		for any $X,Y\in\mathscr{X}(G/K)$ and $x\in G$, where $\tau_{x*}X(gK)=\tau_{x*}(X(x^{-1}gK))$.
	\end{definition}
	\begin{lemma}\label{sigma}
		There exists an open neighborhood $U$ of $K$ in $G/K$ and a smooth map $\sigma:U\to G$ such that $\pi_K\circ\sigma=\operatorname{id}_U$ and $\sigma(\exp(X)K)=\exp(X)$ for any $\exp(X)\in \sigma(U)$.
	\end{lemma}
	\begin{proof}
		Consider $\psi:\mathfrak{m}\to G/K$ with
		\begin{align*}
			\psi(X)=\exp(X)K=\pi_K\circ\exp(X).
		\end{align*}
		Since $\psi_*=\pi_{K*}\circ \exp_*:T_0\mathfrak{m}\to T_KG/K$ is an isomorphism, by the inverse function theorem, there exists an open neighborhood $U'$ of $0\in\mathfrak{m}$ such that $\psi|_{U'}:U'\to\psi(U'):=U$ is a diffeomorphism. Then
		\begin{align*}
			\sigma=\exp\circ(\psi|_{U'})^{-1}:U\to G
		\end{align*}
		satisfies the required properties.
	\end{proof}
	To describe the Levi-Civita connection on $G/K$ algebraically, we apply Nomizu's result \cite[Theorem 8.1]{KN} to obtain the following proposition.
	\begin{proposition}\label{D=D^0+alpha}
		Let $D^0$ and $D'$ be $G$-invariant connections on $G/K$. Then there is a unique bilinear map $\alpha:\mathfrak{m}\times\mathfrak{m}\to\mathfrak{m}$ such that
		\begin{align*}
			D'_{\hat X}\hat Y(K)=D^0_{\hat X}\hat Y(K)+\pi_{K*}\alpha((\pi_{K*}|_{\mathfrak{m}})^{-1}X',(\pi_{K*}|_{\mathfrak{m}})^{-1}Y')
		\end{align*}
		for any $\hat X,\hat Y\in\mathscr{X}(G/K)$ with $\hat X(K)=X'$ and $\hat Y(K)=Y'$. Moreover, $\alpha$ satisfies
		\begin{align*}
			\operatorname{Ad}(k)(\alpha(X,Y))=\alpha(\operatorname{Ad}(k)X,\operatorname{Ad}(k)Y)
		\end{align*}
		for any $X,Y\in\mathfrak{m}$ and $k\in K$.\\
		Moreover, since $D'$ and $D^0$ are $G$-invariant, we have
		\begin{align*}
			D'_{\hat X}\hat Y(gK)=D^0_{\hat X}\hat Y(gK)+\tau_{\sigma(gK)*}\pi_{K*}\alpha((\pi_{K*}|_{\mathfrak{m}})^{-1}\tau_{\sigma(gK)^{-1}*}(\hat X(gK)),(\pi_{K*}|_{\mathfrak{m}})^{-1}\tau_{\sigma(gK)^{-1}*}(\hat Y(gK)))
		\end{align*}
		for any $gK\in U$, where $\sigma:U\to G$ is defined in Lemma \ref{sigma}.
	\end{proposition}
	There is a natural $G$-invariant connection defined by Nomizu \cite[Theorem 10.2, (10.3)]{KN}.
	\begin{definition}
		Set $C_K^\infty(G,\mathfrak{m}):=\{f\in C^\infty(G,\mathfrak{m})\mid f(gk)=\operatorname{Ad}(k^{-1})f(g)\ \forall\ k\in K\}$. For any $X\in\mathscr{X}(G/K)$, define $f_X\in C_K^\infty(G,\mathfrak{m})$ by $f_X(g)=(\pi_{K*}|_{\mathfrak{m}})^{-1}\circ\tau_{g^{-1}*}(X(gK))$. Let $\tilde X$ be the horizontal lift of $X$. Define a $G$-invariant connection $D^0$ on $G/K$ by
		\begin{align*}
			D^0_XY(gK)=\pi_{K*}\circ L_{g*}(\tilde X_g(f_Y)).
		\end{align*}
	\end{definition}
	\begin{lemma}\label{[X^*,Y^*]}
		For any $X,Y\in\mathfrak{m}$, denote $X'=\pi_{K*}X, Y'=\pi_{K*}Y$ and
		\begin{align*}
			X^*(gK)=\tau_{\sigma(gK)*}X',\quad Y^*(gK)=\tau_{\sigma(gK)*}Y'\in\mathscr{X}(U),
		\end{align*}
		where $\sigma:U\to G$ is defined in Lemma \ref{sigma}. It is easy to verify that
		\begin{align*}
			\tilde X(g)=L_{g*}\circ \operatorname{Ad}(g^{-1}\sigma(gK))X,\quad \tilde Y(g)=L_{g*}\circ \operatorname{Ad}(g^{-1}\sigma(gK))Y
		\end{align*}
		are the horizontal lifts of $X^*$ and $Y^*$, respectively. Then we have
		\begin{align*}
			D^0_{X^*}Y^*(K)=D^0_{Y^*}X^*(K)=0.
		\end{align*}
	\end{lemma}
	\begin{proof}
		By the definition of $D^0$, we have $f_{Y^*}(g)=\operatorname{Ad}(g^{-1}\sigma(gK))Y$. Thus
		\begin{align*}
			D^0_{X^*}Y^*(K)=\frac{d}{dt}|_{t=0}\pi_{K*}f_{Y^*}(\exp(tX))=\frac{d}{dt}|_{t=0}\pi_{K*}Y=0,
		\end{align*}
		where we use $\exp(-tX)\sigma(\exp(tX)K)=e$ for small $t$.
	\end{proof}
	\begin{proposition}\label{D_X*Y*}
		Let $D$ be the Levi-Civita connection on $G/K$. For any $X,Y\in\mathfrak{m}$, let $X^*,Y^*\in\mathscr{X}(U)$ be the local vector fields defined in Lemma \ref{[X^*,Y^*]}. Then we have
		\begin{equation}
			\begin{split}
				D_{X^*}Y^*(K)=\left\{
				\begin{array}{ll}
					\frac12[X,Y]_{\mathfrak{m}},&X,Y\in\mathfrak{m}_1,\\
					0,&X\in\mathfrak{m}_1,Y\in\mathfrak{m}_2,\\
					{[}X,Y]_{\mathfrak{m}_2},&X\in\mathfrak{m}_2,Y\in\mathfrak{m}_1,\\
					\frac12[X,Y]_{\mathfrak{m}},&X,Y\in\mathfrak{m}_2.				
				\end{array}
				\right.
			\end{split}
		\end{equation}
		for any $X,Y\in\mathfrak{m}_1\cup\mathfrak{m}_2$. Moreover, we have
		\begin{align}
			[X^*,Y^*](K)=D_{X^*}Y^*(K)-D_{Y^*}X^*(K)=[X,Y]_{\mathfrak{m}}.
		\end{align}
	\end{proposition}
	\begin{proof}
		Applying Proposition \ref{D=D^0+alpha}, we have
		\begin{align*}
			D_{X^*}Y^*(K)=D^0_{X^*}Y^*(K)+\alpha(X,Y)
		\end{align*}
		The homogeneous Koszul formula \cite[Theorem 13.1]{KN}, together with \eqref{eq:bracket-identity}, gives
		\begin{align*}
			\langle\alpha(X,Y),Z\rangle+\langle\alpha(X,Z),Y\rangle&=0,\\
			\alpha(X,Y)-\alpha(Y,X)&=[X,Y]_{\mathfrak{m}}
		\end{align*}
		for any $X,Y,Z\in\mathfrak{m}$, and thus
		\begin{align*}
			\alpha(X,Y)=\left\{
			\begin{array}{ll}
				\frac12[X,Y]_{\mathfrak{m}},&X,Y\in\mathfrak{m}_1,\\
				0,&X\in\mathfrak{m}_1,Y\in\mathfrak{m}_2,\\
				{[}X,Y]_{\mathfrak{m}_2},&X\in\mathfrak{m}_2,Y\in\mathfrak{m}_1,\\
				\frac12[X,Y]_{\mathfrak{m}},&X,Y\in\mathfrak{m}_2.				
			\end{array}
			\right.
		\end{align*}
		Lemma \ref{[X^*,Y^*]} now completes the proof.
	\end{proof}
	We conclude this section by describing a family of local geodesics on $G/K$.
	\begin{lemma}\label{geodesic}
		Let $X\in\mathfrak{m}_1\cup\mathfrak{m}_2$ and set $X'=\pi_{K*}X$. Then
		\begin{align*}
			\exp(tX)K=\exp_K(tX'),
		\end{align*}
		that is, $\gamma(t)=\exp(tX)K$ is a geodesic for all sufficiently small $t$.
	\end{lemma}
	\begin{proof}
		By direct calculation, we have
		\begin{align*}
			\dot\gamma(t)=\frac{d}{ds}\gamma(t+s)|_{s=0}=\tau_{\exp(tX)*}X'=\tau_{\sigma(\gamma(t))*}X'.
		\end{align*}
		Applying Proposition \ref{D=D^0+alpha} and noting that $\alpha(X,X)=0$ for $X\in\mathfrak{m}_1\cup\mathfrak{m}_2$, we obtain 
		\begin{align*}
			D_{\dot\gamma(t)}\dot\gamma(t)=D^0_{\dot\gamma(t)}\dot\gamma(t)
		\end{align*}
		for all sufficiently small $t$. It therefore remains to prove that $D^0_{\dot\gamma(t)}\dot\gamma(t)=0$. Let $X^*(gK)=\tau_{\sigma(gK)*}X'$ and let $\tilde X$ be its horizontal lift. Along $\gamma$, we have $X^*(\gamma(t))=\dot\gamma(t)$ and $\tilde X(\sigma(\gamma(t)))=L_{\sigma(\gamma(t))*}X$. By the definition of $f_{X^*}$, we have
		\begin{align*}
			f_{X^*}(g)=(\pi_{K*}|_{\mathfrak{m}})^{-1}\circ\tau_{g^{-1}\sigma(gK)*}X'
		\end{align*}
		and hence
		\begin{align*}
			\tilde X_{\sigma(\gamma(t))}(f_{X^*})=\frac{d}{ds}|_{s=0}f_{X^*}(\sigma(\gamma(t))\exp(sX))=\frac{d}{ds}    |_{s=0}X=0,		
		\end{align*}
		where we use $\sigma(\exp(sX)K)=\exp(sX)$. Therefore, we have
		\begin{align*}
			D^0_{\dot\gamma(t)}\dot\gamma(t)=\pi_{K*}\circ L_{\sigma(\gamma(t))*}\tilde X_{\sigma(\gamma(t))}(f_{X^*})=0.
		\end{align*}
	\end{proof}
	\subsection{The pullback BPST connection and its curvature}
	In this section, we give the definition of the pullback of the adjoint BPST instanton under the Hopf fibration.
	\begin{definition}\label{instanton}\cite{AW}
		Let $s\in\Gamma(E)$ be written as $s(gK)=[g,Y(g)]$, let $V\in T_{gK}G/K$, and set $X=(\pi_{K*}|_{\mathfrak{m}})^{-1}\circ\tau_{g^{-1}*}V\in\mathfrak{m}$. The pullback BPST connection $\nabla$ is given by
		\begin{align}
			\nabla_Vs(gK)=[g,\frac{d}{dt}|_{t=0}Y(g\exp(tX))].
		\end{align}
		It is easy to verify that for any $B(gK)=[g,B'(g)]\in\Gamma(\mathfrak{g}_E)$, the induced connection $\nabla$ on $\mathfrak{g}_E$ is given by
		\begin{align}
			\nabla_VB(gK)=[g,\frac{d}{dt}|_{t=0}B'(g\exp(tX))].
		\end{align}
	\end{definition}
	To compute the curvature of $\nabla$, we define
	\begin{align*}
		\Phi':\Gamma(E)\to C_K^\infty(G,\mathfrak{k})=\{f\in C^\infty(G,\mathfrak{k})\mid f(gk)=\operatorname{Ad}(k^{-1})f(g)\ \forall\ k\in K\}
	\end{align*}
	such that $\Phi'(s)(g)=X(g)$ for any $s(gK)=[g,X(g)]\in\Gamma(E)$. Using the same argument in Lemma \ref{Phi}, we can prove that $\Phi'$ is well-defined and is an isomorphism.
	\begin{proposition}\label{BPST curvature}
		The curvature of $\nabla$ is given by
		\begin{align}
			F(\tau_{g*}X',\tau_{g*}Y')_{gK}=[g,-[X,Y]_{\mathfrak{k}}]
		\end{align}
		for any $X,Y\in\mathfrak{m}$ and $X'=\pi_{K*}X$, $Y'=\pi_{K*}Y$, that is, for any $s(gK)=[g,Y(g)]\in\Gamma(E)$ and $B(gK)=[g,B'(g)]\in\Gamma(\mathfrak{g}_E)$, we have
		\begin{align*}
			&F(\tau_{g*}X',\tau_{g*}Y')s(gK)=[g,-[[X,Y]_{\mathfrak{k}},Y(g)]],\\
			&F(\tau_{g*}X',\tau_{g*}Y')B(gK)=[g,-[[X,Y]_{\mathfrak{k}},B'(g)]].
		\end{align*}
	\end{proposition}
	\begin{proof}
		Since $F$ is $G$-invariant, we only need to prove the formula at $K$. For any $X,Y\in\mathfrak{m}_1\cup\mathfrak{m}_2$, let
		\begin{align*}
			X^*(gK)=\tau_{\sigma(gK)*}\pi_{K*}X,\quad Y^*(gK)=\tau_{\sigma(gK)*}\pi_{K*}Y
		\end{align*}
		be vector fields on $U$ and
		\begin{align*}
			\tilde X(g)=L_{g*}\circ \operatorname{Ad}(g^{-1}\sigma(gK))X,\quad \tilde Y(g)=L_{g*}\circ \operatorname{Ad}(g^{-1}\sigma(gK))Y
		\end{align*}
		be the horizontal lifts on $\pi_K^{-1}(U)$. We also let $s(gK)=[g,f(g)]\in\Gamma(E)$. For all sufficiently small $r\in\mathbb{R}$, we have
		\begin{align*}
			\Phi'(\nabla_{Y^*}s)(\exp(rX))=\frac{d}{dt}|_{t=0}f(\exp(rX)\exp(tY))=\tilde Y(f)(\exp(rX))
		\end{align*}
		and thus
		\begin{align*}
			\Phi'(\nabla_{X^*}\nabla_{Y^*}s)(e)=\frac{d}{dr}|_{r=0}\tilde Y(f)(\exp(rX))=\tilde X\tilde Y(f)(e).
		\end{align*}
		Similarly, we have
		\begin{align*}
			\Phi'(\nabla_{Y^*}\nabla_{X^*}s)(e)=\tilde Y\tilde X(f)(e).
		\end{align*}
		Using Proposition \ref{D_X*Y*}, we have
		\begin{align*}
			\Phi'(\nabla_{[X^*,Y^*]}s)(e)=\frac{d}{dt}|_{t=0}f(\exp(t[X,Y]_{\mathfrak{m}}))=[X,Y]_{\mathfrak{m}}(f)(e).
		\end{align*}
		A similar argument shows that $[\tilde X,\tilde Y](e)=[X,Y]$. Hence, we obtain
		\begin{align*}
			\Phi'(F(X^*,Y^*)s)(e)&=\Phi'(\nabla_{X^*}\nabla_{Y^*}s-\nabla_{Y^*}\nabla_{X^*}s-\nabla_{[X^*,Y^*]}s)(e)\\
			&=\tilde X\tilde Y(f)(e)-\tilde Y\tilde X(f)(e)-[X,Y]_{\mathfrak{m}}(f)(e)\\
			&=[\tilde X,\tilde Y](f)(e)-[X,Y]_{\mathfrak{m}}(f)(e)\\
			&=[X,Y]_{\mathfrak{k}}(f).
		\end{align*}
		Finally, using $f\in C_K^\infty(G,\mathfrak{k})$, we have
		\begin{align*}
			[X,Y]_{\mathfrak{k}}(f)=\frac{d}{dt}|_{t=0}f(\exp(t[X,Y]_{\mathfrak{k}}))=\frac{d}{dt}|_{t=0}\operatorname{Ad}(\exp(-t[X,Y]_{\mathfrak{k}}))f(e)=-\operatorname{ad}([X,Y]_{\mathfrak{k}})f(e).
		\end{align*}
	\end{proof}
	\subsection{The equivariant Jacobi formula}
	In this subsection, we derive explicit formulas for the Jacobi operator and the codifferential in the equivariant realization. For the fixed connection considered here, write $J=\mathscr S^\nabla$. By the Bochner--Weitzenb\"ock formula \eqref{Bochner},
	\begin{align*}
		J(B)=\nabla^*\nabla B+2\mathfrak R^\nabla(B)+B\circ \operatorname{Ric}
	\end{align*}
	for every $B\in\Omega^1(\mathfrak g_E)$.
	
	Under the $G$-equivariant identifications constructed above, the operator $\tilde J$ is characterized by
	\begin{align*}
		\tilde J\circ\Psi\circ\Phi^{-1}\circ\Theta=\Psi\circ\Phi^{-1}\circ\Theta\circ J.
	\end{align*}
	Let $\Theta_0$, $\Phi_0^{-1}$, and $\Psi_0$ denote the corresponding identifications for zero-forms. The operator $\tilde\delta^\nabla$, abbreviated to $\tilde\delta$, is characterized by
	\begin{align*}
		\tilde\delta\circ\Psi\circ\Phi^{-1}\circ\Theta=\Psi_0\circ\Phi_0^{-1}\circ\Theta_0\circ\delta^\nabla.
	\end{align*}
	
	For $X\in\mathfrak g$ and $U\in C^\infty(G)$, we write
	\begin{align*}
		X(U)(g)=\left.\frac{d}{dt}\right|_{t=0}U(g\exp(tX)),\qquad XX(U)(g)=\left.\frac{d^2}{dt^2}\right|_{t=0}U(g\exp(tX)).
	\end{align*}
	Throughout this subsection, repeated indices $j$ and $\alpha$ are summed over $1\le j\le7$ and $8\le\alpha\le10$, respectively, unless otherwise indicated.
	
	Since the operators and identifications are $G$-equivariant, it suffices to compute at $K\in G/K$. For the basis $\{X_i\}_{i=1}^7$ of $\mathfrak m$ defined in \eqref{X_i}, set $X_i'=\pi_{K*}X_i$ and
	\begin{align*}
		X_i^*(gK)=\tau_{\sigma(gK)*}X_i'.
	\end{align*}
	The rough Laplacian is given by
	\begin{align*}
		\nabla^*\nabla B=-\sum_{i=1}^3(\nabla_{X_i^*}\nabla_{X_i^*}B-\nabla_{D_{X_i^*}X_i^*}B)-2\sum_{i=4}^7(\nabla_{X_i^*}\nabla_{X_i^*}B-\nabla_{D_{X_i^*}X_i^*}B).
	\end{align*}
	We first calculate the derivative term $\nabla^*\nabla B$. Proposition \ref{D_X*Y*} gives
	\begin{align*}
		D_{X_i^*}X_i^*(K)=0
	\end{align*}
	for every $i$. Therefore, it remains to calculate $\nabla_{X_i^*}\nabla_{X_i^*}B$ at $K$.
	\begin{lemma}\label{Dxi D^2xi}
		For each $i$, define
		\begin{align*}
			\xi_i^*(gK)=\tau_{\sigma(gK)^{-1}}^*\xi_i.
		\end{align*}
		Then $\{\xi_i^*\}_{i=1}^7$ is the local coframe dual to
		$\{X_i^*\}_{i=1}^7$.
		For any $X\in\mathfrak m_1\cup\mathfrak m_2$, we have
		\begin{align*}
			D_{X^*}\xi_j^*(K)=\left\{
			\begin{array}{ll}
				-\frac12\sum_{i=1}^3\langle X_j,[X,X_i]\rangle_{alg}\xi_i,&X\in\mathfrak{m}_1,\\
				-\sum_{i=1}^3\langle X_j,[X,X_i]\rangle_{alg}\xi_i-\frac12\sum_{i=4}^7\langle X_j,[X,X_i]\rangle_{alg}\xi_i,&X\in\mathfrak{m}_2.
			\end{array}
			\right.,\\
		\end{align*}
		and
		\begin{align*}
			D_{X^*}D_{X^*}\xi_j^*(K)=\left\{
			\begin{array}{ll}
				\frac14\sum_{i=1}^3\langle [X,[X,X_i]_{\mathfrak{m}}],X_j\rangle_{alg}\xi_i,&X\in\mathfrak{m}_1,\\
				\frac12\sum_{i=1}^7\langle [X,[X,X_i]_{\mathfrak{m}}],X_j\rangle_{alg}\xi_i,&X\in\mathfrak{m}_2.
			\end{array}
			\right.
		\end{align*}
	\end{lemma}
	\begin{proof}
		By Proposition \ref{D_X*Y*}, at $K$,
		\begin{align*}
			\langle D_{X^*}\xi_j^*,X_i^*\rangle&=X^*\langle\xi_j^*,X_i^*\rangle-\langle\xi_j^*,D_{X^*}X_i^*\rangle\\
			&=\left\{
			\begin{array}{ll}
				-\frac12\langle X_j,[X,X_i]\rangle,&X,X_i\in\mathfrak{m}_1\ or\ X,X_i\in\mathfrak{m}_2,\\
				0,&X\in\mathfrak{m}_1, X_i\in\mathfrak{m}_2\\
				-2\langle X_j,[X,X_i]\rangle,&X\in\mathfrak{m}_2, X_i\in\mathfrak{m}_1.
			\end{array}
			\right.
		\end{align*}
		Using the relation between $\langle\ ,\ \rangle$ and $\langle\ ,\ \rangle_{alg}$, we obtain the first formula.
		
		Differentiating the identity $\xi_j^*(X_i^*)=\delta_{ij}$ twice gives
		\begin{align*}
			\langle D_{X^*}D_{X^*}\xi_j^*,X_i^*\rangle=&X^*X^*\langle\xi_j^*,X_i^*\rangle-2\langle D_{X^*}\xi_j^*,D_{X^*}X_i^*\rangle-\langle\xi_j^*,D_{X^*}D_{X^*}X_i^*\rangle\\
			=&-2X^*\langle\xi_j^*,D_{X^*}X_i^*\rangle+\langle\xi_j^*,D_{X^*}D_{X^*}X_i^*\rangle.
		\end{align*}
		Since $D$ is $G$-invariant,
		\begin{align*}
			D_{X^*}X_i^*(\exp(tX)K)=\tau_{\exp(tX)*}(D_{\tau_{\exp(-tX)*}X^*}\tau_{\exp(-tX)*}X_i^*(K))
		\end{align*}
		Let $\hat X$ be the horizontal lift of $\tau_{\exp(-tX)*}X^*$. Then $\hat X(e)=X$. The definition of $D^0$ yields
		\begin{align*}
			&D^0_{\tau_{\exp(-tX)*}X^*}\tau_{\exp(-tX)*}X_i^*(K)\\
			=&\frac{d}{ds}|_{s=0}\pi_{K*}(f_{\tau_{\exp(-tX)*}X_i^*}(\exp(sX)))\\
			=&\frac{d}{ds}|_{s=0}\pi_{K*}\circ(\pi_{K*}|_{\mathfrak{m}})^{-1}\circ\tau_{\exp(-sX)*}(\tau_{\exp(-tX)*}X_i^*(\exp(sX)K))\\
			=&\frac{d}{ds}|_{s=0}X_i'\\
			=&0.
		\end{align*}
		Thus $D_{\tau_{\exp(-tX)*}X^*}\tau_{\exp(-tX)*}X_i^*(K)=\pi_{K*}\alpha(X,X_i)$. Therefore,
		\begin{align*}
			\langle\xi_j^*,D_{X^*}X_i^*\rangle(\exp(tX)K)=\langle\tau_{\exp(-tX)}^*\xi_j,\tau_{\exp(tX)*}\pi_{K*}\alpha(X,X_i)\rangle=\langle\xi_j,\pi_{K*}\alpha(X,X_i)\rangle
		\end{align*}
		is independent of $t$. Hence,
		\begin{align*}
			X^*\langle\xi_j^*,D_{X^*}X_i^*\rangle(K)=\frac{d}{dt}|_{t=0}\langle\xi_j^*,D_{X^*}X_i^*\rangle(\exp(tX)K)=0.
		\end{align*}
		Finally, we compute $D_{X^*}D_{X^*}X_i^*(K)$. By definition, we have
		\begin{align*}
			&D^0_{X^*}D_{X^*}X_i^*(K)\\
			=&\frac{d}{dt}|_{t=0}\pi_{K*}f_{D_{X^*}X_i^*}(\exp(tX))\\
			=&\frac{d}{dt}|_{t=0}\pi_{K*}\circ(\pi_{K*}|_{\mathfrak{m}})^{-1}\circ\tau_{\exp(-tX)*}(D_{X^*}X_i^*(\exp(tX)K))\\
			=&\frac{d}{dt}|_{t=0}\pi_{K*}\circ(\pi_{K*}|_{\mathfrak{m}})^{-1}\circ\tau_{\exp(-tX)*}\circ\tau_{\exp(tX)*}(D_{\tau_{\exp(-tX)*}X^*}\tau_{\exp(-tX)*}X_i^*(K))\\
			=&\frac{d}{dt}|_{t=0}\pi_{K*}\alpha(X,X_i)\\
			=&0.
		\end{align*}
		Thus
		\begin{align*}
			D_{X^*}D_{X^*}X_i^*(K)&=\pi_{K*}\alpha(X,\alpha(X,X_i))\\
			&=\left\{
			\begin{array}{ll}
				\frac14\pi_{K*}[X,[X,X_i]_{\mathfrak{m}}],&X,X_i\in\mathfrak{m}_1,\\
				0,&X\in\mathfrak{m}_1,X_i\in\mathfrak{m}_2,\\
				\frac12\pi_{K*}[X,[X,X_i]_{\mathfrak{m}}],&X\in\mathfrak{m}_2.
			\end{array}
			\right.
		\end{align*}
		Substituting this formula into the preceding identity completes the proof.
	\end{proof}
	For $B\in\Omega^1(\mathfrak g_E)$, write
	\begin{align*}
		\Psi\circ\Phi^{-1}\circ\Theta(B)=\sum_{j=1}^7\sum_{\alpha=8}^{10}U_{j\alpha}\otimes\xi_j\otimes X_\alpha,
	\end{align*}
	and set $U_j=\sum_{\alpha=8}^{10}U_{j\alpha}X_\alpha$. For $g\in\sigma(U)$, the local expression for $B$ is
	\begin{align*}
		B(gK)=\sum_{j=1}^7\xi_j^*(gK)\otimes[g,U_j(g)].
	\end{align*}
	Lemma \ref{Dxi D^2xi} then gives
	\begin{align*}
		&-\sum_{i=1}^3\nabla_{X_i^*}\nabla_{X_i^*}B(K)-2\sum_{i=4}^7\nabla_{X_i^*}\nabla_{X_i^*}B(K)\\
		=&-\frac14\sum_{i,k=1}^3\langle[X_i,[X_i,X_k]_{\mathfrak{m}}],X_j\rangle_{alg}\xi_k\otimes [e,U_j(e)]-\sum_{i=4}^7\sum_{k=1}^7\langle[X_i,[X_i,X_k]_{\mathfrak{m}}],X_j\rangle_{alg}\xi_k\otimes [e,U_j(e)]\\
		&+\sum_{i,k=1}^3\langle X_j,[X_i,X_k]\rangle_{alg}\xi_k\otimes[e,X_i(U_j)(e)]+4\sum_{i=4}^7\sum_{k=1}^3\langle X_j,[X_i,X_k]\rangle_{alg}\xi_k\otimes[e,X_i(U_j)(e)]\\
		&+2\sum_{i,k=4}^7\langle X_j,[X_i,X_k]\rangle_{alg}\xi_k\otimes[e,X_i(U_j)(e)]-\sum_{i=1}^3\xi_j\otimes[e,X_iX_i(U_j)(e)]\\
		&-2\sum_{i=4}^7\xi_j\otimes[e,X_iX_i(U_j)(e)],
	\end{align*}
	Combining the preceding calculation with the curvature formula gives the following global expressions.
	
	\begin{proposition}\label{equivariant Jacobi formula}
		For $\tilde B=U_{j\alpha}\otimes\xi_j\otimes X_\alpha\in(C^\infty(G)\otimes\mathfrak m^*\otimes\mathfrak k)_K$, we have
		\begin{align*}
			\tilde J(\tilde B)=\tilde J_2(\tilde B)+\tilde J_1(\tilde B)+\tilde J_0(\tilde B),
		\end{align*}
		where
		\begin{align*}
			\tilde J_2(\tilde B)=&-\sum_{i=1}^3X_iX_i(U_{j\alpha})\otimes\xi_j\otimes X_\alpha-2\sum_{i=4}^7X_iX_i(U_{j\alpha})\otimes\xi_j\otimes X_\alpha,\\
			\tilde J_1(\tilde B)=&\sum_{i,k=1}^3\langle X_j,[X_i,X_k]\rangle_{alg}X_i(U_{j\alpha})\otimes\xi_k\otimes X_\alpha+4\sum_{i=4}^7\sum_{k=1}^3\langle X_j,[X_i,X_k]\rangle_{alg}X_i(U_{j\alpha})\otimes\xi_k\otimes X_\alpha\\
			&+2\sum_{i,k=4}^7\langle X_j,[X_i,X_k]\rangle_{alg}X_i(U_{j\alpha})\otimes\xi_k\otimes X_\alpha,\\
			\tilde J_0(\tilde B)=&-\frac14\sum_{i,k=1}^3\langle[X_i,[X_i,X_k]_{\mathfrak m}],X_j\rangle_{alg}U_{j\alpha}\otimes\xi_k\otimes X_\alpha-\sum_{i=4}^7\sum_{k=1}^7\langle[X_i,[X_i,X_k]_{\mathfrak m}],X_j\rangle_{alg}U_{j\alpha}\otimes\xi_k\otimes X_\alpha\\
			&-2\sum_{i=1}^3U_{i\alpha}\otimes\xi_j\otimes[[X_i,X_j]_{\mathfrak k},X_\alpha]-4\sum_{i=4}^7U_{i\alpha}\otimes\xi_j\otimes[[X_i,X_j]_{\mathfrak k},X_\alpha]+6U_{j\alpha}\otimes\xi_j\otimes X_\alpha.
		\end{align*}
		The corresponding codifferential is
		\begin{align*}
			\tilde\delta\tilde B=-\left(\sum_{i=1}^3X_i(U_{i\alpha})+2\sum_{i=4}^7X_i(U_{i\alpha})\right)\otimes X_\alpha.
		\end{align*}
		Here $\tilde J_i$ are the terms in $\tilde J$ containing the $i$-th derivative of $U_{j\alpha}$ only.
	\end{proposition}
	
	\begin{proof}
		At $K$, the preceding calculation gives the rough Laplacian in the local frame. Substituting the curvature formula from Proposition \ref{BPST curvature} and $\operatorname{Ric}=6\operatorname{id}$ into the expression for $J$ yields the stated formulas at $e$. The formula for $\tilde\delta$ follows from the local expression for $\delta^\nabla$, using $D_{X_i^*}X_i^*(K)=0$.
		
		Since the identifications and the geometric operators are $G$-equivariant, their values at $e$ determine the corresponding operators globally. Moreover, each $X_i$ commutes with left translations on $G$. Consequently, the same coefficient formulas hold at every $g\in G$, proving the proposition.
	\end{proof}
	\section{The nonpositive spectrum}\label{section spectrum}
	By the Peter-Weyl theorem (see \cite[Theorem 5.12]{BF}), $L^2(G,\mathbb{C})$ is the Hilbert direct sum of the spaces of matrix coefficients of the irreducible representations of $G$. However, it is complicated to calculate the eigenvalues of $\tilde J$ on every such invariant subspace. Therefore, we first estimate these eigenvalues to identify the representations on which non-positive eigenvalues may occur.
	\subsection{Representations and Casimir operators}
	In this subsection, we describe the construction of the finite-dimensional irreducible complex representations of $G$. Define $\ssp=\mathfrak{sp}(2)\otimes_{\mathbb{R}}\mathbb{C}$ and let
	\begin{align*}
		\spp=\{
		\begin{pmatrix}
			A&B\\
			C&-A^T
		\end{pmatrix}
		|A,B,C\in M_2(\mathbb{C}), B=B^T, C=C^T\}\subset M_4(\mathbb{C}).
	\end{align*}
	For any $Z\in M_2(\mathbb{H})$, there exist unique $X,Y\in M_2(\mathbb{C})$ such that $Z=X+jY$. Define $\chi:M_2(\mathbb{H})\to M_4(\mathbb{C})$ by
	\begin{align}\label{chi}
		\chi(Z)=
		\begin{pmatrix}
			X&-\bar Y\\
			Y&\bar X
		\end{pmatrix}.
	\end{align}
	Then $\chi(\mathfrak{sp}(2))\subset\spp$ and $\chi(Z_1Z_2)=\chi(Z_1)\chi(Z_2)$. Furthermore, $\chi$ induces a complex Lie algebra isomorphism
	\begin{align*}
		\chi_{\mathbb{C}}:\ssp\to\spp
	\end{align*}
	given by
	\begin{align*}
		\chi_{\mathbb{C}}(Z_1\otimes1+Z_2\otimes i)=\chi(Z_1)+i\chi(Z_2).
	\end{align*}
	
	Since $\operatorname{Sp}(2)$ is simply connected, the results of Hall \cite[Theorem 4.4-4.6, Theorem 5.6]{Hall} imply that differentiation and integration induce a natural correspondence
	\begin{align*}
		\operatorname{Irr}_{\mathbb{C}}(\operatorname{Sp}(2))\cong \operatorname{Irr}_{fd}(\spp).
	\end{align*}
	We shall use this correspondence without recalling its proof or explicit construction. In what follows, we also use $d\pi$ to denote the complex-linear extension of the differential of $\pi$, regarded as a representation of $\spp$ via $\chi_{\mathbb{C}}$.
	
	Fulton and Harris \cite[Section 16.2]{WJ} show that the finite-dimensional irreducible representations of $\spp$ are classified by their highest weights and hence by pairs of non-negative integers $(a,b)\in\mathbb{N}\times\mathbb{N}$. Following \cite{WJ}, we now construct the corresponding representation space $V_{a,b}$ and denote the representation of $G$ on it by $\pi_{a,b}:G\to \operatorname{GL}(V_{a,b})$.
	
	$\mathbf{V_{0,0}:\ }$Let $V_{0,0}=\mathbb{C}$ with the representation defined by
	\begin{align*}
		\pi_{0,0}(g)=\operatorname{Id}_{V_{0,0}}
	\end{align*}
	for any $g\in G$, or equivalently,
	\begin{align*}
		d\pi_{0,0}(X)=0
	\end{align*}
	for any $X\in\spp$.
	
	$\mathbf{V_{1,0}:\ }$Let $V_{1,0}=\mathbb{C}^4$ with the representation defined by
	\begin{align*}
		\pi_{1,0}(g)v=\chi(g)v
	\end{align*}
	for any $g\in G$ and $v\in V_{1,0}$, or equivalently,
	\begin{align*}
		d\pi_{1,0}(X)v=Xv
	\end{align*}
	for any $X\in\spp$.
	
	$\mathbf{V_{0,1}:\ }$Let $\Lambda^2\mathbb{C}^4=\operatorname{span}_{\mathbb{C}}\{\xi_1,\dots,\xi_6\}$, where
	\begin{align*}
		\xi_1=e_1\wedge e_2,\quad \xi_2=e_1\wedge e_3,\quad \xi_3=e_1\wedge e_4,\quad \xi_4=e_2\wedge e_3,\quad \xi_5=e_2\wedge e_4,\quad \xi_6=e_3\wedge e_4.
	\end{align*}
	Let $\Omega=\begin{pmatrix}0&I\\-I&0\end{pmatrix}$ and define $c_\Omega:\Lambda^2\mathbb{C}^4\to\mathbb{C}$ by
	\begin{align*}
		c_\Omega(w\wedge v)=w^T\Omega v.
	\end{align*}
	A direct computation gives $c_\Omega(\xi_1)=c_\Omega(\xi_3)=c_\Omega(\xi_4)=c_\Omega(\xi_6)=0$ and $c_\Omega(\xi_2)=c_\Omega(\xi_5)=1$. We also define
	\begin{align*}
		V_{0,1}=\ker(c_\Omega)=\operatorname{span}_{\mathbb{C}}\{\xi_1,\xi_2-\xi_5,\xi_3,\xi_4,\xi_6\}.
	\end{align*}
	The induced action of $G$ on $\Lambda^2\mathbb{C}^4$ is given by
	\begin{align*}
		\pi_{0,1}(g)(w\wedge v)=(\chi(g)w)\wedge(\chi(g)v)
	\end{align*}
	for any $g\in G$ and $w,v\in\mathbb{C}^4$. Since $\chi(g)^T\Omega\chi(g)=\Omega$, this action preserves $V_{0,1}$. We denote its restriction to $V_{0,1}$ by $\pi_{0,1}$.
	
	$\mathbf{V_{a,b}:\ }$Let $v_{a,b}=e_1^{\otimes a}\otimes \xi_1^{\otimes b}$ and
	\begin{align*}
		V_{a,b}=\operatorname{span}_{\mathbb{C}}\{X_1\cdots X_rv_{a,b}\mid r\ge0,\ X_1,\dots,X_r\in\spp\},
	\end{align*}
	where
	\begin{align*}
		X(v_1\otimes\dots\otimes v_a\otimes\zeta_1\otimes\dots\otimes\zeta_b)=\sum_{j=1}^av_1\otimes\dots\otimes(d\pi_{1,0}(X)v_j)\otimes\dots\otimes\zeta_b+\sum_{k=1}^bv_1\otimes\dots\otimes(d\pi_{0,1}(X)\zeta_k)\otimes\dots\otimes\zeta_b.
	\end{align*}
	The tensor-product representation $\pi_{1,0}^{\otimes a}\otimes\pi_{0,1}^{\otimes b}$ preserves $V_{a,b}$. We denote its restriction to $V_{a,b}$ by $\pi_{a,b}$.
	
	We next introduce the real matrix-coefficient spaces used in the Peter-Weyl decomposition. Let $l_{a,b}=\dim_{\mathbb{C}}V_{a,b}$ and choose a $G$-invariant Hermitian inner product on $V_{a,b}$ with an orthonormal basis $\{e_s\mid1\le s\le l_{a,b}\}$. For $1\le s,t\le l_{a,b}$, define $U^{st}\in C^\infty(G,\mathbb{C})$ by
	\begin{align}\label{U^st}
		U^{st}(g):=\langle\pi_{a,b}(g)e_s,e_t\rangle
	\end{align}
	and let
	\begin{align*}
		\mathcal{U}_{a,b}=\operatorname{span}_{\mathbb{R}}\{\operatorname{Re}(U^{st}), \operatorname{Im}(U^{st})\mid1\le s,t\le l_{a,b}\}.
	\end{align*}
	Since every irreducible complex representation of $G$ is self-dual, its matrix-coefficient space is invariant under complex conjugation, and $\mathcal{U}_{a,b}$ is a real form of this space. By the Peter-Weyl theorem \cite[(5.6), (5.8), (5.12)]{BF},
	\begin{align*}
		L^2(G,\mathbb{R})=\overline{\bigoplus_{a,b\ge0}\mathcal{U}_{a,b}}.
	\end{align*}
	Since $\tilde J$ is $G$-equivariant, Schur's lemma \cite[Theore 3.5]{BF} implies that the subspace 
	\begin{align*}
		(\mathcal{U}_{a,b}\otimes\mathfrak{m}^*\otimes\mathfrak{k})_K
	\end{align*}
	is invariant under $\tilde J$ for any $a,b\in\mathbb{N}$. The second-order term can be expressed in terms of the Casimir operators of $G$, $H$, and $K$. Define the Casimir differential operators on $C^\infty(G)$ by
	\begin{align*}
		C_G=-\sum_{i=1}^{10}X_i^2,\quad C_H=-\sum_{i=1}^3X_i^2,\quad C_K=-\sum_{\alpha=8}^{10}X_\alpha^2.
	\end{align*}
	For $v,w\in V_{a,b}$, let $U_{v,w}(g)=\langle\pi_{a,b}(g)v,w\rangle$. Since $X(U_{v,w})=U_{d\pi_{a,b}(X)v,w}$, the restrictions of $C_G$, $C_H$, and $C_K$ to $\mathcal{U}_{a,b}$ are induced, respectively, by
	\begin{align*}
		-\sum_{i=1}^{10}d\pi_{a,b}(X_i)^2,\quad-\sum_{i=1}^3d\pi_{a,b}(X_i)^2,\quad-\sum_{\alpha=8}^{10}d\pi_{a,b}(X_\alpha)^2.
	\end{align*}
	Here $C_H$ and $C_K$ correspond to the first and second $\operatorname{Sp}(1)$-factors of $H$, respectively. By the definition of $\tilde J_2$, we have
	\begin{align*}
		\tilde J_2=(2C_G-C_H-2C_K)\otimes I\otimes I
	\end{align*}
	on $(C^\infty(G)\otimes\mathfrak{m}^*\otimes\mathfrak{k})_K$.
	\begin{lemma}\label{eigenvalue of J2}
		As an $H$-module, $V_{a,b}$ admits the multiplicity-free decomposition
		\begin{align*}
			V_{a,b}=\bigoplus_{r=0}^b\bigoplus_{s=0}^aV_{a,b}(r+s,a+r-s),
		\end{align*}
		where $V_{a,b}(p,q)\cong \operatorname{Sym}^p(\mathbb{C}^2)\otimes \operatorname{Sym}^q(\mathbb{C}^2)$. For $p=r+s$ and $q=a+r-s$, define
		\begin{align*}
			\mathcal{U}_{a,b}(p,q)=\operatorname{span}_{\mathbb{R}}\{\operatorname{Re}(U_{v,w}),\operatorname{Im}(U_{v,w})\mid v\in V_{a,b}(p,q),\ w\in V_{a,b}\}.
		\end{align*}
		Then
		\begin{align*}
			\mathcal{U}_{a,b}=\bigoplus_{r=0}^b\bigoplus_{s=0}^a\mathcal{U}_{a,b}(r+s,a+r-s),
		\end{align*}
		and
		\begin{align*}
			\tilde J_2=\lambda_{a,b}(p,q)\operatorname{Id}
		\end{align*}
		on $(\mathcal{U}_{a,b}(p,q)\otimes\mathfrak{m}^*\otimes\mathfrak{k})_K$, where
		\begin{align*}
			\lambda_{a,b}(p,q)=2(a^2+2ab+2b^2+4a+6b)-p(p+2)-2q(q+2).
		\end{align*}
	\end{lemma}
	\begin{proof}
		In the notation of \cite{FEC}, the highest weight of $V_{a,b}$ has coordinates $(h_1,h_2)=(a+b,b)$. Specializing \cite[Theorem 3]{FEC} to $\operatorname{Sp}(2)\supset \operatorname{Sp}(1)\times \operatorname{Sp}(1)$ gives
		\begin{align*}
			V_{a,b}\downarrow_{\operatorname{Sp}(1)\times \operatorname{Sp}(1)}=\bigoplus_{r=0}^b\bigoplus_{s=0}^aV_{a,b}(r+s,a+r-s),
		\end{align*}
		where for $p=r+s$ and $q=a+r-s$, we have
		\begin{align*}
			V_{a,b}(p,q)\cong \operatorname{Sym}^p(\mathbb{C}^2)\otimes \operatorname{Sym}^q(\mathbb{C}^2).
		\end{align*}
		The corresponding decomposition of $\mathcal{U}_{a,b}$ follows by decomposing the first slot of its matrix coefficients. With respect to the normalization determined by $\langle\cdot,\cdot\rangle_{alg}$, the Casimir operators defined above are 4 times those used in \cite[Table 14]{JB}. Hence
		\begin{align*}
			C_G=(a^2+2ab+2b^2+4a+6b)I
		\end{align*}
		on $\mathcal{U}_{a,b}$, while
		\begin{align}\label{C_H,C_K}
			C_H=p(p+2)I,\quad C_K=q(q+2)I
		\end{align}
		on $\mathcal{U}_{a,b}(p,q)$. Combining these identities with $\tilde J_2=(2C_G-C_H-2C_K)\otimes I\otimes I$ gives the result.
	\end{proof}
	\subsection{Eigenvalue estimates for $\tilde J$}
	The aim of this subsection is to reduce the determination of the nonpositive spectrum of $\tilde J$ to 4 representations. We begin with the main estimate, which controls the first-order term by the second-order part and explicit fiberwise constants.
	\begin{lemma}\label{eigenvalue of J1}
		Equip $C^\infty(G)\otimes\mathfrak{m}^*\otimes\mathfrak{k}$ with the inner product defined on pure tensors by
		\begin{align*}
			\langle U\otimes\xi\otimes X,U'\otimes\xi'\otimes X'\rangle:=\int_GUU'dV\cdot\langle\xi,\xi'\rangle\cdot\langle X,X'\rangle,
		\end{align*}
		and denote the induced norm by $\|\cdot\|$. Then, in the sense of quadratic forms,
		\begin{align*}
			\tilde J\ge\frac12\tilde J_2+\tilde J_0-12P_1-6P_2,
		\end{align*}
		where $P_i:(C^\infty(G)\otimes\mathfrak{m}^*\otimes\mathfrak{k})_K\to(C^\infty(G)\otimes\mathfrak{m}_i^*\otimes\mathfrak{k})_K$ denotes the orthogonal projection for $i=1,2$.
	\end{lemma}
	Before returning to the proof of this estimate, we record the spectral information needed for its application. The following three lemmas describe the zero-order term and the restrictions imposed by $K$-equivariance. The tables used in the computations are collected in Appendix \ref{app fiber calculations}.
	
	Since $\tilde J_0$ is the tensor product of the identity map on $C^\infty(G)$ with a linear endomorphism of $\mathfrak{m}^*\otimes\mathfrak{k}$, its fiberwise eigenvalues and eigenspaces can be computed directly.
	\begin{lemma}\label{eigenvalue of J0}
		The fiberwise endomorphism $\tilde J_0|_e$ has eigenvalues $12$, $17$, and $-7$, with multiplicities $9$, $8$, and $4$, respectively. Moreover, the global eigenspaces corresponding to $12$ and $-7$ are
		\begin{align*}
			&(C^\infty(G)\otimes\mathfrak{m}_1^*\otimes\mathfrak{k})_K=\{\sum_{i=1}^3\sum_{\alpha=8}^{10}U_{i\alpha}\otimes\xi_i\otimes X_\alpha\in(C^\infty(G)\otimes\mathfrak{m}^*\otimes\mathfrak{k})_K\},\\
			&(C^\infty(G)\otimes\mathfrak{m}^*\otimes\mathfrak{k})_K\cap C^\infty(G)\otimes \operatorname{span}_{\mathbb{R}}\{F_1,F_2,F_3,F_4\},
		\end{align*}
		respectively, where
		\begin{equation}\label{F1-F4}
			\begin{split}
				&F_1=\xi_7\otimes X_9-\xi_4\otimes X_8-\xi_6\otimes X_{10},\\
				&F_2=\xi_5\otimes X_{10}-\xi_4\otimes X_9-\xi_7\otimes X_8,\\
				&F_3=\xi_6\otimes X_8-\xi_4\otimes X_{10}-\xi_5\otimes X_9,\\
				&F_4=\xi_5\otimes X_8+\xi_6\otimes X_9+\xi_7\otimes X_{10}.
			\end{split}
		\end{equation}
	\end{lemma}
	\begin{proof}
		
		Using the bracket identities in Tables \ref{[X_i,X_j]} and \ref{[X_alpha,X_beta]}, we obtain
		\begin{align*}
			&\sum_{i=1}^3[X_i,[X_i,X_k]_{\mathfrak{m}}]_{\mathfrak{m}}=-8X_k,\\
			&\sum_{i=4}^7[X_i,[X_i,X_k]_{\mathfrak{m}}]_{\mathfrak{m}}=-4X_k
		\end{align*}
		for any $1\le k\le3$ and
		\begin{align*}
			\sum_{i=4}^7[X_i,[X_i,X_k]_{\mathfrak{m}}]_{\mathfrak{m}}=-3X_k
		\end{align*}
		for $4\le k\le7$. Thus we have
		\begin{align*}
			&\tilde J_0(\sum_{j=1}^7\sum_{\alpha=8}^{10}U_{j\alpha}\otimes\xi_j\otimes X_\alpha)(e)\\
			=&12\sum_{j=1}^3\sum_{\alpha=8}^{10}U_{j\alpha}(e)\otimes\xi_j\otimes X_\alpha+9\sum_{j=4}^7\sum_{\alpha=8}^{10}U_{j\alpha}(e)\otimes\xi_j\otimes X_\alpha-4\sum_{i,j=4}^7\sum_{\alpha=8}^{10}U_{j\alpha}(e)\otimes\xi_i\otimes[[X_j,X_i]_{\mathfrak{k}},X_\alpha].
		\end{align*}
		Therefore,
		\begin{align*}
			\tilde J_0=12\operatorname{Id}
		\end{align*}
		on $(C^\infty(G)\otimes\mathfrak{m}_1^*\otimes\mathfrak{k})_K$. Define
		\begin{align*}
			T(\xi_j\otimes X_\alpha)=-4\sum_{i=4}^7\xi_i\otimes[[X_j,X_i]_{\mathfrak{k}},X_\alpha]
		\end{align*}
		for $4\le j\le7$, $8\le\alpha\le 10$.
		
		On $\mathfrak m_2^*\otimes\mathfrak k$, the fiber endomorphism is $9I+T$. The calculation in Appendix \ref{app fiber calculations} shows that $T$ has eigenvalues $8$ and $-16$, with multiplicities $8$ and $4$, respectively, and that its $-16$-eigenspace is spanned by $F_1,\dots,F_4$.
	\end{proof}
	For the subsequent calculations, we give an infinitesimal characterization of $K$-equivariance on $C^\infty(G)\otimes\mathfrak{m}^*\otimes\mathfrak{k}$.
	\begin{lemma}\label{K-equivariant equivalence}
		Let $\tilde B=\sum_{i=1}^7\sum_{\alpha=8}^{10}U_{i\alpha}\otimes\xi_i\otimes X_\alpha\in C^\infty(G)\otimes\mathfrak{m}^*\otimes\mathfrak{k}$. Then\\
		(i)\ $\tilde B\in(C^\infty(G)\otimes\mathfrak{m}^*\otimes\mathfrak{k})_K$ is equivalent to
		\begin{align*}
			(X_\beta\otimes I\otimes I)(\tilde B)=(I\otimes(\operatorname{ad}^*(X_\beta)\otimes I-I\otimes \operatorname{ad}(X_\beta)))(\tilde B)
		\end{align*}
		for any $\beta=8,9,10$.\\
		(ii)\ If $\tilde B\in(C^\infty(G)\otimes\mathfrak{m}^*\otimes\mathfrak{k})_K$, then we have
		\begin{align*}
			(-C_K\otimes I\otimes I)(\tilde B)=(I\otimes\sum_{\beta=8}^{10}(\operatorname{ad}^*(X_\beta)\otimes I-I\otimes \operatorname{ad}(X_\beta))^2)(\tilde B).
		\end{align*}
		(iii)\ The operators $-C_K\otimes I\otimes I$ and $I\otimes\sum_{\beta=8}^{10}(\operatorname{ad}^*(X_\beta)\otimes I-I\otimes \operatorname{ad}(X_\beta))^2$ map $(C^\infty(G)\otimes\mathfrak{m}^*\otimes\mathfrak{k})_K$ into itself.
	\end{lemma}
	\begin{proof}
		(i)\ Substituting $k=\exp(tX_\beta)$ into the $K$-equivariance condition for any $\beta=8,9,10$ and small $t\in\mathbb{R}$, we obtain
		\begin{align}\label{K-equivariance on U_{a,b}}
			\sum_{i=1}^7\sum_{\alpha=8}^{10}U_{i\alpha}(\exp(tX_\beta))\otimes\xi_i\otimes X_\alpha=\sum_{i=1}^7\sum_{\alpha=8}^{10}U_{i\alpha}(e)\otimes \operatorname{Ad}^*(\exp(tX_\beta))\xi_i\otimes \operatorname{Ad}(\exp(-tX_\beta))X_\alpha
		\end{align}
		for any $8\le\beta\le10$. Differentiating with respect to $t$ at $t=0$, we obtain
		\begin{align*}
			X_\beta\otimes I\otimes I=I\otimes(\operatorname{ad}^*(X_\beta)\otimes I-I\otimes \operatorname{ad}(X_\beta))
		\end{align*}
		for any $X_\beta\in\mathfrak{k}$.
		
		Conversely, suppose that these identities hold for $8\le\beta\le10$. They then hold for every $X\in\mathfrak{k}$. For fixed $g\in G$ and $X\in\mathfrak{k}$, the maps
		\begin{align*}
			t\mapsto\tilde B(g\exp(tX)),\qquad t\mapsto(\operatorname{Ad}^*(\exp(tX))\otimes \operatorname{Ad}(\exp(-tX)))\tilde B(g)
		\end{align*}
		satisfy the same first-order differential equation and have the same initial value. Hence they are equal. Since $K$ is connected and generated by its one-parameter subgroups, the $K$-equivariance of $\tilde B$ follows.\\
		(ii)\ Applying $X_\beta\otimes I\otimes I$ once more to the identity in part (i), and using the fact that operators acting on different tensor factors commute, gives
		\begin{align*}
			(X_\beta^2\otimes I\otimes I)(\tilde B)=\left(I\otimes(\operatorname{ad}^*(X_\beta)\otimes I-I\otimes \operatorname{ad}(X_\beta))^2\right)(\tilde B).
		\end{align*}
		Summing over $\beta$ and using $C_K=-\sum_{\beta=8}^{10}X_\beta^2$ proves (ii).\\
		(iii)\ Using Table \ref{[X_alpha,X_beta]}, we can calculate that
		\begin{align*}
			\sum_{\alpha=8}^{10}X_\alpha\circ[X_\alpha,X_\beta]+[X_\alpha,X_\beta]\circ X_\alpha=0
		\end{align*}
		for any $8\le\beta\le10$, and thus
		\begin{align*}
			-C_K\circ X_\beta=X_\beta\circ(-C_K)+\sum_{\alpha=8}^{10}X_\alpha\circ[X_\alpha,X_\beta]+[X_\alpha,X_\beta]\circ X_\alpha=X_\beta\circ(-C_K).
		\end{align*}
		Then we have
		\begin{align*}
			&(X_\beta\otimes I\otimes I)\circ(-C_K\otimes I\otimes I)\\
			=&(-C_K\otimes I\otimes I)\circ(X_\beta\otimes I\otimes I)\\
			=&(-C_K\otimes I\otimes I)\circ(I\otimes(\operatorname{ad}^*(X_\beta)\otimes I-I\otimes \operatorname{ad}(X_\beta)))\\
			=&(I\otimes(\operatorname{ad}^*(X_\beta)\otimes I-I\otimes \operatorname{ad}(X_\beta)))\circ(-C_K\otimes I\otimes I)
		\end{align*}
		for any $8\le\beta\le10$. Hence $-(C_K\otimes I)$ maps $(C^\infty(G)\otimes\mathfrak{m}^*\otimes\mathfrak{k})_K$ to itself.
	\end{proof}
	By Lemma \ref{eigenvalue of J2}, $\tilde J_2$ acts by scalar multiplication on each $(\mathcal{U}_{a,b}(p,q)\otimes\mathfrak{m}^*\otimes\mathfrak{k})_K$. Since $\tilde J_2$ and $\tilde J_0$ commute, they can be diagonalized simultaneously. The $K$-equivariance condition gives the following relation between the eigenvalue of $\tilde J_0$ and $q$.
	\begin{lemma}\label{lambda and q}
		Suppose that $0\ne\tilde B=\sum_{i=1}^7\sum_{\alpha=8}^{10}U_{i\alpha}\otimes\xi_i\otimes X_\alpha\in(\mathcal{U}_{a,b}(p,q)\otimes\mathfrak{m}^*\otimes\mathfrak{k})_K$ is an eigenvector of both $\tilde J_0$ and $\tilde J_2$ and the corresponding eigenvalue of $\tilde J_0$ is $\lambda$. Then
		\begin{align*}
			(\lambda,q)=(12,2)\ or\ (17,3)\ or\ (-7,1).
		\end{align*}
	\end{lemma}
	\begin{proof}
		Applying Lemma \ref{K-equivariant equivalence}, we have
		\begin{align*}
			&-(C_K\otimes I)(\tilde B)(e)\\
			=&\sum_{i=1}^7\sum_{\alpha,\beta=8}^{10}(U_{i\alpha}(e)\otimes \operatorname{ad}^*(X_\beta)^2\xi_i\otimes X_\alpha+U_{i\alpha}(e)\otimes\xi_i\otimes \operatorname{ad}(X_\beta)^2X_\alpha-2U_{i\alpha}(e)\otimes \operatorname{ad}^*(X_\beta)\xi_i\otimes \operatorname{ad}(X_\beta)X_\alpha).
		\end{align*}
		By \eqref{C_H,C_K}, the left-hand side is $-q(q+2)\tilde B$.
		
		Using the identity $\operatorname{ad}^*(X_\beta)\xi_i=\sum_{j=1}^7\langle X_\beta,[X_j,X_i]\rangle_{alg}\xi_j$ and Table \ref{[X_i,X_j]}, we obtain $\operatorname{ad}^*(X_\beta)\xi_i=0$ for $1\le i\le3$. The action on $\mathfrak m_2^*$ is recorded in Table \ref{ad^*(X_beta)xi_i}. Moreover, using Table \ref{[X_alpha,X_beta]}, we can prove that
		\begin{align*}
			\sum_{\beta=8}^{10}\operatorname{ad}(X_\beta)^2X_\alpha=-8X_\alpha
		\end{align*}
		for any $\alpha=8,9,10$. By Lemma \ref{eigenvalue of J0}, either $\lambda=12$ and $\tilde B\in(C^\infty(G)\otimes\mathfrak{m}_1^*\otimes\mathfrak{k})_K$, or $\lambda\in\{17,-7\}$ and $\tilde B\in(C^\infty(G)\otimes\mathfrak{m}_2^*\otimes\mathfrak{k})_K$.\\
		\emph{Case 1: $\tilde B$ has the form $\sum_{i=1}^3\sum_{\alpha=8}^{10}U_{i\alpha}\otimes\xi_i\otimes X_\alpha$.}\ For any $1\le i\le3$, we have $\operatorname{ad}^*(X_\beta)\xi_i=0$ and thus
		\begin{align*}
			&\sum_{\beta=8}^{10}U_{i\alpha}(e)\otimes \operatorname{ad}^*(X_\beta)^2\xi_i\otimes X_\alpha+U_{i\alpha}(e)\otimes\xi_i\otimes \operatorname{ad}(X_\beta)^2X_\alpha+2U_{i\alpha}\otimes \operatorname{ad}^*(X_\beta)\xi_i\otimes \operatorname{ad}(X_\beta)X_\alpha\\
			=&-8\tilde B(e).
		\end{align*}
		Hence $-q(q+2)=-8$, that is, $q=2$ or $-4$. In this case, $q=2$ and the corresponding eigenvalue of $\tilde J_0$ is $12$ (since $q\ge0$).\\
		\emph{Case 2: $\tilde B$ has the form $\sum_{i=4}^7U_{i\alpha}\otimes\xi_i\otimes X_\alpha$.}\ By direct calculation and using Table \ref{ad^*(X_beta)xi_i}, we have
		\begin{align*}
			\sum_{\beta=8}^{10}\operatorname{ad}^*(X_\beta)^2\xi_i=-3\xi_i
		\end{align*}
		for any $4\le i\le7$. Note that
		\begin{align*}
			&2U_{i\alpha}(e)\otimes \operatorname{ad}^*(X_\beta)\xi_i\otimes \operatorname{ad}(X_\beta)X_\alpha\\
			=&2\sum_{i,j=4}^7U_{i\alpha}\otimes\xi_j\otimes[[X_j,X_i]_{\mathfrak{k}},X_\alpha]\\
			=&\frac12(I\otimes T)(U_{i\alpha}\otimes\xi_i\otimes X_\alpha),
		\end{align*}
		where the operator $T$ is defined in Lemma \ref{eigenvalue of J0} and the eigenvalue of $T$ is $8$ or $-16$. Recall that $\tilde J_0=9I+I\otimes T$ on this subspace. If the eigenvalue of $\tilde J_0$ is 17, then the eigenvalue of $T$ is 8, and thus $-q(q+2)=-15$, which implies $q=3$. If the eigenvalue of $\tilde J_0$ is -7, then the eigenvalue of $T$ is $-16$, and thus $-q(q+2)=-3$, which implies $q=1$.
	\end{proof}
	We now return to the first-order term and prove the estimate stated at the beginning of this subsection.
	\begin{proof}[Proof of Lemma \ref{eigenvalue of J1}]
		By the global formula for $\tilde J_1$ obtained above, we have
		\begin{align*}
			\tilde J_1=\sum_{i=1}^3X_i\otimes\tilde J_{1i}\otimes I+\sum_{i=4}^7\sqrt2X_i\otimes\tilde J_{1i}\otimes I
		\end{align*}
		where
		\begin{align*}
			\tilde J_{1i}(\xi_j)=\left\{
			\begin{array}{ll}
				\sum_{k=1}^3\langle X_j,[X_i,X_k]\rangle_{alg}\xi_k,&1\le i\le 3,\\
				2\sqrt2\sum_{k=1}^3\langle X_j,[X_i,X_k]\rangle_{alg}\xi_k+\sqrt2\sum_{k=4}^7\langle X_j,[X_i,X_k]\rangle_{alg}\xi_k,&4\le i\le7.
			\end{array}
			\right.
		\end{align*}
		For $\tilde B=\sum_{j=1}^7\sum_{\alpha=8}^{10}U_{j\alpha}\otimes\xi_j\otimes X_\alpha$, the matrix identities in Appendix \ref{app fiber calculations} give
		\begin{align*}
			\sum_{i=1}^7\|(I\otimes\tilde J_{1i}\otimes I)(\tilde B)\|^2=24\|P_1(\tilde B)\|^2+12\|P_2(\tilde B)\|^2.
		\end{align*}
		Since each $X_i$ is skew-adjoint on $L^2(G)$, integration by parts and the Cauchy--Schwarz inequality give
		\begin{align*}
			|\langle\tilde J_1(\tilde B),\tilde B\rangle|&\le\left(\sum_{i=1}^3\|(X_i\otimes I\otimes I)(\tilde B)\|^2+\sum_{i=4}^7\|(\sqrt2X_i\otimes I\otimes I)(\tilde B)\|^2\right)^{\frac12}\left(\sum_{i=1}^7\|(I\otimes\tilde J_{1i}\otimes I)(\tilde B)\|^2\right)^{\frac12}\\
			&=\langle\tilde J_2(\tilde B),\tilde B\rangle^{\frac12}\left(24\|P_1(\tilde B)\|^2+12\|P_2(\tilde B)\|^2\right)^{\frac12}\\
			&\le\frac12\langle\tilde J_2(\tilde B),B\rangle+12\|P_1(\tilde B)\|^2+6\|P_2(\tilde B)\|^2.
		\end{align*}
		Consequently,
		\begin{align*}
			\langle\tilde J(B),B\rangle\ge\frac12\langle\tilde J_2(B),B\rangle+\langle\tilde J_0(B),B\rangle-12\|P_1(B)\|^2-6\|P_2(B)\|^2,
		\end{align*}
		which proves the result.
	\end{proof}
	Combining the preceding estimate with the Casimir eigenvalues and the restrictions imposed by $K$-equivariance, we obtain the following reduction.
	\begin{proposition}\label{finite reduction}
		If
		\begin{align*}
			(a,b)\notin\{(1,0),(0,1),(1,1),(2,0)\},
		\end{align*}
		then $\tilde J$ is positive definite on $(\mathcal{U}_{a,b}\otimes\mathfrak{m}^*\otimes\mathfrak{k})_K$.
	\end{proposition}
	\begin{proof}
		By Lemma \ref{eigenvalue of J2}, we have $p=r+s$ and $q=a+r-s$, where $0\le r\le b$ and $0\le s\le a$. Hence $p=a+2r-q$, and therefore
		\begin{align*}
			|a-q|\le p\le\min\{a+q,a+2b-q\},\qquad p\equiv a-q\pmod 2.
		\end{align*}
		Combining Lemmas \ref{eigenvalue of J0}, \ref{eigenvalue of J2}, \ref{lambda and q}, and \ref{eigenvalue of J1}, we find that the possible eigenvalues of $\frac12\tilde J_2+\tilde J_0-12P_1-6P_2$ are of the following three forms:
		\begin{align*}
			\lambda_1&=a^2+2ab+2b^2+4a+6b-\frac12p(p+2)-16,\qquad |a-1|\le p\le\min\{a+1,a+2b-1\},\quad p\equiv a-1\pmod 2,\\
			\lambda_2&=a^2+2ab+2b^2+4a+6b-\frac12p(p+2)-8,\qquad |a-2|\le p\le\min\{a+2,a+2b-2\},\quad p\equiv a\pmod 2,\\
			\lambda_3&=a^2+2ab+2b^2+4a+6b-\frac12p(p+2)-4,\qquad |a-3|\le p\le\min\{a+3,a+2b-3\},\quad p\equiv a-1\pmod 2.
		\end{align*}
		Here $\lambda_1$, $\lambda_2$, and $\lambda_3$ correspond respectively to $(\lambda,q)=(-7,1),(12,2),(17,3)$ in Lemma \ref{lambda and q}.
		
		$\mathbf{Claim:\ \lambda_2>0.}$ If $b=0$, then $p=a-2$ and $a\ge2$. Substituting $p=a-2$ into $\lambda_2$, we obtain
		\begin{align*}
			\lambda_2=\frac12a^2+5a-8\ge4.
		\end{align*}
		If $b=1$, then $p\le a$, $a\ge1$, and
		\begin{align*}
			\lambda_2\ge\frac12a^2+5a>0.
		\end{align*}
		If $b\ge2$, then $p\le a+2$ and
		\begin{align*}
			a^2+2ab+4a-\frac12p(p+2)\ge-4.
		\end{align*}
		Therefore,
		\begin{align*}
			\lambda_2\ge2b^2+6b-12\ge8.
		\end{align*}
		
		$\mathbf{Claim:\ \lambda_3>0.}$ If $b=0$, then $p=a-3$ and $a\ge3$. Substituting $p=a-3$ into $\lambda_3$, we obtain
		\begin{align*}
			\lambda_3=\frac12a^2+6a-\frac{11}{2}\ge17.
		\end{align*}
		If $b=1$, then $p\le a-1$, $a\ge2$, and
		\begin{align*}
			\lambda_3\ge\frac12a^2+6a+\frac92>0.
		\end{align*}
		If $b=2$, then $p\le a+1$, $a\ge1$, and
		\begin{align*}
			\lambda_3\ge\frac12a^2+6a+\frac{29}{2}>0.
		\end{align*}
		If $b\ge3$, then $p\le a+3$ and
		\begin{align*}
			a^2+2ab+4a-\frac12p(p+2)\ge-\frac{15}{2}.
		\end{align*}
		Hence
		\begin{align*}
			\lambda_3\ge2b^2+6b-\frac{23}{2}\ge\frac{49}{2}.
		\end{align*}
		
		$\mathbf{Estimate\ of\ \lambda_1.}$ If $b=0$, then $p=a-1$ and $a\ge1$, so
		\begin{align*}
			\lambda_1=\frac12a^2+4a-\frac{31}{2}.
		\end{align*}
		Thus $\lambda_1=-11$ for $a=1$, $\lambda_1=-\frac{11}{2}$ for $a=2$, and $\lambda_1>0$ for $a\ge3$.
		
		If $b=1$, then $p\le a+1$ and
		\begin{align*}
			\lambda_1\ge\frac12a^2+4a-\frac{19}{2}.
		\end{align*}
		For $a=0$, we have $p=1$ and $\lambda_1=-\frac{19}{2}$. For $a=1$, we have $p=0$ or $p=2$, and the corresponding values of $\lambda_1$ are $-1$ and $-5$, respectively. Moreover, $\lambda_1>0$ for $a\ge2$.
		
		If $b\ge2$, then $p\le a+1$ and
		\begin{align*}
			a^2+2ab+4a-\frac12p(p+2)\ge-\frac32.
		\end{align*}
		Therefore,
		\begin{align*}
			\lambda_1\ge2b^2+6b-\frac{35}{2}\ge\frac52.
		\end{align*}
		The preceding estimates show that $\frac12\tilde J_2+\tilde J_0-12P_1-6P_2$ is positive definite whenever
		\begin{align*}
			(a,b)\notin\{(1,0),(0,1),(1,1),(2,0)\}.
		\end{align*}
		Lemma \ref{eigenvalue of J1} now implies that $\tilde J$ is positive definite on every such component.
	\end{proof}
	\subsection{Calculation of eigenvalues on 4 invariant subspaces}
	Proposition \ref{finite reduction} reduces the determination of the nonpositive spectrum to the representations $V_{1,0}$, $V_{0,1}$, $V_{2,0}$, and $V_{1,1}$. We now determine their contributions. The explicit bases, derivative formulas, and matrix computations are collected in Appendix \ref{appendix low modes}.
	
	To simplify the explicit calculations, we now complexify the relevant spaces and operators.
	\begin{align*}
		\mathcal{U}_{a,b}^{\mathbb{C}}&:=\mathcal{U}_{a,b}\otimes_{\mathbb{R}}\mathbb{C}=\operatorname{span}_{\mathbb{C}}\{U^{st}\mid1\le s,t\le l_{a,b}\},\\
		(\mathfrak{m}^*\otimes\mathfrak{k})_{\mathbb{C}}&:=(\mathfrak{m}^*\otimes_{\mathbb{R}}\mathfrak{k})\otimes_{\mathbb{R}}\mathbb{C}.
	\end{align*}
	Then we have
	\begin{align*}
		((\mathcal{U}_{a,b}\otimes_{\mathbb{R}}\mathfrak{m}^*\otimes_{\mathbb{R}}\mathfrak{k})_K)\otimes_{\mathbb{R}}\mathbb{C}\cong(\mathcal{U}_{a,b}^{\mathbb{C}}\otimes_{\mathbb{C}}(\mathfrak{m}^*\otimes\mathfrak{k})_{\mathbb{C}})_K.
	\end{align*}
	Let $\tilde J^{\mathbb{C}}$ and $\tilde\delta^\nabla_{\mathbb{C}}$ denote the complexifications of $\tilde J$ and $\tilde\delta^\nabla$ respectively. Thus,
	\begin{align*}
		&\tilde J^{\mathbb{C}}(B_1+iB_2)=\tilde J(B_1)+i\tilde J(B_2),\\
		&\tilde\delta^\nabla_{\mathbb{C}}(B_1+iB_2)=\tilde\delta^\nabla(B_1)+i\tilde\delta^\nabla(B_2).
	\end{align*}
	Since both operators are $G$-equivariant, $\tilde J^{\mathbb{C}}$ preserves each $(a,b)$-component, while $\tilde\delta^\nabla_{\mathbb{C}}$ maps it into the corresponding component of the complexified zero-form space.
	\begin{proposition}\label{eigenvalue of J^C}
		For every $\lambda\in\mathbb{R}$, we have
		\begin{align*}
			\ker(\tilde J^{\mathbb{C}}-\lambda I)\cap \ker\tilde\delta^\nabla_{\mathbb{C}}=(\ker(\tilde J-\lambda I)\cap \ker\tilde\delta^\nabla)\otimes_{\mathbb{R}}\mathbb{C}.
		\end{align*}
		Consequently,
		\begin{align*}
			\dim_{\mathbb{C}}(\ker(\tilde J^{\mathbb{C}}-\lambda I)\cap \ker\tilde\delta^\nabla_{\mathbb{C}})=\dim_{\mathbb{R}}(\ker(\tilde J-\lambda I)\cap \ker\tilde\delta^\nabla).
		\end{align*}
		Therefore
		\begin{align*}
			i(\nabla)=\sum_{a,b\ge0}\sum_{\lambda<0}\dim_{\mathbb{C}}(\ker(\tilde J^{\mathbb{C}}-\lambda I)\cap(\mathcal{U}_{a,b}^{\mathbb{C}}\otimes(\mathfrak{m}^*\otimes\mathfrak{k})_{\mathbb{C}})_K\cap \ker\tilde\delta^\nabla_{\mathbb{C}}),
		\end{align*}
		and similarly,
		\begin{align*}
			n(\nabla)=\sum_{a,b\ge0}\sum_{\lambda=0}\dim_{\mathbb{C}}(\ker(\tilde J^{\mathbb{C}}-\lambda I)\cap(\mathcal{U}_{a,b}^{\mathbb{C}}\otimes(\mathfrak{m}^*\otimes\mathfrak{k})_{\mathbb{C}})_K\cap \ker\tilde\delta^\nabla_{\mathbb{C}}).
		\end{align*}
	\end{proposition}
	\begin{proof}
		By the definition of the complexified operators,
		\begin{align*}
			\tilde\delta^\nabla_{\mathbb{C}}(B+iB')=0
		\end{align*}
		if and only if
		\begin{align*}
			\tilde\delta^\nabla B=0,\quad\tilde\delta^\nabla B'=0.
		\end{align*}
		Thus
		\begin{align*}
			\ker\tilde\delta^\nabla_{\mathbb{C}}=\ker\tilde\delta^\nabla\otimes_{\mathbb{R}}\mathbb{C}.
		\end{align*}
		Since $\tilde J$ is self-adjoint, every eigenvalue of $\tilde J^{\mathbb{C}}$ is real. For $\lambda\in\mathbb{R}$,
		\begin{align*}
			\tilde J^{\mathbb{C}}(B+iB')=\lambda(B+iB')
		\end{align*}
		if and only if
		\begin{align*}
			\tilde J(B)=\lambda B,\quad\tilde J(B')=\lambda B'.
		\end{align*}
		Therefore
		\begin{align*}
			\ker(\tilde J^{\mathbb{C}}-\lambda I)=\ker(\tilde J-\lambda I)\otimes_{\mathbb{R}}\mathbb{C}.
		\end{align*}
		Since complexification commutes with intersections of real subspaces, combining the last 2 identities gives
		\begin{align*}
			\ker(\tilde J^{\mathbb{C}}-\lambda I)\cap \ker\tilde\delta^\nabla_{\mathbb{C}}=(\ker(\tilde J-\lambda I)\cap \ker\tilde \delta^\nabla)\otimes_{\mathbb{R}}\mathbb{C}.
		\end{align*}
		
		If $E$ is a finite-dimensional real vector space, then
		\begin{align*}
			\dim_{\mathbb{C}}(E\otimes_{\mathbb{R}}\mathbb{C})=\dim_{\mathbb{R}}E.
		\end{align*}
		Hence the complex dimension of every constrained eigensubspace of $\tilde J^{\mathbb{C}}$ equals the real dimension of the corresponding constrained eigensubspace of $\tilde J$. Since $\tilde J$ represents a self-adjoint elliptic operator on the compact manifold $G/K$, its spectrum is discrete and bounded below, and every eigenspace is finite-dimensional. Each constrained eigenspace is $G$-invariant, so its Peter--Weyl decomposition is the direct sum of its intersections with the $(a,b)$-components. The formulas for $i(\nabla)$ and $n(\nabla)$ now follow by summing their dimensions.
	\end{proof}
	The contributions of the 4 remaining representations are summarized below. In the representation calculations, we write $\tilde\delta$ for the complexified codifferential $\tilde\delta^\nabla_{\mathbb C}$. The calculations in Appendix \ref{appendix low modes} yield the following proposition.
	
	\begin{proposition}\label{low representation estimates}
		(i)\ We have $$\mathcal{U}_{1,0}^{\mathbb{C}}\otimes(\mathfrak{m}^*\otimes\mathfrak{k})_{\mathbb{C}})_K\subset\ker\tilde\delta$$ and
		\begin{align}
			\tilde J^{\mathbb{C}}|_{\mathcal{U}_{1,0}^{\mathbb{C}}\otimes(\mathfrak{m}^*\otimes\mathfrak{k})_{\mathbb{C}})_K}=-3\operatorname{Id}.
		\end{align}
		Moreover,
		\begin{align}
			\dim_{\mathbb{C}}\mathcal{U}_{1,0}^{\mathbb{C}}\otimes(\mathfrak{m}^*\otimes\mathfrak{k})_{\mathbb{C}})_K=8.
		\end{align}
		(ii)\ We have $$\mathcal{U}_{0,1}^{\mathbb{C}}\otimes(\mathfrak{m}^*\otimes\mathfrak{k})_{\mathbb{C}})_K\subset\ker\tilde\delta$$ and
		\begin{align}
			\tilde J^{\mathbb{C}}|_{\mathcal{U}_{0,1}^{\mathbb{C}}\otimes(\mathfrak{m}^*\otimes\mathfrak{k})_{\mathbb{C}})_K}=0.
		\end{align}
		Moreover,
		\begin{align}
			\dim_{\mathbb{C}}\mathcal{U}_{0,1}^{\mathbb{C}}\otimes(\mathfrak{m}^*\otimes\mathfrak{k})_{\mathbb{C}})_K=20.
		\end{align}
		(iii)\ We have
		\begin{align}
			\tilde J^{\mathbb{C}}|_{\mathcal{U}_{2,0}^{\mathbb{C}}\otimes(\mathfrak{m}^*\otimes\mathfrak{k})_{\mathbb{C}})_K\cap\tilde\delta}\ge(8-4\sqrt2)\operatorname{Id}.
		\end{align}
		(iv)\ We have
		\begin{align}
			\tilde J^{\mathbb{C}}|_{\mathcal{U}_{1,1}^{\mathbb{C}}\otimes(\mathfrak{m}^*\otimes\mathfrak{k})_{\mathbb{C}})_K}\ge(14-\sqrt{89})\operatorname{Id}.
		\end{align}
	\end{proposition}
	We can now prove Theorem \ref{thm1}.
	\begin{proof}[Proof of Theorem \ref{thm1}]
		By Proposition \ref{finite reduction}, only the 4 representations considered in Proposition \ref{low representation estimates} can contribute to the nonpositive spectrum. Among these, only $V_{1,0}$ contributes negative eigenvalues: the eigenvalue is $-3$, with complex multiplicity $8$. Proposition \ref{eigenvalue of J^C} identifies this multiplicity with the real multiplicity of the corresponding eigenspace of $\mathscr S^\nabla|_{\ker\delta^\nabla}$. Therefore $-3$ is the only negative eigenvalue and $i(\nabla)=8$.
	\end{proof}
	
	The same propositions show that $n(\nabla)=20$ and
	\begin{align*}
		\ker\tilde J^{\mathbb{C}}\cap\ker\tilde\delta=(\mathcal{U}_{0,1}^{\mathbb{C}}\otimes(\mathfrak{m}^*\otimes\mathfrak{k})_{\mathbb{C}})_K.
	\end{align*}
	This proves the nullity assertion in Theorem \ref{thm2}. It remains to identify the kernel geometrically, which we do in the next subsection.
	\subsection{The geometry of the Jacobi kernel}\label{section nullity}
	The spectral calculation identifies the Jacobi kernel with the real form of the component corresponding to $V_{0,1}$. We now express its elements as contractions of the curvature with explicit horizontal vector fields on $S^7$.
	
	By construction, $V_{0,1}=\operatorname{span}_{\mathbb{C}}\{E_1,\dots,E_5\}\subset\Lambda^2\mathbb{C}^4$. For the following calculation, we use the orthogonal basis
	\begin{align*}
		E_1=e_1\wedge e_2,\quad E_2=e_1\wedge e_3-e_2\wedge e_4,\quad E_3=e_1\wedge e_4,\quad E_4=e_2\wedge e_3,\quad E_5=e_3\wedge e_4.
	\end{align*}
	Recall that $\pi_{0,1}$ is the restriction to $V_{0,1}$ of the induced action of $G$ on $\Lambda^2\mathbb{C}^4$. Define
	\begin{align*}
		U^{st}_{0,1}(g)=\langle\pi_{0,1}(g)E_s,E_t\rangle
	\end{align*}
	for $1\le s,t\le5$. These functions span the same matrix-coefficient space as those defined in \eqref{U^st}. 
	\begin{proof}[Proof of Theorem \ref{thm2}]
		The equality $n(\nabla)=20$ was established above. It remains to identify the kernel geometrically. We use the equivariant vectors $\tilde\Phi_{it}$ computed in Appendix \ref{app V01}.
		Using Table \ref{[X_i,X_j]}, we obtain
		\begin{align*}
			F_i=\sum_{j=1}^7\xi_j\otimes[X_j,X_{k(i)}]_{\mathfrak{k}}
		\end{align*}
		for $i=1,\dots,4$, where $k(1)=5$, $k(2)=6$, $k(3)=7$, and $k(4)=4$. Let $\tilde\phi_{it;r}$ and $\tilde\phi_{it;i}$ denote the real one-forms corresponding to the real and imaginary parts of $\tilde\Phi_{it}$ under the identifications in Section \ref{section Jacobi}. By the preceding spectral analysis, the Jacobi kernel on the Coulomb slice is
		\begin{align*}
			\operatorname{span}_{\mathbb{R}}\{\tilde\phi_{it;r},\tilde\phi_{it;i}\mid1\le i\le4,\ 1\le t\le5\},
		\end{align*}
		where
		\begin{align*}
			\tilde\phi_{it;r}=i_{\operatorname{Re}(V_{it})}F,\qquad \tilde\phi_{it;i}=i_{\operatorname{Im}(V_{it})}F,
		\end{align*}
		and
		\begin{align*}
			&V_{1t}(g)=\tau_{g*}\pi_{K*}(\langle\pi_{0,1}(g)E_1,E_t\rangle(X_5+iX_4)-\langle\pi_{0,1}(g)E_3,E_t\rangle(X_7+iX_6)),\\
			&V_{2t}(g)=\tau_{g*}\pi_{K*}(\langle\pi_{0,1}(g)E_1,E_t\rangle(X_6+iX_7)+\langle\pi_{0,1}(g)E_3,E_t\rangle(X_4+iX_5)),\\
			&V_{3t}(g)=\tau_{g*}\pi_{K*}(\langle\pi_{0,1}(g)E_4,E_t\rangle(X_5+iX_4)+\langle\pi_{0,1}(g)E_5,E_t\rangle(X_7+iX_6)),\\
			&V_{4t}(g)=\tau_{g*}\pi_{K*}(\langle\pi_{0,1}(g)E_4,E_t\rangle(X_6+iX_7)-\langle\pi_{0,1}(g)E_5,E_t\rangle(X_4+iX_5)).
		\end{align*}
		
		We now express these vector fields in quaternionic coordinates. Fix $x_0=(1,0)\in S^7\subset\mathbb{H}^2$. Under the diffeomorphism $G/K\to S^7$, $gK\mapsto gx_0$, the tangent vector $\tau_{g*}\pi_{K*}X$ corresponds to $gXx_0$ for $X\in\mathfrak{m}$. Given $x\in S^7$, choose $g=(x,y)\in \operatorname{Sp}(2)$ and write $x=x_1+jx_2$ and $y=y_1+jy_2$, where $x_1,x_2,y_1,y_2\in\mathbb{C}^2$. Then
		\begin{align*}
			\chi(g)=\begin{pmatrix}x_1&y_1&-\bar x_2&-\bar y_2\\x_2&y_2&\bar x_1&\bar y_1\end{pmatrix},
		\end{align*}
		and
		\begin{align*}
			gX_4x_0=-\frac{y}{\sqrt2},\qquad gX_5x_0=\frac{yi}{\sqrt2},\qquad gX_6x_0=\frac{yj}{\sqrt2},\qquad gX_7x_0=\frac{yk}{\sqrt2}.
		\end{align*}
		
		We give the calculation for $V_{11}$ explicitly, showing how the dependence on the choice of $y$ is eliminated. The exterior-square action gives
		\begin{align*}
			\langle\pi_{0,1}(g)E_1,E_1\rangle&=\langle\chi(g)e_1\wedge\chi(g)e_2,e_1\wedge e_2\rangle=\det(x_1,y_1),\\
			\langle\pi_{0,1}(g)E_3,E_1\rangle&=\langle\chi(g)e_1\wedge\chi(g)e_4,e_1\wedge e_2\rangle=-\det(x_1,\bar y_2).
		\end{align*}
		Substituting these identities into the expression for $V_{11}$ yields
		\begin{align*}
			V_{11}(x)=\frac{\det(x_1,y_1)}{\sqrt2}(yi-iy)+\frac{\det(x_1,\bar y_2)}{\sqrt2}(yk+iyj).
		\end{align*}
		This identity is understood in $T_xS^7\otimes_{\mathbb{R}}\mathbb{C}$: $yi$, $yj$, and $yk$ denote right quaternionic multiplication, whereas the determinants and the other occurrences of $i$ act as complex scalars.
		
		Set $\hat x=\begin{pmatrix}x_1\\x_2\end{pmatrix}\in\mathbb{C}^4$ and
		\begin{align*}
			H_1=\begin{pmatrix}0&1&0&0\\-1&0&0&0\end{pmatrix}.
		\end{align*}
		Direct quaternionic multiplication gives
		\begin{align*}
			yy^\dagger H_1\overline{\hat x}=-\operatorname{Re}(\det(x_1,y_1))y+\operatorname{Im}(\det(x_1,y_1))yi+\operatorname{Re}(\det(x_1,\bar y_2))yj+\operatorname{Im}(\det(x_1,\bar y_2))yk.
		\end{align*}
		Comparing this identity with the expression for $V_{11}$ and using $yy^\dagger=I-xx^\dagger$, we obtain
		\begin{align*}
			\operatorname{Re}(V_{11}(x))&=-\frac1{\sqrt2}(I-xx^\dagger)H_1\overline{\hat x}i,\\
			\operatorname{Im}(V_{11}(x))&=\frac1{\sqrt2}(I-xx^\dagger)H_1\overline{\hat x}.
		\end{align*}
		These expressions depend only on $x$ and therefore define global horizontal vector fields on $S^7$.
		
		To state the remaining formulas uniformly, define
		\begin{align*}
			H_2&=\begin{pmatrix}-j&0&1&0\\0&j&0&-1\end{pmatrix},\qquad H_3=\begin{pmatrix}0&0&0&1\\-j&0&0&0\end{pmatrix},\\
			H_4&=\begin{pmatrix}0&-j&0&0\\0&0&1&0\end{pmatrix},\qquad H_5=\begin{pmatrix}0&0&0&j\\0&0&-j&0\end{pmatrix},
		\end{align*}
		and
		\begin{align*}
			L_1&=\begin{pmatrix}0&0&0&-1\\0&0&1&0\end{pmatrix},\qquad L_2=\begin{pmatrix}1&0&j&0\\0&-1&0&-j\end{pmatrix},\\
			L_3&=\begin{pmatrix}0&1&0&0\\0&0&j&0\end{pmatrix},\qquad L_4=\begin{pmatrix}0&0&0&j\\1&0&0&0\end{pmatrix},\\
			L_5&=\begin{pmatrix}0&j&0&0\\-j&0&0&0\end{pmatrix}.
		\end{align*}
		Applying the same calculation to the remaining matrix coefficients gives
		\begin{align*}
			&\operatorname{Re}(V_{1t}(x))=-\frac1{\sqrt2}(I-xx^\dagger)H_t\overline{\hat x}i,\quad \operatorname{Im}(V_{1t}(x))=\frac1{\sqrt2}(I-xx^\dagger)H_t\overline{\hat x},\\
			&\operatorname{Re}(V_{2t}(x))=-\frac1{\sqrt2}(I-xx^\dagger)H_t\overline{\hat x}j,\quad \operatorname{Im}(V_{2t}(x))=-\frac1{\sqrt2}(I-xx^\dagger)H_t\overline{\hat x}k,\\
			&\operatorname{Re}(V_{3t}(x))=\frac1{\sqrt2}(I-xx^\dagger)L_t\hat xi,\quad \operatorname{Im}(V_{3t}(x))=-\frac1{\sqrt2}(I-xx^\dagger)L_t\hat x,\\
			&\operatorname{Re}(V_{4t}(x))=\frac1{\sqrt2}(I-xx^\dagger)L_t\hat xj,\quad \operatorname{Im}(V_{4t}(x))=\frac1{\sqrt2}(I-xx^\dagger)L_t\hat xk
		\end{align*}
		for $1\le t\le5$. Since
		\begin{align*}
			\operatorname{span}_{\mathbb{R}}\{H_t\bar{\hat x}q,L_t{\hat x}q\mid q\in\mathbb{H}, 1\le t\le5\}=\{(M_0x+M_1xi+M_2xj+M_3xk)\mid M_i^\dagger=M_i,\operatorname{tr}(M_i)=0\textrm{ for }0\le i\le3\},
		\end{align*}
		the claimed description of the kernel follows by applying the projection $I-xx^\dagger$ and contracting with $F$.
	\end{proof}
	
	\appendix
	\section{Fiberwise calculations}\label{app fiber calculations}\label{appendix geometric calculations}
	This appendix collects the bracket identities and matrix calculations used in the spectral estimates. We first record the relevant Lie brackets.
	\begin{table}[H]
		\centering
		\captionsetup{labelsep=space,skip=6pt}
		\caption{$\mathbf{[X_i,X_j]}$}
		\label{[X_i,X_j]}
		\(\begin{array}{|c|c|c|c|c|c|c|c|}
			\hline
			[X_i,X_j]&j=1&j=2&j=3&j=4&j=5&j=6&j=7\\
			\hline
			i=1&0&2X_3&-2X_2&X_5&-X_4&X_7&-X_6\\
			\hline
			i=2&-2X_3&0&2X_1&X_6&-X_7&-X_4&X_5\\
			\hline
			i=3&2X_2&-2X_1&0&X_7&X_6&-X_5&-X_4\\
			\hline
			i=4&-X_5&-X_6&-X_7&0&X_1-X_8&X_2-X_9&X_3-X_{10}\\
			\hline
			i=5&X_4&X_7&-X_6&-X_1+X_8&0&X_3+X_{10}&-X_2-X_9\\
			\hline
			i=6&-X_7&X_4&X_5&-X_2+X_9&-X_3-X_{10}&0&X_1+X_8\\
			\hline
			i=7&X_6&-X_5&X_4&-X_3+X_{10}&X_2+X_9&-X_1-X_8&0\\
			\hline
		\end{array}\)
	\end{table}
	\begin{table}[H]
		\centering
		\captionsetup{labelsep=space,skip=6pt}
		\caption{$\mathbf{[X_\alpha,X_\beta]}$}
		\label{[X_alpha,X_beta]}
		\(\begin{array}{|c|c|c|c|}
			\hline
			[X_\alpha,X_\beta]&\beta=8&\beta=9&\beta=10\\
			\hline
			\alpha=8&0&2X_{10}&-2X_9\\
			\hline
			\alpha=9&-2X_{10}&0&2X_8\\
			\hline
			\alpha=10&2X_9&-2X_8&0\\
			\hline
		\end{array}\)
	\end{table}
	For the operator $T$ defined in the proof of Lemma \ref{eigenvalue of J0}, set
	\begin{align*}
		T_1=\begin{pmatrix}
			0&-8&8\\-8&0&8\\8&8&0
		\end{pmatrix},\quad
		T_2=\begin{pmatrix}
			0&-8&-8\\-8&0&-8\\-8&-8&0
		\end{pmatrix}.
	\end{align*}
	Then $\det(\lambda I-T_1)=\det(\lambda I-T_2)=(\lambda-8)^2(\lambda+16)$. With respect to the ordered basis
	\begin{align*}
		(\xi_4\otimes X_8,\xi_6\otimes X_{10},\xi_7\otimes X_9,\xi_4\otimes X_9,\xi_7\otimes X_8,\xi_5\otimes X_{10},\xi_4\otimes X_{10},\xi_5\otimes X_9,\xi_6\otimes X_8,\xi_5\otimes X_8,\xi_6\otimes X_9,\xi_7\otimes X_{10})
	\end{align*}
	we have
	\begin{align*}
		T=\operatorname{diag}(T_1,T_1,T_1,T_2).
	\end{align*}
	Moreover, the $-16$-eigenspace of $T$ is spanned by $F_1,\dots,F_4$ defined in \eqref{F1-F4}. This gives the eigenvalues and eigenspaces stated in Lemma \ref{eigenvalue of J0}.
	
	We next record the infinitesimal isotropy action. It vanishes on $\mathfrak m_1^*$, while its action on $\mathfrak m_2^*$ is given in Table \ref{ad^*(X_beta)xi_i}.
	\begin{table}[H]
		\centering
		\captionsetup{labelsep=colon,skip=6pt}
		\caption{$\mathbf{\operatorname{ad}^*(X_\beta)\xi_i}$}
		\label{ad^*(X_beta)xi_i}
		\(\begin{array}{|c|c|c|c|c|}
			\hline
			\operatorname{ad}^*(X_\beta)\xi_i&i=4&i=5&i=6&i=7\\
			\hline
			\beta=8&\xi_5&-\xi_4&-\xi_7&\xi_6\\
			\hline
			\beta=9&\xi_6&\xi_7&-\xi_4&-\xi_5\\
			\hline
			\beta=10&\xi_7&-\xi_6&\xi_5&-\xi_4\\
			\hline
		\end{array}\)
	\end{table}
	Finally, consider the coefficient operators $\tilde J_{1i}$ defined in the proof of Lemma \ref{eigenvalue of J1}. Their actions, computed from Table \ref{[X_i,X_j]}, are listed in Table \ref{J_{1i}(xi_j)}.
	\begin{table}[H]
		\centering
		\captionsetup{labelsep=colon,skip=6pt}
		\caption{$\mathbf{\tilde J_{1i}(\xi_j)}$}
		\label{J_{1i}(xi_j)}
		\(\begin{array}{|c|c|c|c|c|c|c|c|}
			\hline
			\tilde J_{1i}(\xi_j)&i=1&i=2&i=3&i=4&i=5&i=6&i=7\\
			\hline
			j=1&0&2\xi_3&-2\xi_2&\sqrt2\xi_5&-\sqrt2\xi_4&\sqrt2\xi_7&-\sqrt2\xi_6\\
			\hline
			j=2&-2\xi_3&0&2\xi_1&\sqrt2\xi_6&-\sqrt2\xi_7&-\sqrt2\xi_4&\sqrt2\xi_5\\
			\hline
			j=3&2\xi_2&-2\xi_1&0&\sqrt2\xi_7&\sqrt2\xi_6&-\sqrt2\xi_5&-\sqrt2\xi_4\\
			\hline
			j=4&0&0&0&0&2\sqrt2\xi_1&2\sqrt2\xi_2&2\sqrt2\xi_3\\
			\hline
			j=5&0&0&0&-2\sqrt2\xi_1&0&2\sqrt2\xi_3&-2\sqrt2\xi_2\\
			\hline
			j=6&0&0&0&-2\sqrt2\xi_2&-2\sqrt2\xi_3&0&2\sqrt2\xi_1\\
			\hline
			j=7&0&0&0&-2\sqrt2\xi_3&2\sqrt2\xi_2&-2\sqrt2\xi_1&0\\
			\hline
		\end{array}\)
	\end{table}
	The operators $\tilde J_{11},\dots,\tilde J_{17}$ are represented by $7\times7$ matrices satisfying
	\begin{align*}
		&\sum_{i=1}^3\tilde J_{1i}^T\operatorname{diag}(I_3,2I_4)\tilde J_{1i}=\operatorname{diag}(8I_3,0),\\
		&\sum_{i=4}^7\tilde J_{1i}^T\operatorname{diag}(I_3,2I_4)\tilde J_{1i}=\operatorname{diag}(16I_3,24I_4),
	\end{align*}
	where $\operatorname{diag}(I_3,2I_4)$ is the Gram matrix of the basis $\{\xi_1,\dots,\xi_7\}$ with respect to the induced inner product.
	\section{Calculations for the low representations}\label{appendix low modes}
	Recall the definitions of $\{X_i\}$ and $\chi$ in \eqref{X_i} and \eqref{chi}, respectively. A direct calculation gives
	\begingroup
	\allowdisplaybreaks[4]
	\begin{align*}
		&\chi(\exp(tX_1))=\operatorname{diag}(e^{ti},1,e^{-ti},1),\\
		&\chi(\exp(tX_2))=\begin{pmatrix}
			\cos t&0&-\sin t&0\\
			0&1&0&0\\
			\sin t&0&\cos t&0\\
			0&0&0&1
		\end{pmatrix},\\
		&\chi(\exp(tX_3))=\begin{pmatrix}
			\cos t&0&-i\sin t&0\\
			0&1&0&0\\
			-i\sin t&0&\cos t&0\\
			0&0&0&1
		\end{pmatrix},\\
		&\chi(\exp(tX_4))=\begin{pmatrix}
			\cos\frac t{\sqrt2}&\sin\frac t{\sqrt2}&0&0\\
			-\sin\frac t{\sqrt2}&\cos\frac t{\sqrt2}&0&0\\
			0&0&\cos\frac t{\sqrt2}&\sin\frac t{\sqrt2}\\
			0&0&-\sin\frac t{\sqrt2}&\cos\frac t{\sqrt2}
		\end{pmatrix},\\
		&\chi(\exp(tX_5))=\begin{pmatrix}
			\cos\frac t{\sqrt2}&i\sin\frac t{\sqrt2}&0&0\\
			i\sin\frac t{\sqrt2}&\cos\frac t{\sqrt2}&0&0\\
			0&0&\cos\frac t{\sqrt2}&-i\sin\frac t{\sqrt2}\\
			0&0&-i\sin\frac t{\sqrt2}&\cos\frac t{\sqrt2}
		\end{pmatrix},\\
		&\chi(\exp(tX_6))=\begin{pmatrix}
			\cos\frac t{\sqrt2}&0&0&-\sin\frac t{\sqrt2}\\
			0&\cos\frac t{\sqrt2}&-\sin\frac t{\sqrt2}&0\\
			0&\sin\frac t{\sqrt2}&\cos\frac t{\sqrt2}&0\\
			\sin\frac t{\sqrt2}&0&0&\cos\frac t{\sqrt2}
		\end{pmatrix},\\
		&\chi(\exp(tX_7))=\begin{pmatrix}
			\cos\frac t{\sqrt2}&0&0&-i\sin\frac t{\sqrt2}\\
			0&\cos\frac t{\sqrt2}&-i\sin\frac t{\sqrt2}&0\\
			0&-i\sin\frac t{\sqrt2}&\cos\frac t{\sqrt2}&0\\
			-i\sin\frac t{\sqrt2}&0&0&\cos\frac t{\sqrt2}
		\end{pmatrix}\\
		&\chi(\exp(tX_8))=\operatorname{diag}(1,e^{it},1,e^{-it}),\\
		&\chi(\exp(tX_9))=\begin{pmatrix}
			1&0&0&0\\
			0&\cos t&0&-\sin t\\
			0&0&1&0\\
			0&\sin t&0&\cos t
		\end{pmatrix},\\
		&\chi(\exp(tX_{10}))=\begin{pmatrix}
			1&0&0&0\\
			0&\cos t&0&-i\sin t\\
			0&0&1&0\\
			0&-i\sin t&0&\cos t
		\end{pmatrix}.
	\end{align*}
	\endgroup
	\subsection{The representation $V_{1,0}$}\label{app V10}
	By construction, $V_{1,0}=\mathbb{C}^4$ and $\pi_{1,0}(g)v=\chi(g)v$ for any $g\in G$ and $v\in V_{1,0}$. Let $\langle v,w\rangle:=v^T\bar w$ and define $U^{st}_{1,0}(g)=\langle\chi(g)e_s,e_t\rangle$ for $1\le s,t\le4$. By direct calculation, we have
	\begingroup
	\allowdisplaybreaks[4]
	\begin{alignat*}{2}
		X_1(U^{st}_{1,0})&=\left\{\begin{array}{ll}iU^{1t}_{1,0},&s=1,\\-iU^{3t}_{1,0},&s=3,\\0,&otherwise,\end{array}\right.\qquad&X_6(U^{st}_{1,0})&=\left\{\begin{array}{ll}\frac1{\sqrt2}U^{4t}_{1,0},&s=1,\\\frac1{\sqrt2}U^{3t}_{1,0},&s=2,\\-\frac1{\sqrt2}U^{2t}_{1,0},&s=3,\\-\frac1{\sqrt2}U^{1t}_{1,0},&s=4,\end{array}\right.\\
		X_2(U^{st}_{1,0})&=\left\{\begin{array}{ll}U^{3t}_{1,0},&s=1,\\-U^{1t}_{1,0},&s=3,\\0,&otherwise,\end{array}\right.\qquad&X_7(U^{st}_{1,0})&=\left\{\begin{array}{ll}-\frac i{\sqrt2}U^{4t}_{1,0},&s=1,\\-\frac i{\sqrt2}U^{3t}_{1,0},&s=2,\\-\frac i{\sqrt2}U^{2t}_{1,0},&s=3,\\-\frac i{\sqrt2}U^{1t}_{1,0},&s=4,\end{array}\right.\\
		X_3(U^{st}_{1,0})&=\left\{\begin{array}{ll}-iU^{3t}_{1,0},&s=1,\\-iU^{1t}_{1,0},&s=3,\\0,&otherwise,\end{array}\right.\qquad&X_8(U^{st}_{1,0})&=\left\{\begin{array}{ll}iU^{st}_{1,0},&s=2,\\-iU^{st}_{1,0},&s=4,\\0,&otherwise,\end{array}\right.\\
		X_4(U^{st}_{1,0})&=\left\{\begin{array}{ll}-\frac1{\sqrt2}U^{2t}_{1,0},&s=1,\\\frac1{\sqrt2}U^{1t}_{1,0},&s=2,\\-\frac1{\sqrt2}U^{4t}_{1,0},&s=3,\\\frac1{\sqrt2}U^{3t}_{1,0},&s=4,\end{array}\right.\qquad&X_9(U^{st}_{1,0})&=\left\{\begin{array}{ll}U^{4t}_{1,0},&s=2,\\-U^{2t}_{1,0},&s=4,\\0,&otherwise,\end{array}\right.\\
		X_5(U^{st}_{1,0})&=\left\{\begin{array}{ll}\frac i{\sqrt2}U^{2t}_{1,0},&s=1,\\\frac i{\sqrt2}U^{1t}_{1,0},&s=2,\\-\frac i{\sqrt2}U^{4t}_{1,0},&s=3,\\-\frac i{\sqrt2}U^{3t}_{1,0},&s=4,\end{array}\right.\qquad&X_{10}(U^{st}_{1,0})&=\left\{\begin{array}{ll}-iU^{4t}_{1,0},&s=2,\\-iU^{2t}_{1,0},&s=4,\\0,&otherwise.\end{array}\right.
	\end{alignat*}
	\endgroup
	and
	\begin{align}\label{X_alpha^2 for V_{1,0}}
		X_\alpha^2(U^{st}_{1,0})=\left\{
		\begin{array}{ll}
			-U^{st}_{1,0},&s=2\ or\ s=4,\\
			0,&otherwise,
		\end{array}
		\right.
	\end{align}
	for any $8\le\alpha\le10$.
	
	By Lemma \ref{K-equivariant equivalence} (ii) and (iii), the 2 commuting self-adjoint operators
	\begin{align*}
		-C_K\otimes I\otimes I,\quad I\otimes\sum_{\beta=8}^{10}(\operatorname{ad}^*(X_\beta)\otimes I-I\otimes \operatorname{ad}(X_\beta))^2
	\end{align*}
	preserve $(\mathcal{U}_{1,0}^{\mathbb{C}}\otimes(\mathfrak{m}^*\otimes\mathfrak{k})_{\mathbb{C}})_K$ and coincide on this subspace. By \eqref{X_alpha^2 for V_{1,0}} and the computation in the proof of Lemma \ref{lambda and q}, the spectra of the above 2 operators are $\{0,-3\}$ and $\{-8,-15,-3\}$, respectively. Hence only the common eigenvalue $-3$ can occur, and
	\begin{align*}
		(\mathcal{U}_{1,0}^{\mathbb{C}}\otimes(\mathfrak{m}^*\otimes\mathfrak{k})_{\mathbb{C}})_K\subset \operatorname{span}_{\mathbb{C}}\{U^{2t}_{1,0},U^{4t}_{1,0}\mid1\le t\le4\}\otimes \operatorname{span}_{\mathbb{C}}\{F_1,F_2,F_3,F_4\}.
	\end{align*}
	The eigenvalue $-3$ of the second operator corresponds to $(\lambda,q)=(-7,1)$ by Lemma \ref{lambda and q}. Since the branching rule for $V_{1,0}$ gives only $(p,q)=(0,1)$ and $(1,0)$, we have $p=0$. Therefore, $U^{2t}_{1,0},U^{4t}_{1,0}\in\mathcal{U}_{1,0}(0,1)\otimes_{\mathbb{R}}\mathbb{C}$.
	
	Applying Lemma \ref{K-equivariant equivalence} (i), we have $\tilde B\in(\mathcal{U}_{1,0}^\mathbb{C}\otimes(\mathfrak{m}^*\otimes\mathfrak{k})_{\mathbb{C}})_K$ if and only if
	\begin{align*}
		(X_\beta\otimes I\otimes I)(\tilde B)=(I\otimes(\operatorname{ad}^*(X_\beta)\otimes I-I\otimes \operatorname{ad}(X_\beta)))(\tilde B)
	\end{align*}
	for any $\beta=8,9,10$.  By direct calculation, we have
	\begin{align*}
		\operatorname{ad}^*(X_\beta)\otimes I-I\otimes \operatorname{ad}(X_\beta)=A_\beta'
	\end{align*}
	with respect to the ordered basis $(F_1,\dots,F_4)$, where
	\begin{equation}\label{A'}
		\begin{split}
			A_8'=\begin{pmatrix}
				0&0&0&1\\
				0&0&1&0\\
				0&-1&0&0\\
				-1&0&0&0
			\end{pmatrix}\quad
			A_9'=\begin{pmatrix}
				0&0&-1&0\\
				0&0&0&1\\
				1&0&0&0\\
				0&-1&0&0
			\end{pmatrix}\quad
			A_{10}'=\begin{pmatrix}
				0&1&0&0\\
				-1&0&0&0\\
				0&0&0&1\\
				0&0&-1&0
			\end{pmatrix}.
		\end{split}
	\end{equation}
	For any $1\le t\le4$, we have
	\begin{align*}
		X_\beta\otimes I\otimes I-I\otimes(\operatorname{ad}^*(X_\beta)\otimes I-I\otimes \operatorname{ad}(X_\beta))=A_\beta
	\end{align*}
	with respect to the ordered basis $(U^{2t}_{1,0}\otimes F_1,\dots,U^{2t}_{1,0}\otimes F_4,U^{4t}_{1,0}\otimes F_1,\dots,U^{4t}_{1,0}\otimes F_4)$, where
	\begin{align*}
		A_8=\begin{pmatrix}
			iI-A_8'&0\\
			0&-iI-A_8'
		\end{pmatrix}\quad
		A_9=\begin{pmatrix}
			-A_9'&-I\\
			I&-A_9'
		\end{pmatrix}\quad
		A_{10}=\begin{pmatrix}
			-A_{10}'&-iI\\
			-iI&-A_{10}'
		\end{pmatrix}.
	\end{align*}
	A direct calculation shows that the solution space of the system
	\begin{align*}
		A_8x=A_9x=A_{10}x=0.
	\end{align*}
	is
	\begin{align*}
		\operatorname{span}_{\mathbb{C}}\{(1,0,0,i,0,-i,-1,0),(0,1,i,0,i,0,0,1)\}
	\end{align*}
	and thus
	\begin{align*}
		(\mathcal{U}_{1,0}^\mathbb{C}\otimes(\mathfrak{m}^*\otimes\mathfrak{k})_{\mathbb{C}})_K=\operatorname{span}_{\mathbb{C}}\{\Phi_{1t},\Phi_{2t}\mid1\le t\le 4\},
	\end{align*}
	where
	\begin{align*}
		&\Phi_{1t}=U^{2t}_{1,0}\otimes F_1+iU^{2t}_{1,0}\otimes F_4-iU^{4t}_{1,0}\otimes F_2-U^{4t}_{1,0}\otimes F_3,\\
		&\Phi_{2t}=U^{2t}_{1,0}\otimes F_2+iU^{2t}_{1,0}\otimes F_3+iU^{4t}_{1,0}\otimes F_1+U^{4t}_{1,0}\otimes F_4.
	\end{align*}
	Finally, we verify the Coulomb condition and compute the 3 terms of $\tilde J^{\mathbb{C}}$. Substituting the above vectors into the formula for $\tilde\delta$ gives
	\begin{align*}
		\tilde\delta(\Phi_{1t})=\tilde\delta(\Phi_{2t})=0.
	\end{align*}
	Since $U^{2t}_{1,0},U^{4t}_{1,0}\in\mathcal{U}_{1,0}(0,1)\otimes_{\mathbb{R}}\mathbb{C}$, Lemma \ref{eigenvalue of J2} shows that $\tilde J_2^{\mathbb{C}}=\lambda_{1,0}(0,1)\operatorname{Id}=4\operatorname{Id}$ on $(\mathcal{U}_{1,0}^{\mathbb{C}}\otimes(\mathfrak{m}^*\otimes\mathfrak{k})_{\mathbb{C}})_K$. Lemma \ref{eigenvalue of J0} shows that $\tilde J_0^{\mathbb{C}}=-7\operatorname{Id}$ on this space. Using
	\begin{align*}
		\tilde J_1=\sum_{i=1}^3X_i\otimes\tilde J_{1i}\otimes I+\sum_{i=4}^7\sqrt2X_i\otimes\tilde J_{1i}\otimes I
	\end{align*}
	defined in Lemma \ref{eigenvalue of J1} and Table \ref{J_{1i}(xi_j)}, we can calculate that
	\begin{align*}
		\tilde J_1^{\mathbb{C}}(\Phi_{1t})=\tilde J_1^{\mathbb{C}}(\Phi_{2t})=0.
	\end{align*}
	Hence we obtain
	\begin{align*}
		\tilde J^{\mathbb{C}}=-3\operatorname{Id}
	\end{align*}
	on $(\mathcal{U}_{1,0}^\mathbb{C}\otimes(\mathfrak{m}^*\otimes\mathfrak{k})_{\mathbb{C}})_K$ and
	\begin{align*}
		\dim_{\mathbb{C}}(\mathcal{U}_{1,0}^\mathbb{C}\otimes(\mathfrak{m}^*\otimes\mathfrak{k})_{\mathbb{C}})_K=8.
	\end{align*}
	\subsection{The representation $V_{0,1}$}\label{app V01}
	We use the basis $E_1,\dots,E_5$ of $V_{0,1}$ and the matrix coefficients $U^{st}_{0,1}$ introduced in Section \ref{section nullity}. Differentiating the one-parameter subgroup formulas at the identity gives
	\begingroup
	\allowdisplaybreaks[4]
	\begin{alignat*}{2}
		&X_1(U^{st}_{0,1})=\left\{
		\begin{array}{ll}
			iU^{st}_{0,1},&s=1\ or\ 3,\\
			-iU^{st}_{0,1},&s=4\ or\ 5,\\
			0,&s=2,
		\end{array}
		\right.
		&X_6(U^{st}_{0,1})=\left\{
		\begin{array}{ll}
			-\sqrt2(U^{1t}_{0,1}+U^{5t}_{0,1}),&s=2,\\
			\frac1{\sqrt2}U^{2t}_{0,1},&s=1\ or\ 5,\\
			0,&s=3\ or\ 4,
		\end{array}
		\right.\\
		&X_2(U^{st}_{0,1})=\left\{
		\begin{array}{ll}
			-U^{4t}_{0,1},&s=1,\\
			0,&s=2,\\
			U^{5t}_{0,1},&s=3,\\
			U^{1t}_{0,1},&s=4,\\
			-U^{3t}_{0,1},&s=5,
		\end{array}
		\right.
		&X_7(U^{st}_{0,1})=\left\{
		\begin{array}{ll}
			\sqrt2i(U^{5t}_{0,1}-U^{1t}_{0,1}),&s=2,\\
			-\frac i{\sqrt2}U^{2t}_{0,1},&s=1,\\
			\frac i{\sqrt2}U^{2t}_{0,1},&s=5,\\
			0,&s=3\ or\ 4,
		\end{array}
		\right.\\
		&X_3(U^{st}_{0,1})=\left\{
		\begin{array}{ll}
			iU^{4t}_{0,1},&s=1,\\
			0,&s=2,\\
			-iU^{5t}_{0,1},&s=3,\\
			iU^{1t}_{0,1},&s=4,\\
			-iU^{3t}_{0,1},&s=5,
		\end{array}
		\right.
		&X_8(U^{st}_{0,1})=\left\{
		\begin{array}{ll}
			iU^{st}_{0,1},&s=1\ or\ 4,\\
			-iU^{st}_{0,1},&s=3\ or\ 5,\\
			0,&s=2,
		\end{array}
		\right.\\
		&X_4(U^{st}_{0,1})=\left\{
		\begin{array}{ll}
			-\sqrt2(U^{3t}_{0,1}+U^{4t}_{0,1}),&s=2,\\
			\frac1{\sqrt2}U^{2t}_{0,1},&s=3\ or\ 4,\\
			0,&s=1\ or\ 5,
		\end{array}
		\right.
		&X_9(U^{st}_{0,1})=\left\{
		\begin{array}{ll}
			U^{3t}_{0,1},&s=1,\\
			0,&s=2,\\
			-U^{1t}_{0,1},&s=3,\\
			-U^{5t}_{0,1},&s=4,\\
			U^{4t}_{0,1},&s=5,
		\end{array}
		\right.\\
		&X_5(U^{st}_{0,1})=\left\{
		\begin{array}{ll}
			\sqrt2i(U^{4t}_{0,1}-U^{3t}_{0,1}),&s=2,\\
			-\frac i{\sqrt2}U^{2t}_{0,1},&s=3,\\
			\frac i{\sqrt2}U^{2t}_{0,1},&s=4,\\
			0,&s=1\ or\ 5.
		\end{array}
		\right.
		&X_{10}(U^{st}_{0,1})=\left\{
		\begin{array}{ll}
			-iU^{3t}_{0,1},&s=1,\\
			0,&s=2,\\
			-iU^{1t}_{0,1},&s=3,\\
			iU^{5t}_{0,1},&s=4,\\
			iU^{4t}_{0,1},&s=5
		\end{array}
		\right.
	\end{alignat*}
	\endgroup
	for any $1\le t\le5$ and
	\begin{align*}
		X_\alpha^2(U^{st}_{0,1})=\left\{
		\begin{array}{ll}
			-U^{st}_{0,1},&s\ne2,\\
			0,&s=2,
		\end{array}
		\right.\\
	\end{align*}
	for any $8\le\alpha\le10$. Hence the spectrum of $-C_K\otimes I\otimes I$ is $\{0,-3\}$. As in the preceding case, the only common eigenvalue of the 2 $K$-Casimir operators is $-3$. Therefore,
	\begin{align*}
		(\mathcal{U}_{0,1}^{\mathbb{C}}\otimes(\mathfrak{m}^*\otimes\mathfrak{k})_{\mathbb{C}})_K\subset \operatorname{span}_{\mathbb{C}}\{U^{st}_{0,1}\mid s\ne2,\ 1\le t\le5\}\otimes \operatorname{span}_{\mathbb{C}}\{F_1,F_2,F_3,F_4\}.
	\end{align*}
	This common eigenvalue corresponds to $(\lambda,q)=(-7,1)$. Since the branching rule for $V_{0,1}$ gives $(p,q)=(0,0)$ and $(1,1)$, we have $p=1$. Hence
	\begin{align*}
		U^{st}_{0,1}\in\mathcal{U}^\mathbb{C}_{0,1}(1,1)
	\end{align*}
	for $s\ne2$. For each fixed $1\le t\le5$, with respect to the ordered basis
	\begin{align*}
		(U^{1t}_{0,1}\otimes F_1,\dots,U^{1t}_{0,1}\otimes F_4,U^{3t}_{0,1}\otimes F_1,\dots,U^{3t}_{0,1}\otimes F_4,U^{4t}_{0,1}\otimes F_1,\dots,U^{4t}_{0,1}\otimes F_4,U^{5t}_{0,1}\otimes F_1,\dots,U^{5t}_{0,1}\otimes F_4),
	\end{align*}
	the 3 infinitesimal $K$-equivariance operators are represented by $\tilde B_8,\tilde B_9,\tilde B_{10}$, where
	\begin{align*}
		&\tilde B_8=\begin{pmatrix}
			iI-A_8'&0&0&0\\
			0&-iI-A_8'&0&0\\
			0&0&iI-A_8'&0\\
			0&0&0&-iI-A_8'
		\end{pmatrix},\\
		&\tilde B_9=\begin{pmatrix}
			-A_9'&-I&0&0\\
			I&-A_9'&0&0\\
			0&0&-A_9'&I\\
			0&0&-I&-A_9'
		\end{pmatrix},\\
		&\tilde B_{10}=\begin{pmatrix}
			-A_{10}'&-iI&0&0\\
			-iI&-A_{10}'&0&0\\
			0&0&-A_{10}'&iI\\
			0&0&iI&-A_{10}'
		\end{pmatrix}
	\end{align*}
	where $A_\beta'$ is defined in \eqref{A'}. Direct calculation gives
	\begin{align*}
		(\mathcal{U}_{0,1}^{\mathbb{C}}\otimes(\mathfrak{m}^*\otimes\mathfrak{k})_{\mathbb{C}})_K=\operatorname{span}_{\mathbb{C}}\{\tilde\Phi_{1t},\tilde\Phi_{2t},\tilde\Phi_{3t},\tilde\Phi_{4t}\mid1\le t\le5\},
	\end{align*}
	where
	\begin{align*}
		&\tilde\Phi_{1t}=U^{1t}_{0,1}\otimes F_1+iU^{1t}_{0,1}\otimes F_4-iU^{3t}_{0,1}\otimes F_2-U^{3t}_{0,1}\otimes F_3,\\
		&\tilde\Phi_{2t}=U^{1t}_{0,1}\otimes F_2+iU^{1t}_{0,1}\otimes F_3+iU^{3t}_{0,1}\otimes F_1+U^{3t}_{0,1}\otimes F_4,\\
		&\tilde\Phi_{3t}=U^{4t}_{0,1}\otimes F_1+iU^{4t}_{0,1}\otimes F_4+iU^{5t}_{0,1}\otimes F_2+U^{5t}_{0,1}\otimes F_3,\\
		&\tilde\Phi_{4t}=U^{4t}_{0,1}\otimes F_2+iU^{4t}_{0,1}\otimes F_3-iU^{5t}_{0,1}\otimes F_1-U^{5t}_{0,1}\otimes F_4.
	\end{align*}
	A direct substitution gives
	\begin{align*}
		\tilde\delta(\tilde\Phi_{it})=0
	\end{align*}
	for $1\le i\le4$ and $1\le t\le5$. Since $U^{st}_{0,1}\in\mathcal{U}^\mathbb{C}_{0,1}(1,1)$ for $s\ne2$, Lemma \ref{eigenvalue of J2} gives
	\begin{align*}
		\tilde J_2^{\mathbb{C}}=\lambda_{0,1}(1,1)\operatorname{Id}=7\operatorname{Id}
	\end{align*}
	on $(\mathcal{U}_{0,1}^{\mathbb{C}}\otimes(\mathfrak{m}^*\otimes\mathfrak{k})_{\mathbb{C}})_K$, while Lemma \ref{eigenvalue of J0} gives
	\begin{align*}
		\tilde J_0^{\mathbb{C}}=-7\operatorname{Id}
	\end{align*}
	on the same space. Using Lemma \ref{eigenvalue of J1} and Table \ref{J_{1i}(xi_j)}, we obtain
	\begin{align*}
		\tilde J_1^{\mathbb{C}}(\tilde\Phi_{it})=0
	\end{align*}
	for $1\le i\le4$ and $1\le t\le5$. Consequently,
	\begin{align*}
		\tilde J^{\mathbb{C}}=0
	\end{align*}
	on $(\mathcal{U}_{0,1}^{\mathbb{C}}\otimes(\mathfrak{m}^*\otimes\mathfrak{k})_{\mathbb{C}})_K$ and
	\begin{align*}
		\dim_{\mathbb{C}}(\mathcal{U}_{0,1}^{\mathbb{C}}\otimes(\mathfrak{m}^*\otimes\mathfrak{k})_{\mathbb{C}})_K=20.
	\end{align*}
	\subsection{The representation $V_{2,0}$}\label{app V20}
	For $v,w\in\mathbb{C}^4$, set
	\begin{align*}
		v\odot w=v\otimes w+w\otimes v.
	\end{align*}
	By construction, $V_{2,0}=\operatorname{span}_{\mathbb{C}}\{\bar E_1,\dots,\bar E_{10}\}=\operatorname{Sym}^2\mathbb{C}^4$, where
	\begin{align*}
		&\bar E_1=e_1\otimes e_1,\quad\bar E_2=e_2\otimes e_2,\quad\bar E_3=e_3\otimes e_3,\quad\bar E_4=e_4\otimes e_4,\quad\bar E_5=e_1\odot e_2,\\
		&\bar E_6=e_1\odot e_3,\quad\bar E_7=e_1\odot e_4,\quad\bar E_8=e_2\odot e_3,\quad\bar E_9=e_2\odot e_4,\quad\bar E_{10}=e_3\odot e_4.
	\end{align*}
	The representation $\pi_{2,0}$ is given by
	\begin{align*}
		\pi_{2,0}(g)(v\odot w)=(\chi(g)v)\odot(\chi(g)w)
	\end{align*}
	for $g\in G$ and $v,w\in\mathbb{C}^4$. Equip $V_{2,0}$ with the induced Hermitian inner product and define
	\begin{align*}
		U^{st}_{2,0}(g)=\langle\pi_{2,0}(g)\bar E_s,\bar E_t\rangle
	\end{align*}
	for $1\le s,t\le10$. These functions span the same matrix-coefficient space as those defined in \eqref{U^st}. Differentiating the one-parameter subgroup formulas at the identity gives
	\par\addvspace{10pt}
	\begingroup
	\captionsetup{labelsep=colon,skip=0pt}
	\captionof{table}{$\mathbf{X_i(U^{st}_{2,0})\text{ on }V_{2,0}}$}
	\label{X_i(Ust)-V20}
	\setlength{\abovedisplayskip}{2pt}
	\setlength{\abovedisplayshortskip}{2pt}
	\begin{align*}
		\begin{array}{|c|c|c|c|c|c|}
			\hline
			X_i(U^{st}_{2,0})&i=1&i=2&i=3&i=4&i=5\\
			\hline
			s=1&2iU^{1t}_{2,0}&U^{6t}_{2,0}&-iU^{6t}_{2,0}&-\frac1{\sqrt2}U^{5t}_{2,0}&\frac i{\sqrt2}U^{5t}_{2,0}\\
			\hline
			s=2&0&0&0&\frac1{\sqrt2}U^{5t}_{2,0}&\frac i{\sqrt2}U^{5t}_{2,0}\\
			\hline
			s=3&-2iU^{3t}_{2,0}&-U^{6t}_{2,0}&-iU^{6t}_{2,0}&-\frac1{\sqrt2}U^{10,t}_{2,0}&-\frac i{\sqrt2}U^{10,t}_{2,0}\\
			\hline
			s=4&0&0&0&\frac1{\sqrt2}U^{10,t}_{2,0}&-\frac i{\sqrt2}U^{10,t}_{2,0}\\
			\hline
			s=5&iU^{5t}_{2,0}&U^{8t}_{2,0}&-iU^{8t}_{2,0}&\sqrt2(U^{1t}_{2,0}-U^{2t}_{2,0})&\sqrt2i(U^{1t}_{2,0}+U^{2t}_{2,0})\\
			\hline
			s=6&0&2(U^{3t}_{2,0}-U^{1t}_{2,0})&-2i(U^{1t}_{2,0}+U^{3t}_{2,0})&-\frac1{\sqrt2}(U^{8t}_{2,0}+U^{7t}_{2,0})&\frac i{\sqrt2}(U^{8t}_{2,0}-U^{7t}_{2,0})\\
			\hline
			s=7&iU^{7t}_{2,0}&U^{10,t}_{2,0}&-iU^{10,t}_{2,0}&\frac1{\sqrt2}(U^{6t}_{2,0}-U^{9t}_{2,0})&\frac i{\sqrt2}(U^{9t}_{2,0}-U^{6t}_{2,0})\\
			\hline
			s=8&-iU^{8t}_{2,0}&-U^{5t}_{2,0}&-iU^{5t}_{2,0}&\frac1{\sqrt2}(U^{6t}_{2,0}-U^{9t}_{2,0})&\frac i{\sqrt2}(U^{6t}_{2,0}-U^{9t}_{2,0})\\
			\hline
			s=9&0&0&0&\frac1{\sqrt2}(U^{7t}_{2,0}+U^{8t}_{2,0})&\frac i{\sqrt2}(U^{7t}_{2,0}-U^{8t}_{2,0})\\
			\hline
			s=10&-iU^{10,t}_{2,0}&-U^{7t}_{2,0}&-iU^{7t}_{2,0}&\sqrt2(U^{3t}_{2,0}-U^{4t}_{2,0})&-\sqrt2i(U^{3t}_{2,0}+U^{4t}_{2,0})\\
			\hline
		\end{array}
	\end{align*}
	\endgroup
	\begin{align*}
		\begin{array}{|c|c|c|c|c|c|}
			\hline
			X_i(U^{st}_{2,0})&i=6&i=7&i=8&i=9&i=10\\
			\hline
			s=1&\frac1{\sqrt2}U^{7t}_{2,0}&-\frac i{\sqrt2}U^{7t}_{2,0}&0&0&0\\
			\hline
			s=2&\frac1{\sqrt2}U^{8t}_{2,0}&-\frac i{\sqrt2}U^{8t}_{2,0}&2iU^{2t}_{2,0}&U^{9t}_{2,0}&-iU^{9t}_{2,0}\\
			\hline
			s=3&-\frac1{\sqrt2}U^{8t}_{2,0}&-\frac i{\sqrt2}U^{8t}_{2,0}&0&0&0\\
			\hline
			s=4&-\frac1{\sqrt2}U^{7t}_{2,0}&-\frac i{\sqrt2}U^{7t}_{2,0}&-2iU^{4t}_{2,0}&-U^{9t}_{2,0}&-iU^{9t}_{2,0}\\
			\hline
			s=5&\frac1{\sqrt2}(U^{6t}_{2,0}+U^{9t}_{2,0})&-\frac i{\sqrt2}(U^{6t}_{2,0}+U^{9t}_{2,0})&iU^{5t}_{2,0}&U^{7t}_{2,0}&-iU^{7t}_{2,0}\\
			\hline
			s=6&\frac1{\sqrt2}(U^{10,t}_{2,0}-U^{5t}_{2,0})&-\frac i{\sqrt2}(U^{5t}_{2,0}+U^{10,t}_{2,0})&0&0&0\\
			\hline
			s=7&\sqrt2(U^{4t}_{2,0}-U^{1t}_{2,0})&-\sqrt2i(U^{1t}_{2,0}+U^{4t}_{2,0})&-iU^{7t}_{2,0}&-U^{5t}_{2,0}&-iU^{5t}_{2,0}\\
			\hline
			s=8&\sqrt2(U^{3t}_{2,0}-U^{2t}_{2,0})&-\sqrt2i(U^{2t}_{2,0}+U^{3t}_{2,0})&iU^{8t}_{2,0}&U^{10,t}_{2,0}&-iU^{10,t}_{2,0}\\
			\hline
			s=9&\frac1{\sqrt2}(U^{10,t}_{2,0}-U^{5t}_{2,0})&-\frac i{\sqrt2}(U^{5t}_{2,0}+U^{10,t}_{2,0})&0&2(U^{4t}_{2,0}-U^{2t}_{2,0})&-2i(U^{2t}_{2,0}+U^{4t}_{2,0})\\
			\hline
			s=10&-\frac1{\sqrt2}(U^{6t}_{2,0}+U^{9t}_{2,0})&-\frac i{\sqrt2}(U^{6t}_{2,0}+U^{9t}_{2,0})&-iU^{10,t}_{2,0}&-U^{8t}_{2,0}&-iU^{8t}_{2,0}\\
			\hline
		\end{array}
	\end{align*}
	for any $1\le t\le 10$ and
	\begin{align*}
		&X_8^2(U^{st}_{2,0})=\left\{
		\begin{array}{ll}
			-4U^{st}_{2,0},&s=2,4,\\
			-U^{st}_{2,0},&s=5,7,8,10,\\
			0,&s=1,3,6,9,
		\end{array}
		\right.\\
		&X_9^2(U^{st}_{2,0})=\left\{
		\begin{array}{ll}
			2U^{4t}_{2,0}-2U^{2t}_{2,0},&s=2,\\
			2U^{2t}_{2,0}-2U^{4t}_{2,0},&s=4,\\
			-U^{st}_{2,0},&s=5,7,8,10,\\
			-4U^{9t}_{2,0},&s=9,\\
			0,&s=1,3,6,
		\end{array}
		\right.\\
		&X_{10}^2(U^{st}_{2,0})=\left\{
		\begin{array}{ll}
			-2(U^{2t}_{2,0}+U^{4t}_{2,0}),&s=2,4,\\
			-U^{st}_{2,0},&s=5,7,8,10,\\
			-4U^{9t}_{2,0},&s=9,\\
			0,&s=1,3,6,
		\end{array}
		\right.
	\end{align*}
	for any $1\le t\le10$. Hence the spectrum of $-C_K\otimes I\otimes I$ is $\{0,-3,-8\}$. As in the preceding cases, matching the eigenvalues of the 2 $K$-Casimir operators gives
	\begin{align*}
		(\mathcal{U}_{2,0}^{\mathbb{C}}\otimes(\mathfrak{m}^*\otimes\mathfrak{k})_{\mathbb{C}})_K\subset&\operatorname{span}_{\mathbb{C}}\{U^{st}_{2,0}\mid s=5,7,8,10,\ 1\le t\le 10\}\otimes \operatorname{span}_{\mathbb{C}}\{F_i\mid 1\le i\le 4\}\\
		&\oplus \operatorname{span}_{\mathbb{C}}\{U^{st}_{2,0}\mid s=2,4,9,\ 1\le t\le 10\}\otimes \operatorname{span}_{\mathbb{C}}\{\xi_i\otimes X_\alpha\mid1\le i\le3,\ 8\le\alpha\le10\}.
	\end{align*}
	The first and second summands correspond to $(\lambda,q)=(-7,1)$ and $(12,2)$, respectively. Hence $U^{st}_{2,0}\in\mathcal{U}_{2,0}^{\mathbb{C}}(1,1)$ for $s=5,7,8,10$ and $U^{st}_{2,0}\in \mathcal{U}_{2,0}^{\mathbb{C}}(0,2)$ for $s=2,4,9$. For each fixed $1\le t\le10$, the infinitesimal $K$-equivariance equations on the first summand are
	\begin{align*}
		\bar A_8x=\bar A_9x=\bar A_{10}x=0,\quad x\in\mathbb{C}^{16},
	\end{align*}
	with respect to the ordered basis
	\begin{align*}
		(&U^{5t}_{2,0}\otimes F_1,\dots,U^{5t}_{2,0}\otimes F_4,U^{7t}_{2,0}\otimes F_1,\dots,U^{7t}_{2,0}\otimes F_4,\\
		&U^{8t}_{2,0}\otimes F_1,\dots,U^{8t}_{2,0}\otimes F_4,U^{10t}_{2,0}\otimes F_1,\dots,U^{10t}_{2,0}\otimes F_4).
	\end{align*}
	For each fixed $1\le j\le3$ and $1\le t\le10$, the equations on the second summand are
	\begin{align*}
		\bar B_8y=\bar B_9y=\bar B_{10}y=0,\qquad y\in\mathbb{C}^9,
	\end{align*}
	with respect to the ordered basis
	\begin{align*}
		(&U^{2t}_{2,0}\otimes\xi_j\otimes X_8,U^{2t}_{2,0}\otimes\xi_j\otimes X_9,U^{2t}_{2,0}\otimes\xi_j\otimes X_{10},\\
		&U^{4t}_{2,0}\otimes\xi_j\otimes X_8,U^{4t}_{2,0}\otimes\xi_j\otimes X_9,U^{4t}_{2,0}\otimes\xi_j\otimes X_{10},\\
		&U^{9t}_{2,0}\otimes\xi_j\otimes X_8,U^{9t}_{2,0}\otimes\xi_j\otimes X_9,U^{9t}_{2,0}\otimes\xi_j\otimes X_{10}),
	\end{align*}
	where
	\begin{align*}
		&\bar A_8=\begin{pmatrix}
			iI-A_8'&0&0&0\\
			0&-iI-A_8'&0&0\\
			0&0&iI-A_8'&0\\
			0&0&0&-iI-A_8'
		\end{pmatrix}\\
		&\bar A_9=\begin{pmatrix}
			-A_9'&-I&0&0\\
			I&-A_9'&0&0\\
			0&0&-A_9'&-I\\
			0&0&I&-A_9'
		\end{pmatrix}\\
		&\bar A_{10}=\begin{pmatrix}
			-A_{10}'&-iI&0&0\\
			-iI&-A_{10}'&0&0\\
			0&0&-A_{10}'&-iI\\
			0&0&-iI&-A_{10}'
		\end{pmatrix}
	\end{align*}
	and
	\begin{align*}
		&\bar B_8=\begin{pmatrix}
			2iI+B_8'&0&0\\
			0&-2iI+B_8'&0\\
			0&0&B_8'
		\end{pmatrix}\\
		&\bar B_9=\begin{pmatrix}
			B_9'&0&-2I\\
			0&B_9'&2I\\
			I&-I&B_9'
		\end{pmatrix}\\
		&\bar B_{10}=\begin{pmatrix}
			B_{10}'&0&-2iI\\
			0&B_{10}'&-2iI\\
			-iI&-iI&B_{10}'
		\end{pmatrix}
	\end{align*}
	with
	\begin{align*}
		B_8'=\begin{pmatrix}
			0&0&0\\
			0&0&-2\\
			0&2&0
		\end{pmatrix}\quad
		B_9'=\begin{pmatrix}
			0&0&2\\
			0&0&0\\
			-2&0&0
		\end{pmatrix}\quad
		B_{10}'=\begin{pmatrix}
			0&-2&0\\
			2&0&0\\
			0&0&0
		\end{pmatrix}.
	\end{align*}
	Direct calculation gives
	\begin{align*}
		(\mathcal{U}_{2,0}^{\mathbb{C}}\otimes(\mathfrak{m}^*\otimes\mathfrak{k})_{\mathbb{C}})_K=\operatorname{span}_{\mathbb{C}}\{\bar\Phi_{1t},\bar\Phi_{2t},\bar\Phi_{3t},\bar\Phi_{4t},\bar\Phi_{5jt}\mid 1\le t\le10,\ 1\le j\le3\},
	\end{align*}
	where
	\begin{align*}
		&\bar\Phi_{1t}=U^{5t}_{2,0}\otimes F_1+iU^{5t}_{2,0}\otimes F_4-iU^{7t}_{2,0}\otimes F_2-U^{7t}_{2,0}\otimes F_3,\\
		&\bar\Phi_{2t}=U^{5t}_{2,0}\otimes F_2+iU^{5t}_{2,0}\otimes F_3+iU^{7t}_{2,0}\otimes F_1+U^{7t}_{2,0}\otimes F_4,\\
		&\bar\Phi_{3t}=U^{8t}_{2,0}\otimes F_1+iU^{8t}_{2,0}\otimes F_4-iU^{10,t}_{2,0}\otimes F_2-U^{10,t}_{2,0}\otimes F_3,\\
		&\bar\Phi_{4t}=U^{8t}_{2,0}\otimes F_2+iU^{8t}_{2,0}\otimes F_3+iU^{10,t}_{2,0}\otimes F_1+U^{10,t}_{2,0}\otimes F_4,\\
		&\bar\Phi_{5jt}=(U^{2t}_{2,0}+U^{4t}_{2,0})\otimes\xi_j\otimes X_9+i(U^{2t}_{2,0}-U^{4t}_{2,0})\otimes\xi_j\otimes X_{10}+iU^{9t}_{2,0}\otimes\xi_j\otimes X_8.
	\end{align*}
	A direct substitution gives
	\begin{align*}
		(\mathcal{U}_{2,0}^{\mathbb{C}}\otimes(\mathfrak{m}^*\otimes\mathfrak{k})_{\mathbb{C}})_K\cap\ker\tilde\delta=\operatorname{span}_{\mathbb{C}}\{\bar\Phi_{1t},\bar\Phi_{4t},\bar\Phi_{5jt},\bar\Phi_{2t}-i\bar\Phi_{3t}\mid1\le t\le10,\ 1\le j\le3\}.
	\end{align*}
	Lemma \ref{eigenvalue of J2} and Lemma \ref{eigenvalue of J0} give
	\begin{align*}
		\tilde J_2^{\mathbb{C}}=\lambda_{2,0}(1,1)\operatorname{Id}=15\operatorname{Id},\quad\tilde J_0^{\mathbb{C}}=-7\operatorname{Id}
	\end{align*}
	on $\operatorname{span}_{\mathbb{C}}\{\bar\Phi_{it}\mid1\le i\le4,\ 1\le t\le10\}$, and
	\begin{align*}
		\tilde J_2^{\mathbb{C}}=\lambda_{2,0}(0,2)\operatorname{Id}=8\operatorname{Id},\quad\tilde J_0^{\mathbb{C}}=12\operatorname{Id}
	\end{align*}
	on $\operatorname{span}_{\mathbb{C}}\{\bar\Phi_{5jt}\mid1\le j\le3,\ 1\le t\le10\}$. Using Lemma \ref{eigenvalue of J1} and Table \ref{J_{1i}(xi_j)}, we obtain
	\begin{align*}
		&\tilde J_1^{\mathbb{C}}(\bar\Phi_{1t})=8\sqrt2(\bar\Phi_{53t}+i\bar\Phi_{52t}),\\
		&\tilde J_1^{\mathbb{C}}(\bar\Phi_{2t}-i\bar\Phi_{3t})=-16\sqrt2i\bar\Phi_{51t},\\
		&\tilde J_1^{\mathbb{C}}(\bar\Phi_{4t})=8\sqrt2(\bar\Phi_{52t}+i\bar\Phi_{53t}),\\
		&\tilde J_1^{\mathbb{C}}(\bar\Phi_{51t})=\sqrt2(\bar\Phi_{3t}+i\bar\Phi_{2t}),\\
		&\tilde J_1^{\mathbb{C}}(\bar\Phi_{52t})=\sqrt2(\bar\Phi_{4t}-i\bar\Phi_{1t}),\\
		&\tilde J_1^{\mathbb{C}}(\bar\Phi_{53t})=\sqrt2(\bar\Phi_{1t}-i\bar\Phi_{4t})
	\end{align*}
	for any $1\le t\le10$, that is
	\begin{align*}
		\tilde J_1^{\mathbb{C}}=\begin{pmatrix}
			0&0&0&0&-\sqrt2i&\sqrt2\\
			0&0&0&\sqrt2i&0&0\\
			0&0&0&0&\sqrt2&-\sqrt2i\\
			0&-16\sqrt2i&0&0&0&0\\
			8\sqrt2i&0&8\sqrt2&0&0&0\\
			8\sqrt2&0&8\sqrt2i&0&0&0
		\end{pmatrix}
	\end{align*}
	with respect to the ordered basis $(\bar\Phi_{1t},\bar\Phi_{2t}-i\bar\Phi_{3t},\bar\Phi_{4t},\bar\Phi_{51t},\bar\Phi_{52t},\bar\Phi_{53t})$ for any $1\le t\le10$. Its characteristic polynomial is
	\begin{align*}
		\det(\lambda I-\tilde J_1^{\mathbb{C}})=(\lambda-4\sqrt2)^3(\lambda+4\sqrt2)^3.
	\end{align*}
	Thus the eigenvalues of $\tilde J_1^{\mathbb{C}}$ are $4\sqrt2$ and $-4\sqrt2$. Hence
	\begin{align*}
		\tilde J^{\mathbb{C}}\ge(8-4\sqrt2)\operatorname{Id}>0
	\end{align*}
	on $(\mathcal{U}_{2,0}^{\mathbb{C}}\otimes(\mathfrak{m}^*\otimes\mathfrak{k})_{\mathbb{C}})_K\cap\ker\tilde\delta$.
	\subsection{The representation $V_{1,1}$}\label{app V11}
	By construction, $V_{1,1}=\operatorname{span}_{\mathbb{C}}\{\hat E_1,\dots,\hat E_{16}\}\subset\mathbb{C}^4\otimes\Lambda^2\mathbb{C}^4$, where
	\begin{align*}
		\begin{array}{llll}
			\hat E_1=e_1\otimes E_1,&\hat E_2=e_1\otimes E_2+e_4\otimes E_1,&\hat E_3=e_1\otimes E_3,&\hat E_4=e_1\otimes E_4-e_3\otimes E_1,\\
			\hat E_5=e_1\otimes E_5+e_3\otimes E_3,&\hat E_6=e_2\otimes E_1,&\hat E_7=e_2\otimes E_2+e_3\otimes E_1,&\hat E_8=e_2\otimes E_3+e_4\otimes E_1,\\
			\hat E_9=e_2\otimes E_4,&\hat E_{10}=e_2\otimes E_5-e_4\otimes E_4,&\hat E_{11}=e_3\otimes E_2-e_4\otimes E_4,&\hat E_{12}=e_3\otimes E_3+e_4\otimes E_2,\\
			\hat E_{13}=e_3\otimes E_4,&\hat E_{14}=e_3\otimes E_5,&\hat E_{15}=e_4\otimes E_3,&\hat E_{16}=e_4\otimes E_5.
		\end{array}
	\end{align*}
	The tensor-product action of $G$ on $\mathbb{C}^4\otimes\Lambda^2\mathbb{C}^4$ is given by
	\begin{align*}
		u\otimes(w\wedge v)\longmapsto(\chi(g)u)\otimes((\chi(g)w)\wedge(\chi(g)v)),
	\end{align*}
	and $\pi_{1,1}$ is its restriction to $V_{1,1}$. Equip $V_{1,1}$ with the induced Hermitian inner product and define
	\begin{align*}
		U^{st}_{1,1}(g)=\langle\pi_{1,1}(g)\hat E_s,\hat E_t\rangle
	\end{align*}
	for $1\le s,t\le 16$. These functions span the same matrix-coefficient space as those defined in \eqref{U^st}. Differentiating the one-parameter subgroup formulas at the identity gives
	\par\addvspace{10pt}
	\begingroup
	\captionsetup{labelsep=colon,skip=0pt}
	\captionof{table}{$\mathbf{X_i(U^{st}_{1,1})\text{ on }V_{1,1}}$}
	\label{X_i(U^{st}) V_{1,1}}
	\setlength{\abovedisplayskip}{2pt}
	\setlength{\abovedisplayshortskip}{2pt}
	\begin{align*}
		\begin{array}{|c|c|c|c|c|}
			\hline
			X_i(U^{st}_{1,1})&s=1&s=2&s=3&s=4\\
			\hline
			i=1&2iU^{1t}_{1,1}&iU^{2t}_{1,1}&2iU^{3t}_{1,1}&0\\
			\hline
			i=2&-U^{4t}_{1,1}&U^{11,t}_{1,1}&U^{5t}_{1,1}&2(U^{1t}_{1,1}+U^{13,t}_{1,1})\\
			\hline
			i=3&iU^{4t}_{1,1}&-iU^{11,t}_{1,1}&-iU^{5t}_{1,1}&2i(U^{1t}_{1,1}-U^{13,t}_{1,1})\\
			\hline
			i=4&-\frac1{\sqrt2}U^{6t}_{1,1}&-\sqrt2(U^{3t}_{1,1}+U^{4t}_{1,1}+\frac12U^{7t}_{1,1})&\frac1{\sqrt2}(U^{2t}_{1,1}-U^{8t}_{1,1})&\frac1{\sqrt2}(U^{2t}_{1,1}-U^{9t}_{1,1})\\
			\hline
			i=5&\frac i{\sqrt2}U^{6t}_{1,1}&\sqrt2i(U^{4t}_{1,1}-U^{3t}_{1,1}+\frac12U^{7t}_{1,1})&\frac i{\sqrt2}(U^{8t}_{1,1}-U^{2t}_{1,1})&\frac i{\sqrt2}(U^{2t}_{1,1}+U^{9t}_{1,1})\\
			\hline
			i=6&\frac1{\sqrt2}U^{2t}_{1,1}&\sqrt2(-\frac32U^{1t}_{1,1}-U^{5t}_{1,1}+U^{12,t}_{1,1})&\frac1{\sqrt2}U^{15,t}_{1,1}&\frac1{\sqrt2}(U^{6t}_{1,1}-U^{11,t}_{1,1})\\
			\hline
			i=7&-\frac i{\sqrt2}U^{2t}_{1,1}&\sqrt2i(-\frac32U^{1t}_{1,1}+U^{5t}_{1,1}-U^{12,t}_{1,1})&-\frac i{\sqrt2}U^{15,t}_{1,1}&\frac i{\sqrt2}(U^{6t}_{1,1}+U^{11,t}_{1,1})\\
			\hline
			i=8&iU^{1t}_{1,1}&0&-iU^{3t}_{1,1}&iU^{4t}_{1,1}\\
			\hline
			i=9&U^{3t}_{1,1}&U^{15,t}_{1,1}-U^{6t}_{1,1}&-U^{1t}_{1,1}&-U^{5t}_{1,1}\\
			\hline
			i=10&-iU^{3t}_{1,1}&-i(U^{6t}_{1,1}+U^{15,t}_{1,1})&-iU^{1t}_{1,1}&iU^{5t}_{1,1}\\
			\hline
		\end{array}
	\end{align*}  
	\endgroup
	\begin{align*}
		\begin{array}{|c|c|c|c|c|}
			\hline
			X_i(U^{st}_{1,1})&s=5&s=6&s=7&s=8\\
			\hline
			i=1&0&iU^{6t}_{1,1}&0&iU^{8t}_{1,1}\\
			\hline
			i=2&2(U^{14,t}_{1,1}-U^{3t}_{1,1})&-U^{9t}_{1,1}&-U^{1t}_{1,1}-U^{13,t}_{1,1}&U^{10,t}_{1,1}\\
			\hline
			i=3&-2i(U^{3t}_{1,1}+U^{14,t}_{1,1})&iU^{9t}_{1,1}&i(U^{13,t}_{1,1}-U^{1t}_{1,1})&-iU^{10,t}_{1,1}\\
			\hline
			i=4&\frac1{\sqrt2}(-U^{10,t}_{1,1}+U^{11,t}_{1,1}-U^{15,t}_{1,1})&\frac1{\sqrt2}U^{1t}_{1,1}&\sqrt2(\frac12U^{2t}_{1,1}-U^{8t}_{1,1}-U^{9t}_{1,1})&\frac1{\sqrt2}(U^{3t}_{1,1}+U^{7t}_{1,1})\\
			\hline
			i=5&\frac i{\sqrt2}(U^{10,t}_{1,1}-U^{11,t}_{1,1}-U^{15,t}_{1,1})&\frac i{\sqrt2}U^{1t}_{1,1}&\sqrt2i(\frac12U^{2t}_{1,1}-U^{8t}_{1,1}+U^{9t}_{1,1})&\frac i{\sqrt2}(U^{3t}_{1,1}-U^{7t}_{1,1})\\
			\hline
			i=6&\frac1{\sqrt2}(U^{2t}_{1,1}-U^{8t}_{1,1}+U^{16,t}_{1,1})&\frac1{\sqrt2}U^{7t}_{1,1}&\sqrt2(-\frac32U^{6t}_{1,1}-U^{10,t}_{1,1}+U^{11,t}_{1,1})&\frac1{\sqrt2}(U^{12,t}_{1,1}-U^{1t}_{1,1})\\
			\hline
			i=7&\frac i{\sqrt2}(U^{2t}_{1,1}-U^{8t}_{1,1}-U^{16,t}_{1,1})&-\frac i{\sqrt2}U^{7t}_{1,1}&\sqrt2i(-\frac32U^{6t}_{1,1}+U^{10,t}_{1,1}-U^{11,t}_{1,1})&-\frac i{\sqrt2}(U^{1t}_{1,1}+U^{12,t}_{1,1})\\
			\hline
			i=8&-iU^{5t}_{1,1}&2iU^{6t}_{1,1}&iU^{7t}_{1,1}&0\\
			\hline
			i=9&U^{4t}_{1,1}&U^{8t}_{1,1}&U^{12,t}_{1,1}&2(U^{15,t}_{1,1}-U^{6t}_{1,1})\\
			\hline
			i=10&iU^{4t}_{1,1}&-iU^{8t}_{1,1}&-iU^{12,t}_{1,1}&-2i(U^{6t}_{1,1}+U^{15,t}_{1,1})\\
			\hline
		\end{array}
	\end{align*}   
	\begin{align*}
		\begin{array}{|c|c|c|c|c|}
			\hline
			X_i(U^{st}_{1,1})&s=9&s=10&s=11&s=12\\
			\hline
			i=1&-iU^{9t}_{1,1}&-iU^{10,t}_{1,1}&-iU^{11,t}_{1,1}&0\\
			\hline
			i=2&U^{6t}_{1,1}&-U^{8t}_{1,1}&-U^{2t}_{1,1}&U^{14,t}_{1,1}-U^{3t}_{1,1}\\
			\hline
			i=3&iU^{6t}_{1,1}&-iU^{8t}_{1,1}&-iU^{2t}_{1,1}&-i(U^{3t}_{1,1}+U^{14,t}_{1,1})\\
			\hline
			i=4&\frac1{\sqrt2}(U^{4t}_{1,1}+U^{7t}_{1,1})&\frac1{\sqrt2}(U^{5t}_{1,1}-U^{12,t}_{1,1}-U^{13,t}_{1,1})&-\sqrt2(U^{12,t}_{1,1}+\frac32U^{13,t}_{1,1})&\sqrt2(U^{11,t}_{1,1}-\frac32U^{15,t}_{1,1})\\
			\hline
			i=5&\frac i{\sqrt2}(U^{4t}_{1,1}+U^{7t}_{1,1})&\frac i{\sqrt2}(U^{5t}_{1,1}-U^{12,t}_{1,1}+U^{13,t}_{1,1})&\sqrt2i(\frac32U^{13,t}_{1,1}-U^{12,t}_{1,1})&-\sqrt2i(\frac32U^{15,t}_{1,1}+U^{11,t}_{1,1})\\
			\hline
			i=6&\frac1{\sqrt2}U^{13,t}_{1,1}&\frac1{\sqrt2}(U^{4t}_{1,1}+U^{7t}_{1,1}+U^{14,t}_{1,1})&\frac1{\sqrt2}(U^{4t}_{1,1}-U^{7t}_{1,1}-2U^{14,t}_{1,1})&-\frac1{\sqrt2}(U^{2t}_{1,1}+U^{8t}_{1,1}+2U^{16,t}_{1,1})\\
			\hline
			i=7&-\frac i{\sqrt2}U^{13,t}_{1,1}&\frac i{\sqrt2}(U^{4t}_{1,1}+U^{7t}_{1,1}-U^{14,t}_{1,1})&\frac i{\sqrt2}(U^{4t}_{1,1}-U^{7t}_{1,1}+2U^{14,t}_{1,1})&-\frac i{\sqrt2}(U^{2t}_{1,1}+U^{8t}_{1,1}-2U^{16,t}_{1,1})\\
			\hline
			i=8&2iU^{9t}_{1,1}&0&0&-iU^{12,t}_{1,1}\\
			\hline
			i=9&-U^{10,t}_{1,1}&2(U^{9t}_{1,1}+U^{16,t}_{1,1})&U^{9t}_{1,1}+U^{16,t}_{1,1}&-U^{7t}_{1,1}\\
			\hline
			i=10&iU^{10,t}_{1,1}&2i(U^{9t}_{1,1}-U^{16,t}_{1,1})&i(U^{9t}_{1,1}-U^{16,t}_{1,1})&-iU^{7t}_{1,1}\\
			\hline
		\end{array}
	\end{align*}
	\begin{align*}
		\begin{array}{|c|c|c|c|c|}
			\hline
			X_i(U^{st}_{1,1})&s=13&s=14&s=15&s=16\\
			\hline
			i=1&-2iU^{13,t}_{1,1}&-2iU^{14,t}_{1,1}&iU^{15,t}_{1,1}&-iU^{16,t}_{1,1}\\
			\hline
			i=2&-U^{4t}_{1,1}&-U^{5t}_{1,1}&U^{16,t}_{1,1}&-U^{15,t}_{1,1}\\
			\hline
			i=3&-iU^{4t}_{1,1}&-iU^{5t}_{1,1}&-iU^{16,t}_{1,1}&-iU^{15,t}_{1,1}\\
			\hline
			i=4&\frac1{\sqrt2}U^{11,t}_{1,1}&-\frac1{\sqrt2}U^{16,t}_{1,1}&\frac1{\sqrt2}U^{12,t}_{1,1}&\frac1{\sqrt2}U^{14,t}_{1,1}\\
			\hline
			i=5&\frac i{\sqrt2}U^{11,t}_{1,1}&-\frac i{\sqrt2}U^{16,t}_{1,1}&-\frac i{\sqrt2}U^{12,t}_{1,1}&-\frac i{\sqrt2}U^{14,t}_{1,1}\\
			\hline
			i=6&-\frac1{\sqrt2}U^{9t}_{1,1}&\frac1{\sqrt2}(U^{11,t}_{1,1}-U^{10,t}_{1,1})&-\frac1{\sqrt2}U^{3t}_{1,1}&\frac1{\sqrt2}(U^{12,t}_{1,1}-U^{5t}_{1,1})\\
			\hline
			i=7&-\frac i{\sqrt2}U^{9t}_{1,1}&\frac i{\sqrt2}(U^{11,t}_{1,1}-U^{10,t}_{1,1})&-\frac i{\sqrt2}U^{3t}_{1,1}&\frac i{\sqrt2}(U^{12,t}_{1,1}-U^{5t}_{1,1})\\
			\hline
			i=8&iU^{13,t}_{1,1}&-iU^{14,t}_{1,1}&-2iU^{15,t}_{1,1}&-2iU^{16,t}_{1,1}\\
			\hline
			i=9&-U^{14,t}_{1,1}&U^{13,t}_{1,1}&-U^{8t}_{1,1}&-U^{10,t}_{1,1}\\
			\hline
			i=10&iU^{14,t}_{1,1}&iU^{13,t}_{1,1}&-iU^{8t}_{1,1}&-iU^{10,t}_{1,1}\\
			\hline
		\end{array}
	\end{align*}
	for any $1\le t\le16$, and
	\begin{align*}
		&X_8^2(U^{st}_{1,1})=\left\{
		\begin{array}{ll}
			-U^{st}_{1,1},&s=1,3,4,5,7,12,13,14,\\
			-4U^{st}_{1,1},&s=6,9,15,16,\\
			0,&s=2,8,10,11,
		\end{array}
		\right.\\
		&X_9^2(U^{st}_{1,1})=\left\{
		\begin{array}{ll}
			-U^{st}_{1,1},&s=1,3,4,5,7,12,13,14,\\
			-4U^{st}_{1,1},&s=8,10,\\
			-2U^{8t}_{1,1},&s=2,\\
			2(U^{15,t}_{1,1}-U^{6t}_{1,1}),&s=6,\\
			-2(U^{9t}_{1,1}+U^{16,t}_{1,1}),&s=9,16,\\
			-2U^{10,t}_{1,1},&s=11,\\
			2(U^{6t}_{1,1}-U^{15,t}_{1,1}),&s=15,
		\end{array}
		\right.\\
		&X_{10}^2(U^{st}_{1,1})=\left\{
		\begin{array}{ll}
			-U^{st}_{1,1},&s=1,3,4,5,7,12,13,14,\\
			-4U^{st}_{1,1},&s=8,10,\\
			-2U^{8t}_{1,1},&s=2,\\
			-2(U^{6t}_{1,1}+U^{15,t}_{1,1}),&s=6,15,\\
			2(U^{16,t}_{1,1}-U^{9t}_{1,1}),&s=9,\\
			-2U^{10,t}_{1,1},&s=11,\\
			2(U^{9t}_{1,1}-U^{16,t}_{1,1}),&s=16
		\end{array}
		\right.
	\end{align*}
	for any $1\le t\le16$. Hence $-C_K$ acts as $0$, $-3\operatorname{Id}$, and $-8\operatorname{Id}$ on the subspaces $V_1$, $V_2$, and $V_3$, respectively, where
	\begin{align*}
		&V_1=\operatorname{span}_{\mathbb{C}}\{U^{8t}_{1,1}-2U^{2t}_{1,1}, U^{10,t}_{1,1}-2U^{11,t}_{1,1}\mid 1\le t\le16\},\\
		&V_2=\operatorname{span}_{\mathbb{C}}\{U^{st}_{1,1}\mid s=1,3,4,5,7,12,13,14,\ 1\le t\le 16\},\\
		&V_3=\operatorname{span}_{\mathbb{C}}\{U^{st}_{1,1}\mid s=6,8,9,10,15,16,\ 1\le t\le 16\}.
	\end{align*}
	As in the preceding cases, matching the eigenvalues of the 2 $K$-Casimir operators gives
	\begin{align*}
		(\mathcal{U}_{1,1}^{\mathbb{C}}\otimes(\mathfrak{m}^*\otimes\mathfrak{k})_{\mathbb{C}})_K\subset V_2\otimes \operatorname{span}_{\mathbb{C}}\{F_1,F_2,F_3,F_4\}\oplus V_3\otimes \operatorname{span}_{\mathbb{C}}\{\xi_i\otimes X_\alpha\mid1\le i\le3,\ 8\le\alpha\le10\}.
	\end{align*}
	The first and second summands correspond to $(\lambda,q)=(-7,1)$ and $(12,2)$, respectively. The branching rule for $V_{1,1}$ therefore gives
	\begin{align*}
		V_2&\subset\mathcal{U}^{\mathbb{C}}_{1,1}(2,1)\oplus\mathcal{U}^{\mathbb{C}}_{1,1}(0,1),\\
		V_3&\subset\mathcal{U}^\mathbb{C}_{1,1}(1,2).
	\end{align*}
	For each fixed $1\le t\le16$, the infinitesimal $K$-equivariance equations on the first summand are
	\begin{align*}
		\hat A_8x=\hat A_9x=\hat A_{10}x=0,\qquad x\in\mathbb{C}^{32},
	\end{align*}
	with respect to the lexicographically ordered basis
	\begin{align*}
		(U^{st}_{1,1}\otimes F_r)_{s\in(1,3,4,5,7,12,13,14),\ 1\le r\le4},
	\end{align*}
	with $r$ varying fastest, where
	\begin{align*}
		&\hat A_8=\operatorname{diag}(iI-A_8',-iI-A_8',iI-A_8',-iI-A_8',iI-A_8',-iI-A_8',iI-A_8',-iI-A_8'),\\
		&\hat A_9=\operatorname{diag}(\begin{pmatrix}
			-A_9'&-I\\I&-A_9'
		\end{pmatrix},
		\begin{pmatrix}
			-A_9'&I\\-I&-A_9'
		\end{pmatrix},
		\begin{pmatrix}
			-A_9'&-I\\I&-A_9'
		\end{pmatrix},
		\begin{pmatrix}
			-A_9'&I\\-I&-A_9'
		\end{pmatrix}),\\
		&\hat A_{10}=\operatorname{diag}(\begin{pmatrix}
			-A_{10}'&-iI\\-iI&-A_{10}'
		\end{pmatrix},
		\begin{pmatrix}
			-A_{10}'&iI\\iI&-A_{10}'
		\end{pmatrix},
		\begin{pmatrix}
			-A_{10}'&-iI\\-iI&-A_{10}'
		\end{pmatrix},
		\begin{pmatrix}
			-A_{10}'&iI\\iI&-A_{10}'.
		\end{pmatrix})
	\end{align*}
	For each fixed $1\le i\le3$ and $1\le t\le16$, the equations on the second summand are
	\begin{align*}
		\hat B_8y=\hat B_9y=\hat B_{10}y=0,\qquad y\in\mathbb{C}^{18},
	\end{align*}
	with respect to the lexicographically ordered basis
	\begin{align*}
		(U^{st}_{1,1}\otimes\xi_i\otimes X_\alpha)_{s\in(6,8,9,10,15,16),\ 8\le\alpha\le10},
	\end{align*}
	with $\alpha$ varying fastest, where
	\begin{align*}
		&\hat B_8=\begin{pmatrix}
			2iI+B_8'&0&0&0&0&0\\
			0&B_8'&0&0&0&0\\
			0&0&2iI+B_8'&0&0&0\\
			0&0&0&B_8'&0&0\\
			0&0&0&0&-2iI+B_8'&0\\
			0&0&0&0&0&-2iI+B_8'
		\end{pmatrix},\\
		&\hat B_9=\begin{pmatrix}
			B_9'&-2I&0&0&0&0\\
			I&B_9'&0&0&-I&0\\
			0&0&B_9'&2I&0&0\\
			0&0&-I&B_9'&0&-I\\
			0&2I&0&0&B_9'&0\\
			0&0&0&2I&0&B_9'
		\end{pmatrix},\\
		&\hat B_{10}=\begin{pmatrix}
			B_{10}'&-2iI&0&0&0&0\\
			-iI&B_{10}'&0&0&-iI&0\\
			0&0&B_{10}'&2iI&0&0\\
			0&0&iI&B_{10}'&0&-iI\\
			0&-2iI&0&0&B_{10}'&0\\
			0&0&0&-2iI&0&B_{10}'
		\end{pmatrix}
	\end{align*}
	Direct calculation gives
	\begin{align*}
		(\mathcal{U}_{1,1}^{\mathbb{C}}\otimes(\mathfrak{m}^*\otimes\mathfrak{k})_{\mathbb{C}})_K=\operatorname{span}_{\mathbb{C}}\{\hat\Phi_{jt},\hat\Phi_{9it},\hat\Phi_{10,it}\mid1\le j\le8,\ 1\le i\le3,\ 1\le t\le16\},
	\end{align*}
	where
	\begin{align*}
		&\hat\Phi_{1t}=U^{1t}_{1,1}\otimes F_1+iU^{1t}_{1,1}\otimes F_4-iU^{3t}_{1,1}\otimes F_2-U^{3t}_{1,1}\otimes F_3,\\
		&\hat\Phi_{2t}=U^{1t}_{1,1}\otimes F_2+iU^{1t}_{1,1}\otimes F_3+iU^{3t}_{1,1}\otimes F_1+U^{3t}_{1,1}\otimes F_4,\\
		&\hat\Phi_{3t}=U^{4t}_{1,1}\otimes F_1+iU^{4t}_{1,1}\otimes F_4+iU^{5t}_{1,1}\otimes F_2+U^{5t}_{1,1}\otimes F_3,\\
		&\hat\Phi_{4t}=U^{4t}_{1,1}\otimes F_2+iU^{4t}_{1,1}\otimes F_3-iU^{5t}_{1,1}\otimes F_1-U^{5t}_{1,1}\otimes F_4,\\
		&\hat\Phi_{5t}=U^{7t}_{1,1}\otimes F_1+iU^{7t}_{1,1}\otimes F_4-iU^{12,t}_{1,1}\otimes F_2-U^{12,t}_{1,1}\otimes F_3,\\
		&\hat\Phi_{6t}=U^{7t}_{1,1}\otimes F_2+iU^{7t}_{1,1}\otimes F_3+iU^{12,t}_{1,1}\otimes F_1+U^{12,t}_{1,1}\otimes F_4,\\
		&\hat\Phi_{7t}=U^{13,t}_{1,1}\otimes F_1+iU^{13,t}_{1,1}\otimes F_4+iU^{14,t}_{1,1}\otimes F_2+U^{14,t}_{1,1}\otimes F_3,\\
		&\hat\Phi_{8t}=U^{13,t}_{1,1}\otimes F_2+iU^{13,t}_{1,1}\otimes F_3-iU^{14,t}_{1,1}\otimes F_1-U^{14,t}_{1,1}\otimes F_4,\\
		&\hat\Phi_{9it}=U^{6t}_{1,1}\otimes\xi_i\otimes X_9+iU^{6t}_{1,1}\otimes\xi_i\otimes X_{10}+iU^{8t}_{1,1}\otimes\xi_i\otimes X_8+U^{15,t}_{1,1}\otimes\xi_i\otimes X_9-iU^{15,t}_{1,1}\otimes\xi_i\otimes X_{10},\\
		&\hat\Phi_{10,it}=U^{9t}_{1,1}\otimes\xi_i\otimes X_9+iU^{9t}_{1,1}\otimes\xi_i\otimes X_{10}-iU^{10,t}_{1,1}\otimes\xi_i\otimes X_8-U^{16,t}_{1,1}\otimes\xi_i\otimes X_9+iU^{16,t}_{1,1}\otimes\xi_i\otimes X_{10}.
	\end{align*}
	Lemma \ref{eigenvalue of J2} and Lemma \ref{eigenvalue of J0} give
	\begin{align*}
		\tilde J_2^{\mathbb{C}}\ge\lambda_{1,1}(2,1)\operatorname{Id}=16\operatorname{Id},\quad\tilde J_0^{\mathbb{C}}=-7\operatorname{Id}
	\end{align*}
	on $\operatorname{span}_{\mathbb{C}}\{\hat\Phi_{jt}\mid1\le j\le8,\ 1\le t\le16\}$, and
	\begin{align*}
		&\tilde J_2^{\mathbb{C}}=\lambda_{1,1}(1,2)\operatorname{Id}=11\operatorname{Id},\quad\tilde J_0^{\mathbb{C}}=12\operatorname{Id}
	\end{align*}
	on $\operatorname{span}_{\mathbb{C}}\{\hat\Phi_{9it},\hat\Phi_{10,it}\mid1\le i\le3,\ 1\le t\le16\}$. Consequently,
	\begin{align*}
		\tilde J^{\mathbb{C}}\ge\tilde J_1^{\mathbb{C}}+23P_1+9P_2.
	\end{align*}
	For each fixed $1\le t\le16$, with respect to the ordered basis $(\hat\Phi_{1t},\dots,\hat\Phi_{8t},\hat\Phi_{91t},\hat\Phi_{92t},\hat\Phi_{93t},\hat\Phi_{10,1t},\hat\Phi_{10,2t},\hat\Phi_{10,3t})$, the operator $\tilde J_1^{\mathbb{C}}$ is represented by
	\begin{align*}
		\tilde J_1^{\mathbb{C}}=\begin{pmatrix}
			0_{8\times8}&D_1\\
			D_2&D_3
		\end{pmatrix},
	\end{align*}
	where
	\begin{align*}
		&D_1=\begin{pmatrix}
			0&-\sqrt2 i&\sqrt2&0&0&0\\
			\sqrt2 i&0&0&0&0&0\\
			0&0&0&0&-\sqrt2 i&\sqrt2\\
			0&0&0&\sqrt2 i&0&0\\
			\sqrt2&0&0&0&-\sqrt2 i&\sqrt2\\
			0&\sqrt2&-\sqrt2 i&\sqrt2 i&0&0\\
			0&0&0&\sqrt2&0&0\\
			0&0&0&0&\sqrt2&-\sqrt2 i
		\end{pmatrix},\\
		&D_2=\begin{pmatrix}
			0&-4\sqrt2i&-4\sqrt2&0&12\sqrt2&0&0&0\\
			4\sqrt2i&0&0&-4\sqrt2&0&12\sqrt2&0&0\\
			4\sqrt2&0&0&-4\sqrt2i&0&12\sqrt2i&0&0\\
			0&0&0&-4\sqrt2i&0&-8\sqrt2i&4\sqrt2&0\\
			0&0&4\sqrt2i&0&8\sqrt2i&0&0&4\sqrt2\\
			0&0&4\sqrt2&0&8\sqrt2&0&0&4\sqrt2i
		\end{pmatrix},\\
		&D_3=\begin{pmatrix}
			0&0&0&0&2i&-2\\
			0&0&2i&-2i&0&0\\
			0&-2i&0&2&0&0\\
			0&2i&2&0&0&0\\
			-2i&0&0&0&0&-2i\\
			-2&0&0&0&2i&0
		\end{pmatrix}.
	\end{align*}
	Thus, in this ordered basis, the operator $\tilde J_1^{\mathbb{C}}+23P_1+9P_2$ is represented by
	\begin{align*}
		\tilde J_1^{\mathbb{C}}+\operatorname{diag}(9I_8,23I_6).
	\end{align*}
	Its characteristic polynomial is
	\begin{align*}
		\det(\lambda I-(\tilde J_1^{\mathbb{C}}+\operatorname{diag}(9I,23I)))=(\lambda-9)^2(\lambda-17-4\sqrt5)^4(\lambda-17+4\sqrt5)^4(\lambda-14-\sqrt{89})^2(\lambda-14+\sqrt{89})^2.
	\end{align*}
	Therefore, we obtain
	\begin{align*}
		\tilde J^{\mathbb{C}}\ge(14-\sqrt{89})\operatorname{Id}>0
	\end{align*}
	on $(\mathcal{U}_{1,1}^{\mathbb{C}}\otimes(\mathfrak{m}^*\otimes\mathfrak{k})_{\mathbb{C}})_K$, and thus
	\begin{align*}
		\tilde J^{\mathbb{C}}>0
	\end{align*}
	on $(\mathcal{U}_{1,1}^{\mathbb{C}}\otimes(\mathfrak{m}^*\otimes\mathfrak{k})_{\mathbb{C}})_K\cap \ker\tilde\delta$.
	\section{Use of AI}
	In Appendix \ref{app V11}, GPT-5.6-sol was used to compute the $K$-equivariant spaces and the matrices representing the operator $\tilde J_1^{\mathbb{C}}$ on these spaces.
	
	All AI-assisted computations were independently verified by the authors, who take full responsibility for the results presented in this article.
	
	\section*{Conflict of interest}
	
	The authors declare no conflict of interest.
	
	\section*{Acknowledgment}
	The first author is supported by the National Key R\&D Program of China (Grant No. 2022YFA1005400). The second author is supported by the Stability Support Program for Basic Research (Grant No. WZJC202601001).
	
\end{document}